\documentclass{amsart}

\usepackage{import}
\usepackage{graphicx} 
\usepackage{amssymb,amsthm,amstext,amsmath}
\usepackage{fullpage}
\usepackage{tikz-cd}
\usepackage{hyperref}
\usepackage{tikz,pgfplots}
\usetikzlibrary{decorations.markings, intersections, arrows.meta, positioning, calc, math, matrix}
\pgfplotsset{compat=1.18}

\tikzset{middlearrow/.style={
        decoration={markings,
            mark= at position 0.5 with {\arrow{#1}} ,
        },
        postaction={decorate}
    }
}

\newcommand{\N}{\mathbb{N}}

\newcommand{\R}{\mathbb{R}}

\DeclareFontFamily{U}{mathb}{}
\DeclareFontSubstitution{U}{mathb}{m}{n}
\DeclareFontShape{U}{mathb}{m}{n}{
  <-5.5> mathb5
  <5.5-6.5> mathb6
  <6.5-7.5> mathb7
  <7.5-8.5> mathb8
  <8.5-9.5> mathb9
  <9.5-11> mathb10
  <11-> mathb12
}{}
\DeclareSymbolFont{mathb}{U}{mathb}{m}{n}
\DeclareMathSymbol{\boxtimes}{\mathbin}{mathb}{"62}

\DeclareMathOperator{\gr}{gr}
\DeclareMathOperator{\ex}{ex}
\DeclareMathOperator{\ind}{ind}
\DeclareMathOperator{\pr}{pr}
\DeclareMathOperator{\lcs}{lcs}
\DeclareMathOperator{\im}{im}
\DeclareMathOperator{\Ham}{Ham}

\DeclareMathOperator{\supp}{supp}

\DeclareMathOperator{\id}{id}

\DeclareMathOperator{\Cont}{Cont}

\DeclareMathOperator{\symp}{symp}

\DeclareMathOperator{\cont}{cont}

\newtheorem{thm}{Theorem}[section]
\newtheorem{prop}[thm]{Proposition}
\newtheorem{cor}[thm]{Corollary}
\newtheorem{lemma}[thm]{Lemma}

\newtheorem{example}[thm]{Example}
\newtheorem{remark}[thm]{Remark}

\theoremstyle{remark}

\newtheorem*{noTitle*}{}

\title{Non-squeezing and other global rigidity results
in locally conformal symplectic geometry}

\author{M\'elanie Bertelson}
\address{D\'epartement de Math\'ematiques, C.P. 218 Universit\'e Libre de Bruxelles, 1050 Bruxelles, Belgium}
\email{melanie.bertelson@ulb.be}

\author{Pranav Chakravarthy}
\address{D\'epartement de Math\'ematiques, C.P. 218 Universit\'e Libre de Bruxelles, 1050 Bruxelles, Belgium}
\email{pranav.vijay.chakravarthy@ulb.be}

\author{Sheila Sandon}
\address{Universit\'e de Strasbourg, CNRS, IRMA UMR 7501, F-67000 Strasbourg, France}
\email{sandon@math.unistra.fr}

\begin{document}
\maketitle
\begin{center}
\textit{We dedicate this work to Michèle Audin}
\end{center}
\begin{abstract}
\noindent Using generating functions quadratic at infinity
for Lagrangian submanifolds of twisted cotangent bundles,
we define spectral selectors
for compactly supported lcs Hamiltonian diffeomorphisms
of the locally conformal symplectizations
$S^1 \times \mathbb{R}^{2n+1}$ and $S^1 \times \mathbb{R}^{2n} \times S^1$
of $\mathbb{R}^{2n+1}$ and $\mathbb{R}^{2n} \times S^1$,
and obtain several applications:
the construction of a bi-invariant partial order
on the group of compactly supported lcs Hamiltonian diffeomorphisms
of $S^1 \times \mathbb{R}^{2n+1}$ and $S^1 \times \mathbb{R}^{2n} \times S^1$,
of an integer-valued bi-invariant metric
on the group of compactly supported lcs Hamiltonian diffeomorphisms
of $S^1 \times \mathbb{R}^{2n} \times S^1$,
and of an integer-valued lcs capacity
for domains of $S^1 \times \mathbb{R}^{2n} \times S^1$.
The lcs capacity is used to prove a lcs non-squeezing theorem
in $S^1 \times \mathbb{R}^{2n} \times S^1$
analogous to the contact non-squeezing theorem
in $\mathbb{R}^{2n} \times S^1$
discovered in 2006 by Eliashberg, Kim and Polterovich.
Along the way we introduce 
for Liouville lcs manifolds the notions
of essential Lee chords between exact Lagrangian submanifolds
and of essential translated points
of exact lcs diffeomorphisms.
We prove that essential translated points always exist
for time-$1$ maps of sufficiently $\mathcal{C}^0$-small lcs Hamiltonian isotopies
of compact Liouville lcs manifolds
and for all compactly supported lcs Hamiltonian diffeomorphisms
of $S^1 \times \mathbb{R}^{2n+1}$ and $S^1 \times \mathbb{R}^{2n} \times S^1$.
We also obtain an existence result for essential Lee chords
between the zero section of a twisted cotangent bundle with compact base
and its image by any lcs Hamiltonian isotopy,
which can be thought of as a lcs analogue
of the Lagrangian and Legendrian Arnold conjectures
on usual cotangent and $1$-jet bundles.
Finally,
we introduce the notion of orderability for lcs manifolds,
and prove that $S^1 \times \mathbb{R}^{2n+1}$, $S^1 \times \mathbb{R}^{2n} \times S^1$
and twisted cotangent bundles are orderable.

\end{abstract}

\section{Introduction}


A locally conformal symplectic (lcs) form
on a manifold $M$ of dimension $2n \geq 4$
is a non-degenerate $2$-form $\omega$ such that,
for some cover $\{\mathcal{U}_i\}$ of $M$,
$\left. \omega \right\lvert_{\mathcal{U}_i} = e^{\mu_i} \, \omega_i$
for a function $\mu_i: \mathcal{U}_i \rightarrow \mathbb{R}$
and a symplectic form $\omega_i$ on $\mathcal{U}_i$.
Equivalently,
a non-degenerate $2$-form $\omega$ is lcs
if there exists a closed $1$-form $\eta$
(determined by $\omega$, and called the Lee form)
such that $d_{\eta}\omega = 0$,
where $d_{\eta}$ is the Lichnerowicz--de Rham differential
(defined by $d_{\eta}\sigma = d\sigma - \eta \wedge \sigma$).
A lcs structure is an equivalence class of lcs forms
under the equivalence relation
$\omega \sim e^f \omega$ for $f: M \rightarrow \mathbb{R}$,
equivalently an equivalence class $[(\eta, \omega)]$
of pairs formed by a closed $1$-form $\eta$
and a non-degenerate $2$-form $\omega$ with $d_{\eta}\omega=0$,
for the equivalence relation $(\eta, \omega) \sim (\eta + df, e^f \omega)$. A lcs manifold is a manifold endowed with a lcs structure.
If $\big( M , [(\eta, \omega)] \big)$ is a lcs manifold
such that for some (hence any) representative $(\eta, \omega)$
the Lee form $\eta$ is exact,
then for any primitive $\mu$ of $\eta$
the form $e^{-\mu} \omega$ is a symplectic form
whose conformal class is determined
by the lcs structure $[(\eta, \omega)]$.
Conformal symplectic manifolds are thus examples of lcs manifolds.
Other important examples of lcs manifolds
are the twisted cotangent bundle
$T_{\beta}^{\ast}B := \big( T^{\ast}B \,,\,  [(\pi^{\ast}\beta, d_{\pi^{\ast}\beta}\lambda)] \big)$
of a manifold $B$ with respect to a closed $1$-form $\beta$ on $B$,
where $\pi: T^{\ast}B \rightarrow B$ is the projection
and $\lambda$ is the tautological $1$-form,
and the locally conformal symplectization
$\big( S^1 \times Y \,,\, [(-d\theta, d_{-d\theta}\alpha)] \big)$
of a co-oriented contact manifold $(Y, \xi)$
with respect to a contact form $\alpha$,
where $\theta$ denotes the coordinate in $S^1$.
In the present article we study in particular
the locally conformal symplectization
of the standard contact Euclidean space $(\mathbb{R}^{2n+1}, \xi_0)$
with respect to the contact form
$\alpha_0 = \sum_{j=1}^n \frac{x_idy_j - y_jdx_j}{2} - dz$,
and of the quotient
$\mathbb{R}^{2n} \times S^1 = \mathbb{R}^{2n} \times \mathbb{R}/\mathbb{Z}$,
endowed with the induced contact structure and contact form.
We denote these two lcs manifolds simply by
$S^1 \times \mathbb{R}^{2n+1}$
and $S^1 \times \mathbb{R}^{2n} \times S^1$.

Lcs manifolds have been studied since the 40s
by several authors,
among others Lee \cite{Lee}, Libermann \cite{Libermann},
Vaisman \cite{Vaisman 76, Vaisman},
Lichnerowicz \cite{Lich},
Banyaga \cite{Banyaga lcs},
Otiman and Stanciu \cite{OS},
Apostolov and Dloussky \cite{AD, AD18},
Bazzoni and Marrero \cite{BaM}, Chantraine and Murphy \cite{CM},
Eliashberg and Murphy \cite{EM},
Bertelson and Meigniez \cite{BM},
Gironella and Toussaint \cite{GT}
Allais and Arnaud \cite{AA},
Chantraine and Sackel \cite{CS},
Currier \cite{Currier, Currier 2}
and Van Overschelde \cite{VOa, VOb}.
Most of the basic notions of symplectic topology
can be generalized to the lcs case.
For instance,
Lagrangian submanifolds and Hamiltonian isotopies
on a lcs manifold $\big( M , [(\eta, \omega)] \big)$
are defined as follows.
A Lagrangian submanifold
is a half-dimensional submanifold
on which for some (hence any) representative $(\eta, \omega)$
the lcs form $\omega$ vanishes.
A vector field $X_t$ is said to be Hamiltonian
if for some (hence any) representative $(\eta, \omega)$
the form $\iota_{X_t}\omega$ is $\eta$-exact. 
The flow of $X_t$ is then said
to be a lcs Hamiltonian isotopy,
and its time-$1$ map a lcs Hamiltonian diffeomorphism.
In particular,
lcs geometry can be regarded as a generalization of symplectic geometry
that allows to study Hamiltonian dynamics in a more general context
(and on many more spaces,
since the class of manifolds admitting a lcs structure
is much bigger than the class of manifolds
admitting a global symplectic form
\cite{EM, BM}).
Another important reason for being interested in locally conformal symplectic structures comes from the hope that, in higher dimension, taut codimension one foliations can be deformed into tight contact structures, as in dimension 3 (cf.\ \cite{ET}). If the foliation is defined by a $1$-form $\alpha$, with holonomy $1$-form $\eta$ ($d\alpha = \eta \wedge \alpha$) the easiest way to deform the foliation is to deform $\alpha$, that is consider a smooth family of $1$-forms $\alpha_t$ with $\alpha_0 = \alpha$. If the first order of this deformation $\dot{\alpha}_t$ is a $\eta$-Liouville form, then the deformation is a contact structure for all sufficiently small positive $t$'s. This was the motivation of the authors of \cite{BM} for their construction of (foliated) exact lcs form with a fixed (leafwise) Lee form.

The usual local flexibility properties of symplectic and contact topology
(in particular, the Darboux and Weinstein theorems)
have natural lcs analogues.
One of the central results of symplectic geometry,
$\mathcal{C}^0$-rigidity of symplectomorphisms,
also readily implies its lcs counterpart 
(see Section~\ref{section: C0 rigidity}).
On the other hand,
the study of other global rigidity phenomena
is less explored in the lcs case.
One of the reasons is 
the absence of a Stokes theorem for the Lichnerowicz--de Rham differential,
and thus the lack of $C^0$-bounds on the energy of pseudo-holomorphic curves,
which proves to be a major hurdle to studying deformations
of holomorphic curves on these manifolds
(see Section 2.5 in \cite{CM}).

However,
as noticed by Chantraine and Murphy \cite{CM},
the theory of generating functions
adapts elegantly to the lcs case.
In particular,
Chantraine and Murphy use generating functions
to prove the following global rigidity result
(see also \cite{Currier 2} for refinements).
Consider a twisted cotangent bundle $T_{\beta}^{\ast}B$
with compact base $B$,
and a Lagrangian submanifold $L$ in $T_{\beta}^{\ast}B$
that is lcs Hamiltonian isotopic to the zero section.
Suppose that the intersections of $L$ with the zero section are transverse.
Then their number is at least equal
to the rank of the Novikov homology of $B$
with respect to $\beta$.
Notice that this result does not imply
that $L$ must always intersect the zero section,
indeed the Novikov homology can vanish
(for instance if $\beta$ has no zeros).
In fact, it is easy to see that if $\beta$ has no zeros
then there is a lcs Hamiltonian isotopy
(the standard Lee flow)
that displaces the zero section.
The naive analogue of the Lagrangian Arnold conjecture
on usual cotangent bundles,
proved by Laudenbach and Sikorav \cite{LS},
thus does not hold in the lcs case. 
However,
the situation is similar
to what already happens in the contact case
for $1$-jet bundles,
where the standard Reeb flow also displaces the zero section.
In the contact case
the following analogue
of the Lagrangian Arnold conjecture,
where intersections are replaced by Reeb chords,
has been proved by Chekanov \cite{Chekanov}
and Chaperon \cite{Chaperon}.
We endow the $1$-jet bundle $J^1B = T^{\ast}B \times \mathbb{R}$
with its standard contact structure $\xi = \ker (\lambda - dz)$,
where $z$ is the coordinate on $\mathbb{R}$.
If $B$ is compact 
then for any Legendrian submanifold $\Lambda$ of $J^1B$
that is the image of the zero section by a contact isotopy
the number of Reeb chords with respect to $\lambda - dz$
between $\Lambda$ and the zero section
is greater than or equal to the cup-length of $B$,
and greater than or equal to the sum of the Betti numbers
if all the Reeb chords are transverse.

Our first observation is that
a similar result also holds in the lcs case
for essential Lee chords
on twisted cotangent bundles.
Before precisely stating the result,
we need to recall or introduce a few more notions,
referring to Sections \ref{section: lcs manifolds}
and \ref{section: Lee chords and translated points}
for more details.
The Lee flow on a lcs manifold $\big( M , [(\eta, \omega)] \big)$
with respect to a representative $(\eta, \omega)$
is the flow $\{\varphi_t^{\omega}\}$ of the vector field $R_{\omega}$
defined by the relation $\omega (R_{\omega}, \,\cdot\,) = \eta$
(equivalently, the lcs Hamiltonian isotopy
with Hamiltonian function $1$
with respect to $(\eta, \omega)$).
For instance,
the Lee flow on a locally conformal symplectization
$\big( S^1 \times Y \,,\, [(-d\theta, d_{-d\theta}\alpha)] \big)$
with respect to $(-d\theta, d_{-d\theta}\alpha)$
is $ \{ \id \times \varphi_t^{\alpha} \}$,
where $\{\varphi_t^{\alpha}\}$ is the Reeb flow on $Y$
with respect to $\alpha$,
and the Lee flow on a twisted cotangent bundle $T^{\ast}_{\beta}B$
with respect to $(\pi^{\ast}\beta, d_{\pi^{\ast}\beta}\lambda)$
is given by $\varphi_t^{\omega} (\sigma_q) = \sigma_q + t \beta (q)$.
Let $L_1$ and $L_2$ be two Lagrangian submanifolds
of a lcs manifold $\big( M , [(\eta, \omega)] \big)$.
A Lee chord with respect to $(\eta, \omega)$ from $L_1$ to $L_2$
is a pair $(T, \gamma)$ with $T$ a real number
(called the time-shift of the Lee chord)
and $\gamma: [0, 1] \rightarrow M$
a curve such that $\gamma(0) \in L_1$,
$\gamma(1) \in L_2$
and $\dot{\gamma} (t) = T \, R_{\omega} \big(\gamma(t)\big)$
for all $t$. 
We say that a Lee chord $(T, \gamma)$
with respect to $(\eta, \omega)$
from $L_1$ to $L_2$
is transverse if
$(\varphi_T^{\omega})_{\ast} \big( T_{\gamma(0)}L_1 \big) + T_{\gamma(1)}L_2
= T_{\gamma(1)} M$.
A lcs manifold $\big( M , [(\eta, \omega)] \big)$
is said to be exact if for some (hence any) representative $(\eta, \omega)$
the lcs form $\omega$ is $\eta$-exact.
Any $1$-form $\lambda$ such that $\omega = d_{\eta}\lambda$
is then called a Liouville form for $(\eta, \omega)$.
If $\lambda$ is a Liouville form for $(\eta, \omega)$
then, for every function $f$,
$e^f \lambda$ is a Liouville form for $(\eta + df, e^f\omega)$.
A Liouville lcs manifold is a manifold $M$
endowed with an equivalence class $[(\eta, \omega,\lambda)]$
of triples formed by a closed $1$-form $\eta$,
a non-degenerate $2$-form $\omega$
and a $1$-form $\lambda$ such that $\omega = d_{\eta}\lambda$,
under the equivalence relation $(\eta, \omega, \lambda) \sim (\eta + df, e^f\omega, e^f\lambda)$.
For instance,
twisted cotangent bundles
$\big( T^{\ast}B \,,\,  [(\pi^{\ast}\beta, d_{\pi^{\ast}\beta}\lambda, \lambda)] \big)$
and locally conformal symplectizations
$\big( S^1 \times Y \,,\, [(-d\theta, d_{-d\theta}\alpha, \alpha)] \big)$
are Liouville lcs manifolds. 
A Lagrangian submanifold $L$
of a Liouville lcs manifold $\big( M , [(\eta, \omega, \lambda)] \big)$
is said to be exact if the $1$-form $i^{\ast}\lambda$,
where $i: L \rightarrow M$ is the inclusion,
is $i^{\ast}\eta$-exact.
In particular,
any Lagrangian submanifold of $T_{\beta}^{\ast}B$
that is the image of the zero section
by a lcs Hamiltonian diffeomorphism is exact
(Corollary \ref{corollary: image of zero section is exact}).
We say that an exact Lagrangian submanifold $L$
of a Liouville lcs manifold $\big( M, [(\eta, \omega, \lambda)] \big)$
has action $S_L: L \rightarrow \mathbb{R}$
with respect to $(\eta, \omega, \lambda)$
if $i^{\ast}\lambda = d_{i^{\ast}\eta} S_L$.
If $L$ is connected
and the restriction of $\eta$ to $L$ is not exact
then such action $S_L$ is uniquely defined
(see Remark \ref{remark: unique solution}).
In this case we define the action
with respect to $(\eta, \omega, \lambda)$
of a Lee chord $(T, \gamma)$
between two exact Lagrangian submanifolds $L_1$ and $L_2$
to be the real number
\[
\mathcal{A}_{L_1, L_2} (\gamma)
= S_{L_1} \big(\gamma(0)\big) - S_{L_2} \big(\gamma(1)\big) + \int_{\gamma} \lambda \,.
\]
This notion generalizes the usual symplectic action
of a Lagrangian intersection
(see Example \ref{example: action Lee chors symplectic})
and the usual contact action of a Reeb chord
between two Legendrian submanifolds
of a contact manifold
(the integral of the contact form along the Reeb chord,
see Example \ref{example: Lee chords lifts of Legendrians}).
While in the contact case the action of a Reeb chord
is always equal to its time-shift,
in the lcs case these two notions differ in general.
The Lee chords of time-shift equal to their action
play an important role in our work.
In analogy with the terminology
introduced in \cite{Currier} for Liouville chords,
we call them \emph{essential Lee chords}.
Our first result is then the following theorem.

\begin{thm}[Essential Lee chords]\label{theorem: lcs Arnold conjecture intro}
Consider a connected compact manifold $B$,
a closed non-exact $1$-form $\beta$ on $B$
and a Lagrangian submanifold $L$ 
of the twisted cotangent bundle $T^{\ast}_{\beta}B$
that is lcs Hamiltonian isotopic to the zero section.
Then the number of essential Lee chords
with respect to $(\pi^{\ast}\beta, d_{\pi^{\ast}\beta}\lambda, \lambda)$
from $L$ to the zero section
is greater than or equal to the cup-length of $B$,
and greater than or equal to the sum of the Betti numbers of $B$
if all the essential Lee chords are transverse.
\end{thm}

Similarly to the theorem of Chekanov \cite{Chekanov} and Chaperon \cite{Chaperon}
in the contact case,
Theorem \ref{theorem: lcs Arnold conjecture intro}
is a direct consequence
of existence of generating functions quadratic at infinity.
As observed by Chantraine and Murphy \cite{CM},
the theory of generating functions in symplectic and contact topology
generalizes quite easily to the lcs case
(in contrast to the theory of holomorphic curves,
whose generalization seems more delicate).
Recall first that in the symplectic and contact cases
the easiest examples of generating functions
are given by the following observation:
for every function $f: B \rightarrow \mathbb{R}$,
the differential $df: B \rightarrow T^{\ast}B$
and the $1$-jet $j^1f: B \rightarrow J^1B$
are respectively a Lagrangian and a Legendrian embedding.
We say that $f$ is a generating function for the images of these embeddings.
Similarly,
the twisted differential $d_{\beta}f: B \rightarrow T_{\beta}^{\ast}B$
is a Lagrangian embedding,
and we say that $f$ is a generating function
for the image of this embedding.
As in the symplectic and contact cases,
this elementary observation
can be generalized as follows.
Let $F: E \rightarrow \mathbb{R}$ be a function
defined on the total space
of a trivial vector bundle $E = B \times \mathbb{R}^N \rightarrow B$.
Suppose that the differential $dF: E \rightarrow T^{\ast}E$
is transverse to the subbundle $N_E^{\ast}$ of $T^{\ast}E$
formed by the covectors that vanish on vertical vectors.
Then the space
$\Sigma_F = dF^{-1} \big(N_E^{\ast} \cap \im (dF)\big)$
of fibre critical points of $F$
is a submanifold of $E$ of dimension equal to the dimension of $B$,
and the map $i_{F, \beta}: \Sigma_F \rightarrow T^{\ast}_{\beta}B$
that sends $e = (q, \zeta) \in \Sigma_F$
to the covector $i_{F, \beta}(e)$ at $q$
defined by $$i_{F, \beta} (e) (X) = d_{\beta}F (\hat{X}),$$
where $\hat{X}$ is any vector in $T_eE$ projecting to $X$,
is a Lagrangian immersion.
If $i_{F, \beta}$ is an embedding
we say that $F$ is a generating function
for the Lagrangian submanifold
$L_{F, \beta} = \im (i_{F, \beta})$ of $T_{\beta}^{\ast}B$.
Chantraine and Murphy \cite{CM} proved that
if $B$ is compact then every Lagrangian submanifold of $T_{\beta}^{\ast}B$
that is lcs Hamiltonian isotopic to the zero section
has a generating function quadratic at infinity,
unique up to stabilization and fibre preserving diffeomorphism.
This result is in fact a consequence of the corresponding existence and uniqueness theorems
for generating functions quadratic at infinity for Legendrian submanifolds of $1$-jet bundles,
proved by Chekanov \cite{Chekanov}, Chaperon \cite{Chaperon},
Viterbo \cite{Viterbo} and Th\'eret \cite{Theret, Theret_thesis}.
Indeed,
let $L$ be a Lagrangian submanifold of $T_{\beta}^{\ast}B$
that is lcs Hamiltonian isotopic to the zero section.
Then we can consider the image
of the lift $\widetilde{L}$ of $L$ to the contactization
$J^1_{\beta}B := \big( J^1B \,,\, \ker (\lambda - d_{\beta}z) \big)$
of $T^{\ast}_{\beta}B$
by the strict contactomorphism
$$
U_{\beta}: J^1_{\beta}B \to J^1B \,,\;
(\sigma_q, z) \mapsto \big( \sigma_q + z \beta(q) \,,\, z \big) \,,
$$
and observe that $F$ is a generating function for $L$
if and only if it is a generating function
(in the usual sense) for the Legendrian submanifold
$U_{\beta} (\widetilde{L})$ of $J^1B$
(see \cite{CM} or Section \ref{section: gf}
for more details).
To obtain their result on intersections with the zero section
of a Lagrangian submanifold $L$ of $T_{\beta}^{\ast}B$
lcs Hamiltonian isotopic to the zero section,
Chantraine and Murphy looked at the twisted critical points
of a generating function quadratic at infinity $F$ for $L$
(the zeros of the twisted differential $d_{\beta}F$).
These points do not always exist
(their number in the non-degenerate case
being bounded below by the rank of the Novikov homology \cite{CM, Currier 2}).
On the other hand,
in order to prove Theorem \ref{theorem: lcs Arnold conjecture intro}
we look at the usual critical points of $F$,
which always exist
(since $F$ is quadratic at infinity).
We observe that the (non-degenerate) critical points of $F$
are in bijection with the (transverse) essential Lee chords
from $L$ to the zero section
(Lemma \ref{lemma: bijection critical points F}).
Theorem \ref{theorem: lcs Arnold conjecture intro}
is then a consequence of the fact that
the number of critical points of a function quadratic at infinity over $B$
is greater than or equal to the cup-length of $B$,
and greater than or equal to the sum of the Betti numbers of $B$
if all the critical points are non-degenerate.

Our second result concerns essential translated points
of lcs Hamiltonian diffeomorphisms.
Recall that a point $p$ of a co-oriented contact manifold $(Y, \xi)$
is said to be a translated point of a contactomorphism $\phi$
with respect to a contact form $\alpha$
if $p$ and $\phi(p)$ are in the same orbit
of the Reeb flow of $\alpha$
and $\alpha$ is preserved by $\phi$ at $p$.
This notion has been introduced in \cite{Bhupal, San11}
in the case of the standard contact Euclidean space,
and in \cite{S - Morse estimate for translated points, S - Iterated} in general.
In \cite{S - Morse estimate for translated points, S - Iterated},
the third author has observed
that the time-$1$ map of a sufficiently $\mathcal{C}^0$-small
contact isotopy
on a compact contact manifold
always has translated points
(at least as many as the cup-length of the manifold,
and as many as the sum of the Betti numbers
in the non-degenerate case)
and,
in analogy with the  Arnold conjecture
on fixed points of Hamiltonian diffeomorphisms,
proposed to study the question
of existence of translated points
for more general contactomorphisms
contact isotopic of the identity.
Existence of translated points
(for compactly supported contactomorphisms
when the contact manifold is not compact)
has been proved beyond the $\mathcal{C}^0$-small case
on several contact manifolds
\cite{S - Morse estimate for translated points, S - Iterated,
AM, AFM, Shelukhin, MN, MU, GKPS, Tervil, Allais lens, Allais unit tangent bundles, AArl}.
On the other hand,
examples of contactomorphisms contact isotopic to the identity
without translated points
have also been found
on the standard contact sphere \cite{Cant}
and on a more general class of contact manifolds \cite{HS}.
The question of existence of translated points
is probably related to the notion of orderability,
introduced by Eliashberg and Polterovich \cite{EP}.
We say that a contact isotopy on a co-oriented contact manifold $(Y, \xi)$
is non-negative if it is generated by a non-negative Hamiltonian function,
and thus moves every point in a direction positively transverse
or tangent to the contact distribution.
The contact manifold $(Y, \xi)$
is then said to be orderable if the partial relation $\leq$
on the universal cover of the identity component of the group
of compactly supported contactomorphisms
defined by posing  $[\{\phi_t^{1}\}] \leq [\{\phi_t^{2}\}]$
if $[\{\phi_t^{2}\}] \cdot [\{\phi_t^{1}\}]^{-1}$
can be represented by a non-negative contact isotopy
is a bi-invariant partial order.
Most contact manifolds for which translated points
are proved to exist for all contactomorphisms contact isotopic to the identity
are also known to be orderable.
On the other hand,
the standard contact sphere,
on which existence of translated points fail,
is not orderable \cite{EKP},
and the results in \cite{HS, AArl}
also seem to suggest a relation between orderability
and existence of translated points.
We also note that the standard contact sphere
does not admit any non-trivial quasimorphism
or bi-invariant metric on the universal cover
of the contactomorphism group \cite{FPR},
in contrast to several of the contact manifolds
where translated points always exist
\cite{San10, CS, Givental, GKPS, FPR, BZ}.
Existence of translated points of contactomorphisms
is a manifestation of global rigidity,
as well as orderability and the existence of quasimorphisms
and bi-invariant metrics on (the universal cover of) the contactomorphism group.

In analogy with the contact case,
we say that a point $p$ on a lcs manifold $\big( M , [(\eta, \omega)] \big)$
is a translated point of a lcs diffeomorphism $\varphi$
(a diffeomorphism such that $\varphi^{\ast}\omega = e^g \omega$
for some function $g: M \rightarrow \mathbb{R}$,
called the conformal factor)
with respect to a representative $(\eta, \omega)$
if $p$ and $\varphi(p)$ are in the same orbit
of the Lee flow of $(\eta, \omega)$
and $\omega$ is preserved by $\varphi$ at $p$
(thus $g(p) = 0$).
This notion generalizes the notions of fixed points
and leafwise intersections of symplectomorphisms
(see Example \ref{example: translated points symplectic})
and of translated points of contactomorphisms
(see Example \ref{example: translated points contact}).
A pair $(T, \gamma)$ with $T \in \mathbb{R}$
and $\gamma: [0, 1] \rightarrow M$
a curve such that $\gamma(0) = \varphi (p)$,
$\gamma(1) = p$
and $\dot{\gamma} (t) = T \, R_{\omega} \big(\gamma(t)\big)$
for all $t$
is said to be a Lee chord of time-shift $T$
for the translated point $p$ of $\varphi$.
Such a Lee chord is said to be transverse 
if there is no non-zero vector $Y \in T_pM$
such that $(\varphi_T^{\omega} \circ \varphi)_{\ast} \, (Y) = Y$
and $dg (Y) = 0$.
A lcs diffeomorphism $\varphi$
of a Liouville lcs manifold $\big( M , [(\eta, \omega)] \big)$
is said to be exact
if for some (hence any) representative $(\eta, \omega, \lambda)$
the $1$-form $\lambda - e^{-g} \varphi^{\ast}\lambda$,
where $g$ is the conformal factor of $\varphi$,
is $\eta$-exact.
In particular, lcs Hamiltonian diffeomorphisms are exact
(Proposition \ref{proposition: Hamiltonian vs exact}).
We say that an exact lcs diffeomorphism
$\varphi$ has action $S_{\varphi}: M \rightarrow \mathbb{R}$
with respect to $(\eta, \omega, \lambda)$
if $\lambda - e^{-g} \varphi^{\ast}\lambda = d_{\eta}S_{\varphi}$.
If $M$ is connected and $\eta$ is not exact
then such action $S_{\varphi}$ is uniquely defined
(see Remark \ref{remark: unique solution}).
In this case,
we define the action with respect to $(\eta, \omega, \lambda)$
of a Lee chord $(T, \gamma)$
for a translated point $p$
of a lcs Hamiltonian diffeomorphism $\varphi$
to be the real number
\[
\mathcal{A}_{\varphi}(T, \gamma)
= - S_{\varphi}(p) + \int_{\gamma} \lambda \,.
\]
This notion generalizes the usual symplectic action of fixed points
of Hamiltonian diffeomorphisms
(see Example \ref{example: action translated point symplectic})
and the usual contact action of Reeb chords for translated points
of contactomorphisms
(see Example \ref{example: action translated point in conformal symplectization}).
We say that a Lee chord for a translated point is essential
if it has time-shift equal to the action,
and that a translated point is essential
if it admits an essential Lee chord.

\begin{thm}[Essential translated points]\label{theorem: translated points - intro}
\renewcommand{\theenumi}{\roman{enumi}} \
\begin{enumerate}
\item Let $\varphi$ be the time-$1$ map
of a sufficiently $\mathcal{C}^0$-small
lcs Hamiltonian isotopy
of a compact connected Liouville lcs manifold
$\big( M , [(\eta, \omega, \lambda)]\big)$
with $\eta$ not exact.
Then the number of essential translated points of $\varphi$
with respect to any representative $(\eta, \omega, \lambda)$
is greater than or equal to the cup-length of $M$,
and greater than or equal to the sum of the Betti numbers of $M$
if all the essential Lee chords for the essential translated points
are non-degenerate.
\item Every compactly supported lcs Hamiltonian diffeomorphisms
of $S^1 \times \mathbb{R}^{2n+1}$ or $S^1 \times \mathbb{R}^{2n} \times S^1$
with non-empty support
has at least one essential translated point
with respect to $(-d\theta, d_{-d\theta}\alpha_0, \alpha_0)$
in the interior of its support.
\end{enumerate}
\end{thm}

The first point of Theorem \ref{theorem: translated points - intro}
is a consequence of Theorem \ref{theorem: lcs Arnold conjecture intro}.
The argument is analogue to the one
for the corresponding result on translated points of contactomorphisms
\cite{S - Morse estimate for translated points, S - Iterated}.
In the contact case one observes that
the translated points of a contactomorphism $\phi$
of a contact manifold $(Y, \xi)$
with respect to a contact form $\alpha$
are in bijection with the Reeb chords
between the contact graph of $\phi$
in the contact product of $(Y, \xi)$
with respect to $\alpha$
and the contact graph of the identity.
By the $\mathcal{C}^0$-smallness assumption,
the contact graph of $\phi$ is contained in a Weinstein neighborhood
of the contact graph of the identity,
and so it can be identified to a Legendrian submanifold
of a $1$-jet bundle,
Legendrian isotopic to the zero section.
Existence of translated points then follows
from the existence of Reeb chords
between such Legendrian and the zero section.
In the lcs case,
there is a similar construction.
The lcs product of (Liouville) lcs manifolds
and the lcs graph of lcs diffeomorphisms
have been defined by Chantraine and Sackel \cite{CS}
(although we give a simpler description
of the lcs graph of divergence free lcs diffeomorphisms,
in particular lcs Hamiltonian diffeomorphisms,
see Section \ref{section: lcs product}).
We observe that the (essential) Lee chords
for the translated points
of a lcs Hamiltonian diffeomorphism $\varphi$
are in bijection with the (essential) Lee chords
between its lcs graph
and the lcs graph of the identity
(Proposition \ref{proposition: bijection translated points Lee chords graph}).
Using the Weinstein theorem in lcs geometry
to identify a neighborhood of the lcs graph of the identity
with a neighborhood of the zero section of a twisted cotangent bundle
we then deduce Theorem \ref{theorem: translated points - intro} (i)
from Theorem \ref{theorem: lcs Arnold conjecture intro}
(see Section \ref{section: gf}).

Before indicating the main ideas
of the proof of Theorem \ref{theorem: translated points - intro} (ii),
we discuss how the notion of orderability
can also be generalized to the lcs case,
referring to Section \ref{section: orderability}
for more details.
Let $\big( M, [(\eta, \omega)] \big)$ be a connected lcs manifold
such that $\eta$ is not exact.
For any lcs Hamiltonian vector field $X_t$
there is then a unique time dependent function $H_t$
such that $\iota_{X_t}\omega = - d_{\eta} H_t$
(see Remark \ref{remark: unique solution});
we say that $H_t$ is the Hamiltonian function
with respect to $(\eta, \omega)$
of the lcs Hamiltonian isotopy
generated by $X_t$.
We then define a lcs Hamiltonian isotopy
to be non-negative (respectively positive)
if its Hamiltonian function
with respect to some (hence any) representative $(\eta, \omega)$
is non-negative (respectively positive),
and consider the partial relation $\leq$
on the universal cover of the group
of compactly supported lcs Hamiltonian diffeomorphisms
defined by posing $[\{\varphi_t^{1}\}] \leq [\{\varphi_t^{2}\}]$
if $[\{\varphi_t^{2}\}] \cdot [\{\varphi_t^{1}\}]^{-1}$
can be represented by a non-negative lcs Hamiltonian isotopy.
We say that the lcs manifold $\big( M, [(\eta, \omega)] \big)$ is orderable
if $\leq$ is a bi-invariant partial order,
equivalently if there does not exist
any non-negative non-constant contractible loop
of compactly supported lcs Hamiltonian diffeomorphisms.
The locally conformal symplectization
of a contact manifold that is not orderable
(for instance the standard contact sphere)
is also not orderable,
because a non-constant non-negative contractible loop of contactomorphisms
can be lifted to a non-constant non-negative contractible loop
of lcs Hamiltonian diffeomorphisms of the locally conformal symplectization
(see Example \ref{Example:Lift of contact Hamiltonians}).
The converse implication is in general not clear,
but we show that it holds
for $S^1 \times \mathbb{R}^{2n+1}$ and $S^1 \times \mathbb{R}^{2n} \times S^1$.
We also consider the partial relation $\leq$
on the group of compactly supported lcs Hamiltonian diffeomorphisms
of a connected lcs manifold $\big( M, [(\eta, \omega)] \big)$
with $\eta$ not exact
defined by posing $\varphi_1 \leq \varphi_2$
if $\varphi_2 \circ \varphi_1^{-1}$
is the time-$1$ map of a non-negative lcs Hamiltonian isotopy,
and say that $\big( M, [(\eta, \omega)] \big)$ is strongly orderable
if $\leq$ is a bi-invariant partial order,
equivalently
if there does not exist
any non-negative non-constant loop
of compactly supported lcs Hamiltonian diffeomorphisms.
In particular,
strongly orderable lcs manifolds
are orderable.

\begin{thm}[Orderability]\label{theorem: orderability - intro}
\renewcommand{\theenumi}{\roman{enumi}}\
\begin{enumerate}
\item Let $B$ be a compact connected manifold,
and $\beta$ a closed non-exact $1$-form.
Then the twisted cotangent bundle $T_{\beta}^{\ast}B$
is strongly orderable.
\item The lcs manifolds $S^1 \times \mathbb{R}^{2n+1}$
and $S^1 \times \mathbb{R}^{2n} \times S^1$
are strongly orderable.
\end{enumerate}
\end{thm}

The first point of Theorem \ref{theorem: orderability - intro} 
is an easy consequence of the corresponding result
in the contact case for $1$-jet bundles,
proved by Colin, Ferrand and Pushkar \cite{CFP}
and Chernov and Nemirovsky \cite{CN}.
Indeed,
a non-negative non-constant loop
of compactly supported lcs Hamiltonian diffeomorphisms of  $T_{\beta}^{\ast}B$
would lift to a non-negative non-constant loop
of compactly supported contactomorphisms of  $J_{\beta}^1B$,
which would induce,
after conjugating with the untwisting map
$U_{\beta}: J^1_{\beta} B\rightarrow J^1B$,
a non-negative non-constant loop of compactly supported contactomorphisms of  $J^1B$
(see Section \ref{section: orderability}).

In order to prove Theorem \ref{theorem: translated points - intro} (ii)
and Theorem \ref{theorem: orderability - intro} (ii)
we use generating functions to define spectral selectors
for compactly supported lcs Hamiltonian diffeomorphisms
of $S^1 \times \mathbb{R}^{2n+1}$ and $S^1 \times \mathbb{R}^{2n} \times S^1$,
analogous to the spectral selectors for compactly supported
Hamiltonian diffeomorphisms of $\mathbb{R}^{2n}$
defined by Viterbo \cite{Viterbo}
and the spectral selectors for compactly supported contactomorphisms
contact isotopic to the identity of $\mathbb{R}^{2n+1}$ and $\mathbb{R}^{2n} \times S^1$
defined respectively by Bhupal \cite{Bhupal} and the third author \cite{San10, San11}.
The key observation is that
for the lcs manifold $S^1 \times \mathbb{R}^{2n+1}$
we do not just have the local identifications
of neighborhoods of Lagrangian submanifolds of the lcs product
with neighborhoods of the zero section of a twisted cotangent bundle
given by the lcs Weinstein theorem,
but also a global identification of the whole lcs product
with the twisted cotangent bundle $T_{-d\theta} (S^1 \times \mathbb{R}^{2n+1})$,
where $\theta$ denotes the coordinate on $S^1$.
Indeed,
the lcs product of $S^1 \times \mathbb{R}^{2n+1}$
can be identified to the locally conformal symplectization
$S^1 \times (\mathbb{R}^{2n+1} \times \mathbb{R}^{2n+1} \times \mathbb{R})$
of the contact product of $\mathbb{R}^{2n+1}$
(Proposition \ref{proposition: lcs product}),
and the twisted cotangent bundle $T_{-d\theta} (S^1 \times \mathbb{R}^{2n+1})$
can be identified to the locally conformal symplectization $S^1 \times J^1\mathbb{R}^{2n+1}$
of $J^1\mathbb{R}^{2n+1}$
(Lemma \ref{lemma: identification twisted cotangent bundle with conformal symplectization}).
We can thus lift the strict contactomorphism
from $\mathbb{R}^{2n+1} \times \mathbb{R}^{2n+1} \times \mathbb{R}$
to $J^1\mathbb{R}^{2n+1}$
used in \cite{Bhupal, San10, San11}
to the locally conformal symplectizations
to obtain the desired identification
(see equation \eqref{equation: tau} in Section \ref{section: lcs product etc}).
By considering the image by this identification
of the lcs graph of a lcs Hamiltonian diffeomorphism $\varphi$ of $S^1 \times \mathbb{R}^{2n+1}$
we obtain a Lagrangian submanifold $L_{\varphi}$ of $T_{-d\theta} (S^1 \times \mathbb{R}^{2n+1})$.
If $\varphi$ is compactly supported
then $L_{\varphi}$ can be compactified to a Lagrangian submanifold
$\overline{L_{\varphi}}$ of $T_{-d\theta} (S^1 \times S^{2n+1})$,
lcs Hamiltonian isotopic to the zero section.
Similarly,
we can associate to a compactly supported lcs Hamiltonian diffeomorphism $\varphi$
of $S^1 \times \mathbb{R}^{2n} \times S^1$
a Lagrangian submanifold $L_{\varphi}$ of $T_{-d\theta} (S^1 \times \mathbb{R}^{2n} \times S^1)$,
which can be compactified to a Lagrangian submanifold
$\overline{L_{\varphi}}$ of $T_{-d\theta} (S^1 \times S^{2n} \times S^1)$,
lcs Hamiltonian isotopic to the zero section.
Using generating functions quadratic at infinity for $\overline{L_{\varphi}}$
we define, exactly as in the contact case
for contactomorphisms of $\mathbb{R}^{2n+1}$ and $\mathbb{R}^{2n} \times S^1$
\cite{Bhupal, San11, San10}
spectral selectors $c_+ (\varphi) \geq 0$ and $c_- (\varphi) \leq 0$
for any compactly supported lcs Hamiltonian diffeomorphism
$\varphi$ of $S^1 \times \mathbb{R}^{2n+1}$ or $S^1 \times \mathbb{R}^{2n} \times S^1$.
These spectral selectors have the same properties
as in the contact case
(Proposition \ref{proposition: properties of spectral invariants}),
and thus can be used to obtain similar applications.
To start with,
spectrality 
($c_+ (\varphi)$ and $c_- (\varphi)$ are equal to the action
of an essential translated point of $\varphi$)
and non-degeneracy
($c_+ (\varphi) = c_- (\varphi) = 0$ if and only if $\varphi$ is the identity)
imply Theorem \ref{theorem: translated points - intro} (ii).
Furthermore,
as we will see,
the properties listed in Proposition \ref{proposition: properties of spectral invariants}
imply that the partial relation $\preceq$
on the group of compactly supported lcs Hamiltonian diffeomorphisms
of $S^1 \times \mathbb{R}^{2n+1}$ or $S^1 \times \mathbb{R}^{2n} \times S^1$
defined by posing $\varphi_1 \preceq \varphi_2$
if $c_+ (\varphi_1 \circ \varphi_2^{-1}) = 0$
is a bi-invariant partial order
(Theorem \ref{theorem: po 1}).
By comparing the partial relation $\preceq$
with the one $\leq$ used to define orderability
we then obtain a proof of Theorem \ref{theorem: orderability - intro} (ii).
Moreover,
we show that the arguments in \cite{San11, San10}
can be generalized to the lcs case
to obtain an unbounded integer-valued bi-invariant metric
$d$ on the group of compactly supported lcs Hamiltonian diffeomorphisms
of $S^1 \times \mathbb{R}^{2n} \times S^1$,
compatible with the partial orders $\preceq$ and $\leq$,
by posing 
\[
d (\varphi_1, \varphi_2)
= \lceil c_+ (\varphi_1 \circ \varphi_2^{-1}) \rceil
- \lfloor c_- (\varphi_1 \circ \varphi_2^{-1}) \rfloor
\]
(Theorems \ref{theorem: metric}
and \ref{theorem: partiall ordered metric spaces}).
A lcs analogue of a known argument in contact geometry
shows that the group of compactly supported lcs diffeomorphisms
of any lcs manifold $\big( M , [(\eta, \omega)] \big)$
does not admit any bi-invariant metric
that takes values arbitrarily close to zero
(See Remark~\ref{Rmk:Finemetrics}).
This argument, however, does not directly apply
to the group of lcs Hamiltonian diffeomorphisms.
This motivates the following question.

\textbf{Question.}
\textit{Does there exist a bi-invariant metric on $\Ham(S^1 \times \mathbb{R}^{2n} \times S^1)$
that takes values arbitrarily close to zero?}

Finally,
we show that the arguments in \cite{San11, San10}
can be generalized to the lcs case
to obtain an integer-valued  lcs capacity for domains,
by posing
\[
c (\mathcal{U})
= \sup \{\, \lceil c_+ (\varphi) \rceil \;\lvert \; \varphi \in \Ham (\mathcal{U})\,\} \,,
\]
where $\Ham (\mathcal{U})$ denotes
the group of lcs diffeomorphisms of $S^1 \times \mathbb{R}^{2n} \times S^1$
that are the time-$1$ map
of a compactly supported lcs Hamiltonian isotopy supported in $\mathcal{U}$
(see Lemma \ref{lemma: displacement}
and Proposition \ref{proposition: properties lcs capacity}).
Using the lcs capacity we then obtain the following lcs non-squeezing result,
which is analogous to the contact non-squeezing theorem for integers
in $\mathbb{R}^{2n} \times S^1$
discovered by Eliashberg, Kim and Polterovich \cite{EKP}
and reproved in \cite{San11} using generating functions.

\begin{thm}[Lcs non-squeezing for integers]\label{theorem: main}
If $\pi R_2^2 \leq k \leq \pi R_1^2$ for $k \in \mathbb{N}_0$
then there is no compactly supported lcs Hamiltonian diffeomorphism $\varphi$
of $S^1 \times \mathbb{R}^{2n} \times S^1$
such that
\[
\varphi \big(\overline{S^1 \times B^{2n}(R_1) \times S^1}\big)
\subset S^1 \times B^{2n}(R_2) \times S^1 \,.
\]
\end{thm}

In \cite{EKP} it is also proved that
if $\pi R_1^2 < 1$ then $B^{2n}(R_1) \times S^1$
can be squeezed by a contact isotopy into $B^{2n}(R_2) \times S^1$
for $R_2$ arbitrarily small.
The analogous result is true also in the lcs case,
indeed (using Example \ref{Example:Lift of contact Hamiltonians})
the contact isotopy squeezing $B^{2n}(R_1) \times S^1$ into $B^{2n}(R_2) \times S^1$
can be lifted to a lcs Hamiltonian isotopy of $S^1 \times \mathbb{R}^{2n} \times S^1$
squeezing $S^1 \times B^{2n}(R_1) \times S^1$ into $S^1 \times B^{2n}(R_2) \times S^1$.
In \cite{Chiu, Fraser, FSZ}
the contact non-squeezing theorem for integers
of \cite{EKP} and \cite{San10}
has been generalized to prove that
if $1 \leq \pi R_2^2 \leq \pi R_1^2$
then there is no compactly supported contact isotopy
of $\mathbb{R}^{2n} \times S^1$
squeezing the closure of  $B^{2n}(R_1) \times S^1$ into $B^{2n}(R_2) \times S^1$.
We expect that it should be possible to extend to the lcs case
the equivariant generating function homology defined in \cite{FSZ},
in order to prove that if $1 \leq \pi R_2^2 \leq \pi R_1^2$
then there is no compactly supported lcs Hamiltonian diffeomorphism $\varphi$
of $S^1 \times \mathbb{R}^{2n} \times S^1$
such that
\[
\varphi \big(\overline{S^1 \times B^{2n}(R_1) \times S^1}\big)
\subset S^1 \times B^{2n}(R_2) \times S^1 \,.
\]
This will be the object of a forthcoming work.
The above discussion also motivates the following question.

\textbf{Question.} \textit{Does there exist a lcs diffeomorphism $\varphi$ of $S^1 \times \R^{2n} \times S^1 $ such that  $\varphi(\overline{S^1 \times B^{2n}(R_1) \times S^1}) \subset S^1 \times B^{2n}(R_2) \times S^1 $ with $1 \leq R_2 < R_1$?} 

When the lcs manifold is $\big(\R^{2n}, [(0,\omega_0)]\big)$ then such a lcs diffeomorphism is provided by the Liouville flow. However, we currently do not know whether such a lcs squeezing exists for $S^1 \times \R^{2n} \times S^1$. 

Finally we remark that,
as in the work of Serraille and Stojisavljevi\'c \cite{SS},
Theorem \ref{theorem: main} implies that
the group of compactly supported lcs Hamiltonian diffeomorphisms
of $S^1 \times \mathbb{R}^{2n} \times S^1$
does not have the Rokhlin property
(see Section \ref{section: Rokhlin}).

The article is organized as follows.
In Section \ref{section: lcs manifolds}
we gather the basic notions on lcs manifolds
that are needed in the rest of the article.
Moreover, we discuss in \ref{section: orderability}
the notion of orderability
and prove Theorem \ref{theorem: orderability - intro} (i).
In Section \ref{section: Lee chords and translated points}
we define essential Lee chords and translated points.
In Section \ref{section: gf}
we recall from \cite{CM}
the theory of generating functions
for Lagrangian submanifolds of twisted cotangent bundles,
and prove Theorem \ref{theorem: lcs Arnold conjecture intro}
and Theorem \ref{theorem: translated points - intro} (i).
In Section \ref{section: lcs product etc}
we define generating functions for compactly supported
lcs Hamiltonian diffeomorphisms of $S^1 \times \mathbb{R}^{2n+1}$
and $S^1 \times \mathbb{R}^{2n} \times S^1$,
and in Section \ref{section: spectral invariants}
we use them to define the spectral selectors $c_+$ and $c_-$.
A first application of these spectral selectors,
Theorem \ref{theorem: translated points - intro} (ii),
is also discussed in Section \ref{section: spectral invariants}.
In Section \ref{section: order and metric}
we use the spectral selectors
to define a partial order on the group of compactly supported
lcs Hamiltonian diffeomorphisms of $S^1 \times \mathbb{R}^{2n+1}$
and $S^1 \times \mathbb{R}^{2n} \times S^1$,
and a bi-invariant metric on the group of compactly supported
lcs Hamiltonian diffeomorphisms of $S^1 \times \mathbb{R}^{2n} \times S^1$,
and obtain a proof of Theorem \ref{theorem: orderability - intro} (ii).
In Section \ref{section: capacity}
we define a lcs capacity for domains of $S^1 \times \mathbb{R}^{2n} \times S^1$,
and use it to prove the lcs non-squeezing theorem for integers
(Theorem \ref{theorem: main}).
Finally, in Section \ref{section: Rokhlin}
we discuss, following \cite{SS},
how this non-squeezing result
implies that the group of compactly supported lcs Hamiltonian diffeomorphisms
of $S^1 \times \mathbb{R}^{2n} \times S^1$
does not have the Rokhlin property.

\subsection*{Acknowledgments} 
We thank the organizers of the conference
\textit{Workshop on conservative dynamics and symplectic geometry} 
at IMPA in July 2025 and those of the \textit{Workshop on locally conformal symplectic topology} at Nantes University in November 2025
for the opportunity to present this work.
We also thank Baptiste Chantraine, Yasha Eliashberg and Pacôme Van Overschelde for stimulating discussions. MB and PC are supported by the EoS grant ``Geometry and Beyond" (EOS project 40007524).

\section{Locally conformal symplectic manifolds}\label{section: lcs manifolds}

In this section we present
the basic notions and results on lcs geometry
that are needed in the rest of the article,
and take the opportunity
to also discuss $\mathcal{C}^0$-rigidity
of the group of lcs diffeomorphisms
\ref{section: C0 rigidity}.
Although the material is classic
(except for $\mathcal{C}^0$-rigidity
\ref{section: C0 rigidity},
for the lcs product and graph \ref{section: lcs product},
introduced by Chantraine and Sackel \cite{CS},
and for orderability \ref{section: orderability}),
we include detailed proofs
whenever we could not find them in the literature.
Among the basic results discussed in this section,
Corollary \ref{corollary: symplectic lift}
is particularly important for us,
because it allows to define
the (unique) lift to the symplectic cover
of a divergence free lcs diffeomorphism.
Using this notion we give in Section \ref{section: lcs product}
a description of the lcs graph of a divergence free lcs diffeomorphism
that seems simpler than the one in \cite{CS}.
In \ref{section: orderability}
we introduce the notion of orderability in lcs geometry,
and prove Theorem \ref{theorem: orderability - intro} (i).

\subsection{The Lichnerowicz--de Rham differential}

The Lichnerowicz--de Rham differential
with respect to a closed 1-form $\eta$ on a manifold $M$
is the map $\Omega^d (M) \rightarrow \Omega^{d+1}(M)$
defined by
\[
d_{\eta} \sigma = d\sigma - \eta \wedge \sigma \,.
\]
A form $\sigma \in \Omega^d(M)$ is said to be $\eta$-closed
if $d_{\eta}\sigma = 0$,
and $\eta$-exact if $\sigma = d_{\eta}\gamma$
for some $\gamma \in \Omega^{d-1}(M)$.
The Lichnerowicz--de Rham differential
satisfies the relations
\[
d_{\eta}^2 = 0 \,,
\]
\[
d_{\eta + df} \sigma = e^f \, d_{\eta} (e^{-f} \sigma)
\]
for any function $f: M \rightarrow \mathbb{R}$,
and
$$
(\iota_X \circ d_{\eta} + d_{\eta} \circ \iota_X) \sigma
= \mathcal{L}_X \sigma - \eta(X) \sigma
$$
for any vector field $X$.

\begin{remark}\label{remark: unique solution}
If $M$ is connected and $\eta$ is not exact
then the equation $d_\eta f = 0$
for $f: M \rightarrow \mathbb{R}$
has a unique solution $f \equiv 0$.
Indeed,
suppose by contradiction that the open subset $\mathcal{U}$
on which $f$ does not vanish is not empty.
Then $\ln |f|_\mathcal{U}|$ is a primitive of $\eta$ on $\mathcal{U}$,
thus $\mathcal{U}$ is a proper subset of $M$
and its closure $\overline{\mathcal{U}}$
has a non-empty boundary $\partial \overline{\mathcal{U}}$.
Let $p$ be a point of $\partial \overline{\mathcal{U}}$,
and $\mathcal{V}_p$ a contractible neighborhood of $p$.
On $\mathcal{V}_p$ the $1$-form $\eta$ has a primitive, say $h_p$.
On each connected component of $\mathcal{U} \cap \mathcal{V}_p$
the functions $\ln|f|$ and $h_p$ differ by a constant.
The point $p$ belongs to the closure
of (at least) one of these connected components,
which we denote by ${\mathcal U}_p$.
Then $\ln|f|_{{\mathcal U}_p}| = h_p + c$,
and so
$$
|f|_{{\mathcal U}_p}| = e^{h_p + c} \,.
$$
In particular,
$f$ does not vanish at $p$
and so $p$ cannot belong to $\partial \overline{\mathcal{U}}$.
This contradiction
forces $\mathcal{U}$ to be empty,
and so $f$ to be constantly zero.
\end{remark}

\subsection{Lcs structures}

A locally conformal symplectic (lcs) form
on a (necessarily even dimensional) manifold $M$
is a non-degenerate $2$-form $\omega$
such that, for some open cover $\{\mathcal{U}_i\}$ of $M$,
\[
\left. \omega \right\lvert_{\mathcal{U}_i}
= e^{\mu_i} \, \omega_i
\]
for a function $\mu_i: \mathcal{U}_i \rightarrow \mathbb{R}$
and a symplectic form $\omega_i$ on $\mathcal{U}_i$.
If $M$ is $2$-dimensional then $\omega$ is necessarily closed,
hence a symplectic form.
From now on we thus assume that
the dimension of $M$ is at least $4$.
The fact that on each intersection
$\mathcal{U}_i \cap \mathcal{U}_j$
we have $e^{\mu_i} \, \omega_i = e^{\mu_j} \, \omega_j$,
and so
\[
0 = d\omega_i = d (e^{\mu_j - \mu_i} \omega_j)
= d (e^{\mu_j - \mu_i}) \wedge \omega_j \,,
\]
implies then that $d (e^{\mu_j - \mu_i}) = 0$
and so $d\mu_i = d\mu_j$.
The $1$-forms $d\mu_i$
thus glue to a well-defined closed $1$-form $\eta$ on $M$.
Since
\[
d (\left. \omega \right\lvert_{\mathcal{U}_i})
= e^{\mu_i} \, d\mu_i \wedge \omega_i
= d\mu_i \wedge \left. \omega \right\lvert_{\mathcal{U}_i}
\]
for every $i$,
we have $d\omega = \eta \wedge \omega$
and so $d_{\eta}\omega = 0$.
Conversely,
if $\omega$ is a non-degenerate $2$-form on $M$
such that $d_{\eta}\omega = 0$
for some closed $1$-form $\eta$
then $\omega$ is a lcs form.
Indeed, there exists an open cover $\{\mathcal{U}_i\}$
such that $\eta$ is exact on each $\mathcal{U}_i$,
and so $\left. \eta \right\lvert_{\mathcal{U}_i} = d\mu_i$
for functions $\mu_i: \mathcal{U}_i \rightarrow \mathbb{R}$.
Posing $\omega_i = e^{-\mu_i} \, \left. \omega \right\lvert_{\mathcal{U}_i}$,
since $d\omega = \eta \wedge \omega$ we have
\[
d\omega_i
= e^{-\mu_i} \, \big( d(\left. \omega \right\lvert_{\mathcal{U}_i})
- \left. \eta \right\lvert_{\mathcal{U}_i} \wedge \left. \omega \right\lvert_{\mathcal{U}_i} \big)
= 0 \,,
\]
and so $\omega_i$ is a symplectic form.
A lcs form on $M$ can thus equivalently be defined
as a non-degenerate $2$-form $\omega$
such that $d_{\eta}\omega = 0$
for some closed $1$-form $\eta$.
Since $\omega$ is non-degenerate,
the relation $d_{\eta}\omega = 0$
uniquely determines the closed $1$-form $\eta$.
We say that $\eta$ is the Lee form
of the lcs form $\omega$.

A locally conformal symplectic (lcs) structure on $M$
is an equivalence class of lcs forms
under the equivalence relation
\[
\omega \sim e^f \omega
\quad \text{ for } f: M \rightarrow \mathbb{R} \,.
\]
If $\eta$ is the Lee form of $\omega$
then the Lee form of $e^f \omega$ is $\eta + df$.
We can thus equivalently define a lcs structure
as an equivalence class $[(\eta, \omega)]$
of pairs formed by a closed $1$-form $\eta$
and a non-degenerate $\eta$-closed $2$-form $\omega$,
under the equivalence relation
\[
(\eta, \omega) \sim (\eta + df \,,\, e^f \omega)
\quad \text{ for } f: M \rightarrow \mathbb{R} \,.
\]
A locally conformal symplectic (lcs) manifold
is a manifold endowed with a lcs structure.

A lcs manifold $\big(M, [(\eta, \omega)] \big)$
is said to be exact
if for some (hence any) representative $(\eta, \omega)$
the $2$-form $\omega$ is $\eta$-exact.
In this case,
any $1$-form $\lambda$ such that $\omega = d_{\eta} \lambda$
is said to be a Liouville form for $(\eta, \omega)$.
For any $f: M \rightarrow \mathbb{R}$,
if $\lambda$ is a Liouville form for $(\eta, \omega)$
then $e^f \lambda$ is a Liouville form
for $(\eta + df \,,\, e^f \omega)$.
We denote by $[(\eta, \omega, \lambda)]$
the equivalence class of triples $(\eta, \omega, \lambda)$
under the equivalence relation
\[
(\eta, \omega, \lambda) \sim (\eta + df \,,\, e^f \omega \,,\, e^f \lambda)
\quad \text{ for } f: M \rightarrow \mathbb{R} \,,
\]
and we say that $\big(M, [(\eta, \omega, \lambda)] \big)$
is a Liouville lcs manifold.
The Liouville vector field
of a Liouville lcs manifold $\big(M, [(\eta, \omega, \lambda)] \big)$
is the vector field $Z$ on $M$
defined by $\iota_{Z} \omega = \lambda$
for some (hence any) representative $(\eta, \omega, \lambda)$.
It only depends on the Liouville lcs structure $[(\eta, \omega, \lambda)]$,
and not on a specific representative.

\begin{example}[Conformal symplectic manifolds]\label{example: symplectic manifolds}
If $\omega$ and $e^f \omega$
are symplectic forms on a manifold $M$
for some function $f: M \rightarrow \mathbb{R}$,
then $f$ is necessarily constant.
A conformal symplectic manifold
is thus defined to be a manifold $M$
endowed with an equivalence class $[\omega]$
of symplectic forms
under the equivalence relation
$\omega \sim e^c \omega$ for $c \in \mathbb{R}$.
If $\big( M , [(\eta, \omega)] \big)$
is a lcs manifold such that
for some (hence any) representative $(\eta, \omega)$
the Lee form $\eta$ is exact,
then for any primitive $\mu$ of $\eta$ we have
\[
(\eta, \omega) = (d\mu, \omega) \sim (0, e^{-\mu}\omega) \,,
\]
and so $e^{-\mu}\omega$ is a symplectic form.
The conformal class $[e^{-\mu}\omega]$
of this symplectic form
only depends on the lcs structure $[(\eta, \omega)]$.
We can thus associate to the lcs manifold $\big( M , [(\eta, \omega)] \big)$
the conformal symplectic manifold $(M, [e^{-\mu}\omega])$.
Conversely,
we can identify a conformal symplectic manifold $(M, [\omega])$
with the lcs manifold $\big( M , [(0, \omega)] \big)$.
We say that a conformal symplectic manifold $(M, [\omega])$ is exact
if some (hence any) representative $\omega$ is exact.
In this case,
any $1$-form $\lambda$
such that $\omega = d\lambda$
is said to be a Liouville form for $\omega$.
If $\lambda$ is a Liouville form for $\omega$
then $e^c \lambda$ is a Liouville form for $e^c \omega$.
We denote by $[(\omega, \lambda)]$
the equivalence class of pairs $(\omega, \lambda)$
under the equivalence relation
$(\omega, \lambda) \sim (e^c \omega, e^c \lambda)$,
and say that $\big( M, [(\omega, \lambda)] \big)$
is a conformal Liouville manifold.
If $\big( M, [(\eta, \omega)] \big)$
is an exact lcs manifold
with $\eta = d\mu$
then the associated conformal symplectic manifold
$(M, [e^{-\mu} \omega])$ is exact,
and if $\lambda$ is a Liouville form for $(\eta, \omega)$
then $e^{-\mu} \lambda$ is a Liouville form
for $e^{-\mu} \omega$.
We say that $\big( M, [( e^{-\mu} \omega, e^{-\mu} \lambda )] \big)$
is the conformal Liouville manifold
associated to the Liouville lcs manifold
$\big( M, [(\eta, \omega, \lambda)] \big)$.
\end{example}

\begin{example}[Twisted cotangent bundles]\label{example: twisted cotangent bundle}
Let $\beta$ be a closed 1-form on a manifold $B$,
and consider the cotangent bundle $\pi: T^{\ast}B \rightarrow B$,
endowed with the tautological $1$-form $\lambda$.
The twisted cotangent bundle of $B$ with respect to $\beta$
is the Liouville lcs manifold
$$
T^*_\beta B :=
\big( \, T^{\ast}B \,,\,
[(\pi^{\ast}\beta \,,\, d_{\pi^{\ast}\beta} \lambda \,,\, \lambda) ] \, \big).
$$
The Liouville vector field of $T^*_\beta B$
coincides with the standard Liouville vector field,
which in local coordinates $(q_1, \cdots, q_n, p_1, \cdots, p_n)$
is given by
$Z = \sum_{i} p_i \, \frac{\partial}{\partial p_i}$.
If $\beta$ is exact with primitive $\mu$ then
\[
(\pi^{\ast}\beta \,,\, d_{\pi^{\ast}\beta} \lambda \,,\, \lambda)
= \big( d (\mu \circ \pi) \,,\, d_{d(\mu \circ \pi)} \lambda \,,\, \lambda \big)
\sim (0 \,,\, e^{-\mu \circ \pi} d_{d(\mu \circ \pi)} \lambda \,,\, e^{-\mu \circ \pi} \lambda )
= \big( 0 \,,\, d( e^{-\mu \circ \mu} \lambda ) \,,\, e^{-\mu \circ \mu} \lambda \big) \,,
\]
thus the conformal Liouville manifold associated to $T^*_\beta B$
is $\big( T^{\ast}B \,,\,
[( d( e^{-\mu \circ \mu} \lambda ), e^{-\mu \circ \mu} \lambda )] \big)$.
The map $\sigma_q \mapsto e^{\mu(q)} \sigma_q$
is a diffeomorphism of $T^{\ast}B$
that pulls back $e^{-\mu \circ \pi} \lambda$ to $\lambda$,
hence the conformal Liouville structure
$[( d( e^{-\mu \circ \mu} \lambda ), e^{-\mu \circ \mu} \lambda )]$
to the conformal class $[( d\lambda , \lambda )]$
of the standard Liouville structure.
\end{example}

\begin{example}[Locally conformal symplectizations]
The locally conformal symplectization
of a co-oriented contact manifold $( Y, \xi)$
with respect to a contact form $\alpha$ for $\xi$ is the Liouville lcs manifold
\[
\big(\, S^1 \times Y \,,\, [(-d\theta \,,\, d_{-d\theta}\alpha \,,\, \alpha )] \,\big),
\]
where $\theta$ denotes the standard coordinate on $S^1$
and where we still denote by $\alpha$
its pullback by the projection $S^1 \times Y \rightarrow Y$.
Its Liouville vector field is $\frac{\partial}{\partial \theta}$.
\end{example}

\subsection{Symplectic covers}

Let $\big(M, [(\eta, \omega)]\big)$ be a lcs manifold.
Consider the cover $\pi: \bar{M} \rightarrow M$
corresponding to the kernel of the homomorphism
\[
\langle [\eta] , \,\cdot\, \,\rangle: \pi_1(M) \rightarrow \mathbb{R} \,,\;
\langle [\eta] , [\gamma] \rangle
= \int_0^1 \eta \Big(\frac{d \gamma}{d t}\Big) \, dt \,,
\]
thus $\bar{M}$ is the space
of equivalence classes of paths in $M$
starting at a fixed base point
for the equivalence relation
$$
\gamma_1 \sim \gamma_2
\quad \text{ if } \gamma_1 (1) = \gamma_2 (1)
\text{ and }
\int_0^1 \eta \big(\frac{d \gamma_1}{d t}\big) \, dt =
\int_0^1 \eta \big(\frac{d \gamma_2}{d t}\big) \, dt \,,
$$
and the projection $\pi: \bar{M} \rightarrow M$
is given by $\pi ([\gamma]) = \gamma(1)$.
The $1$-form $\pi^{\ast}\eta$ is exact,
with primitive
\[
\bar{M} \rightarrow \mathbb{R} \,,\;
[\gamma] \mapsto \int_0^1 \eta \Big(\frac{d \gamma}{d t}\Big) \, dt \,.
\]
The next lemma follows directly from the definitions.

\begin{lemma}\label{lemma: symplectic cover}
Let $\mu: \bar{M} \rightarrow \mathbb{R}$
be a primitive of $\pi^{\ast}\eta$.
If two points $x$ and $y$ of $\bar{M}$ satisfy
$\pi (x) = \pi (y)$ and $\mu (x) = \mu (y)$
then $x = y$.
\end{lemma}

The form $\pi^{\ast}\omega$ is a lcs form on $\bar{M}$,
with Lee form $\pi^{\ast}\eta$.
Since $\pi^{\ast}\eta$ is exact,
by Example \ref{example: symplectic manifolds}
we can associate to the lcs manifold
$\big( \bar{M} , [(\pi^{\ast}\eta, \pi^{\ast}\omega)] \big)$
the conformal symplectic manifold
$( \bar{M}, [e^{-\mu} \pi^{\ast}\omega])$,
where $\mu$ is any primitive of $\pi^{\ast}\eta$.
We say that $(\bar{M}, [e^{-\mu} \pi^{\ast} \omega])$
is the symplectic cover
of the lcs manifold $\big(M, [(\eta, \omega)]\big)$.
If the lcs manifold $\big(M, [(\eta, \omega)]\big)$ is exact
and $\lambda$ a Liouville form for $(\eta, \omega)$
then the symplectic cover $(\bar{M}, [e^{-\mu} \pi^{\ast} \omega])$
is exact and $e^{-\mu} \pi^{\ast} \lambda$
is a Liouville form for $e^{-\mu} \pi^{\ast} \omega$.
We say that the conformal Liouville manifold
$\big(\bar{M}, [(e^{-\mu} \pi^{\ast} \omega, e^{-\mu} \pi^{\ast} \lambda)]\big)$
is the symplectic cover
of the Liouville lcs manifold $\big(M, [(\eta, \omega, \lambda)]\big)$.

\begin{example}
The symplectic cover
of the locally conformal symplectization
$\big(\, S^1 \times Y \,,\, [(-d\theta \,,\, d_{-d\theta}\alpha \,,\, \alpha )] \,\big)$
of a co-oriented contact manifold $\big( Y, \xi = \ker (\alpha)\big)$
is the conformal class
$\big(\mathbb{R} \times Y \,,\, [(d (e^{\theta}\alpha) \,,\, e^{\theta} \alpha)]\big)$
of the usual symplectization.
\end{example}

\subsection{Lcs diffeomorphisms}

A diffeomorphism $\varphi: M_1 \rightarrow M_2$
between two lcs manifolds $\big(M_1, [(\eta_1, \omega_1 )]\big)$
and $\big(M_2, [(\eta_2, \omega_2 )]\big)$
is said to be a locally conformal symplectic (lcs) diffeomorphism
if for any representatives $(\eta_1, \omega_1)$ and $(\eta_2, \omega_2)$
there is a function $g: M_1 \rightarrow \mathbb{R}$,
called the conformal factor of $\varphi$
with respect to $(\eta_1, \omega_1)$ and $(\eta_2, \omega_2)$,
such that $\varphi^{\ast} \omega_2 = e^g \omega_1$.
The relation $\varphi^{\ast} \omega_2 = e^g \omega_1$
implies that $\varphi^{\ast} \eta_2 = \eta_1 + dg$.
We say that the lcs diffeomorphism $\varphi$ is strict
with respect to $(\eta_1, \omega_1)$ and $(\eta_2, \omega_2)$
(or just that it is strict
when the representatives we consider are clear from the context)
if $\varphi^{\ast} \omega_2 = \omega_1$,
and so $\varphi^{\ast} \eta_2 = \eta_1$.

\begin{example}
A conformal symplectomorphism
between two conformal symplectic manifolds
$(M_1, [\omega_1])$ and $(M_2, [\omega_2])$
is a diffeomorphism $\varphi: M_1 \rightarrow M_2$
such that for any representatives
$\omega_1$ and $\omega_2$
there is a constant $c \in \mathbb{R}$,
called the conformal factor of $\varphi$
with respect to $\omega_1$ and $\omega_2$,
such that $\varphi^{\ast} \omega_2 = e^c \, \omega_1$.
Seeing $(M_j, [\omega_j])$ for $j = 1, 2$
as lcs manifolds $\big( M_j , [(0, \omega_j)] \big)$,
any conformal symplectomorphism is a lcs diffeomorphism.
Conversely,
any lcs diffeomorphism $\varphi$
between two lcs manifolds $\big( M_1, [(\eta_1, \omega_1)] \big)$
and $\big( M_2, [(\eta_2, \omega_2)] \big)$
with $\eta_1 = d\mu_1$ and $\eta_2 = d\mu_2$
is a conformal symplectomorphism
between the associated conformal symplectic manifolds
$(M_1, [e^{-\mu_1}\omega_1])$ and $(M_2, [e^{-\mu_2}\omega_2])$.
Indeed,
\[
\varphi^{\ast} (e^{-\mu_2}\omega_2)
= e^{- \mu_2 \circ \varphi + g + \mu_1} \, e^{-\mu_1} \omega_1 \,,
\]
where $g$ is the conformal factor of $\varphi$
with respect to $(\eta_1, \omega_1)$ and $(\eta_2, \omega_2)$,
and
\[
d (\mu_2 \circ \varphi) = \varphi^{\ast}\eta_2
= \eta_1 + dg = d (\mu_1 + g) \,,
\]
thus $- \mu_2 \circ \varphi + g + \mu_1 = c$
for some $c \in \mathbb{R}$.
\end{example}

Similarly, we have the following result.

\begin{prop}\label{proposition: lift of lcs diffeo to symplectic cover}
Consider two lcs manifolds $\big(M_1 , [(\eta_1, \omega_1)]\big)$
and $\big(M_2 , [(\eta_2, \omega_2)]\big)$,
and their symplectic covers
$(\bar{M}_1, [e^{-\mu_1}\pi_1^* \omega_1])$
and $(\bar{M}_2, [e^{-\mu_2}\pi_2^* \omega_2])$.
Suppose that $\varphi: M_1 \rightarrow M_2$
and $\bar{\varphi}: \bar{M}_1 \rightarrow \bar{M}_2$
are diffeomorphisms such that
$\pi_2 \circ \bar{\varphi} = \varphi \circ \pi_1$.
Then $\varphi$ is a lcs diffeomorphism
if and only if $\bar{\varphi}$ is a conformal symplectomorphism.
In this case,
if $g$ is the conformal factor of $\varphi$
with respect to $(\eta_1, \omega_1)$ and $(\eta_2, \omega_2)$
then
$$
- \mu_2 \circ \bar{\varphi} + g \circ \pi_1 + \mu_1 = c
$$
for some $c \in \mathbb{R}$,
and $c$ is the conformal factor of $\bar{\varphi}$
with respect to $e^{-\mu_1}\pi_1^* \omega_1$ and $e^{-\mu_2}\pi_2^* \omega_2$.
\end{prop}

\begin{proof}
We denote $\bar{\omega}_j = e^{-\mu_j} \, \pi_j^{\ast}\omega_j$.
Suppose that $\varphi$ is a lcs diffeomorphism,
with conformal factor $g$
with respect to $(\eta_1, \omega_1)$ and $(\eta_2, \omega_2)$.
Then $\varphi^{\ast}\omega_2 = e^g\omega_1$,
thus $\pi_1^{\ast} (\varphi^{\ast} \omega_2) = \pi_1^{\ast} (e^g \omega_1)$.
But
\[
\pi_1^{\ast} (\varphi^{\ast} \omega_2)
= \bar{\varphi}^{\ast} (\pi_2^{\ast} \omega_2)
= \bar{\varphi}^{\ast} ( e^{\mu_2} \bar{\omega}_2)
= e^{\mu_2 \circ \bar{\varphi}} \, \bar{\varphi}^{\ast} \bar{\omega}_2
\]
and
\[
\pi_1^{\ast} (e^g \omega_1)
= e^{g \circ \pi_1} \, \pi_1^{\ast} \omega_1
= e^{g \circ \pi_1 + \mu_1} \, \bar{\omega}_1 \,,
\]
thus
\[
\bar{\varphi}^{\ast} \bar{\omega}_2
= e^{- \mu_2 \circ \bar{\varphi} + g \circ \pi_1 + \mu_1} \, \bar{\omega}_1 \,.
\]
Since
\[
d \, (- \mu_2 \circ \bar{\varphi} + g \circ \pi_1 + \mu_1)
= \pi_1^{\ast} (- \varphi^{\ast} \eta_2 + dg + \eta_1) = 0 \,,
\]
we have $- \mu_2 \circ \bar{\varphi} + g \circ \pi_1 + \mu_1 = c$
for some $c \in \mathbb{R}$.
We conclude that
$\bar{\varphi}^{\ast} \bar{\omega}_2 = e^c \, \bar{\omega}_1$,
thus $\bar{\varphi}$ is a conformal symplectomorphism
with conformal factor $c$
with respect to $\bar{\omega}_1$ and $\bar{\omega}_2$.
Conversely,
suppose that $\bar{\varphi}$ is a conformal symplectomorphism
with conformal factor $c$
with respect to $\bar{\omega}_1$ and $\bar{\omega}_2$.
Then
$$
\pi_1^{\ast} (\varphi^{\ast} \omega_2)
= \bar{\varphi}^{\ast} (\pi_2^{\ast} \omega_2)
= \bar{\varphi}^{\ast} (e^{\mu_2} \bar{\omega}_2)
= e^{\mu_2 \circ \bar{\varphi} + c - \mu_1} \, \pi_1^{\ast} \omega_1 \,.
$$
This implies that
$\mu_2 \circ \bar{\varphi} + c - \mu_1 = g \circ \pi_1$
for some $g: M_1 \rightarrow \mathbb{R}$.
We thus have
\[
\pi_1^{\ast} (\varphi^{\ast} \omega_2) = e^{g \circ \pi_1} \, \pi_1^{\ast} \omega_1
= \pi_1^{\ast} (e^g \omega_1) \,,
\]
and so $\varphi^{\ast}\omega_2 = e^g \omega_1$.
This shows that $\varphi$ is a lcs diffeomorphism
with conformal factor $g$
with respect to $(\eta_1, \omega_1)$ and $(\eta_2, \omega_2)$.
\end{proof}

We say that a conformal symplectomorphism $\varphi$
of a conformal symplectic manifold $(M, [\omega])$
is a symplectomorphism
if $\varphi^{\ast}\omega = \omega$
for some (hence any) representative $\omega$.
Proposition \ref{proposition: lift of lcs diffeo to symplectic cover}
implies the following result.

\begin{cor}\label{corollary: symplectic lift}
Let $\varphi$ be a lcs diffeomorphism
of a lcs manifold $\big(M, [(\eta, \omega)]\big)$,
and suppose that $\bar{\varphi}$ is a conformal symplectomorphism
of the symplectic cover $(\bar{M}, [e^{-\mu} \pi^{\ast}\omega])$
such that $\pi \circ \bar{\varphi} = \varphi \circ \pi$.
Then $\bar{\varphi}$ is a symplectomorphism if and only if for some (hence any)
representative $(\eta, \omega)$
and primitive $\mu$ of $\pi^{\ast}\eta$
we have
\[
\mu \circ \bar{\varphi} - \mu = g \circ \pi \,,
\]
where $g$ is the conformal factor of $\varphi$
with respect to $(\eta, \omega)$.
\end{cor}

\begin{proof}
Proposition \ref{proposition: lift of lcs diffeo to symplectic cover}
implies that $\bar{\varphi}$ is a symplectomorphism
if and only if for any representative $(\eta, \omega)$
and primitive $\mu$ of $\pi^{\ast}\eta$
we have $\mu \circ \bar{\varphi} - \mu = g \circ \pi$.
If this relation holds for one primitive $\mu$ of $\pi^{\ast}\eta$
then it also holds for any other primitive $\mu + c$.
Moreover,
if the relation holds for one representative $(\eta, \omega)$
then it also holds for any other representative
$(\eta + d f , e^f \omega)$.
Indeed,
since $\pi^{\ast} (\eta + df) = d (\mu + f \circ \pi)$
and since the conformal factor of $\varphi$
with respect to $(\eta + df, e^f \omega)$
is $f \circ \varphi - f + g$,
the desired relation for $(\eta + df , e^f \omega)$
is
\[
(\mu + f \circ \pi ) \circ \bar{\varphi}
- \mu - f \circ \pi
= (f \circ \varphi - f + g) \circ \pi \,,
\]
which is equivalent to
$\mu \circ \bar{\varphi} - \mu = g \circ \pi$.
\end{proof}

We say that a lcs diffeomorphism $\varphi$
of a lcs manifold $\big(M, [(\eta, \omega)]\big)$
is divergence free
if there is a symplectomorphism $\bar{\varphi}$
of the symplectic cover $(\bar{M} , [e^{-\mu}\pi^{\ast}\omega])$
such that $\pi \circ \bar{\varphi} = \varphi \circ \pi$.
By Lemma \ref{lemma: symplectic cover}
and Corollary \ref{corollary: symplectic lift},
such a symplectomorphism is then unique.
We call it the lift to the symplectic cover
of the divergence free lcs diffeomorphism $\varphi$.
In particular,
the lift of the identity of $M$
is the identity of $\bar{M}$.

Let $\varphi: M_1 \rightarrow M_2$
be a lcs diffeomorphism
between two Liouville lcs manifolds
$\big(M_1, [(\eta_1, \omega_1, \lambda_1)]\big)$
and $\big(M_2, [(\eta_2, \omega_2, \lambda_2)]\big)$.
We say that $\varphi$ is exact
if for some (hence any) representatives
$(\eta_1, \omega_1, \lambda_1)$ and $(\eta_2, \omega_2, \lambda_2)$
the $\eta_1$-closed $1$-form $\lambda_1 - e^{-g} \varphi^{\ast}\lambda_2$,
where $g$ is the conformal factor of $\varphi$
with respect to $(\eta_1, \omega_1)$ and $(\eta_2, \omega_2)$,
is $\eta_1$-exact.
In this case,
we say that $\varphi$ has action $S: M_1 \rightarrow \mathbb{R}$
with respect to $(\eta_1, \omega_1, \lambda_1)$ and $(\eta_2, \omega_2, \lambda_2)$
(or just action $S$
when the representatives we consider are clear from the context)
if
\[
\lambda_1 - e^{-g} \varphi^{\ast}\lambda_2 = d_{\eta_1}S \,.
\]
By Remark \ref{remark: unique solution},
if $M$ is connected and $\eta$ is not exact
then the action $S$ is uniquely defined
(otherwise it is just determined up to an additive constant
on each connected component).

\begin{example}\label{example: exact lcs on conformal}
Let $\varphi: M_1 \rightarrow M_2$
be a conformal symplectomorphism
between two conformal Liouville manifolds
$\big( M_1, [(\omega_1, \lambda_1)] \big)$
and $\big( M_2, [(\omega_2, \lambda_2)] \big)$.
We say that $\varphi$ is exact
if for some (hence any) representatives
$(\omega_1, \lambda_1)$ and $(\omega_2, \lambda_2)$
the closed $1$-form $\lambda_1 - e^{-c} \varphi^{\ast}\lambda_2$,
where $c$ is the conformal factor of $\varphi$
with respect to $\omega_1$ and $\omega_2$,
is exact.
We then say that $\varphi$ has action $S$
with respect to $(\omega_1, \lambda_1)$ and $(\omega_2, \lambda_2)$
if $\lambda_1 - e^{-c} \varphi^{\ast}\lambda_2 = dS$.
Let $\big(M_1, [(\eta_1, \omega_1, \lambda_1)]\big)$
and $\big(M_2, [(\eta_2, \omega_2, \lambda_2)]\big)$
be Liouville lcs manifolds
with exact Lee forms $\eta_j = d\mu_j$,
and consider the associated conformal Liouville manifolds
$\big(M_j, [(e^{-\mu_j}\omega_j, e^{-\mu_j}\lambda_j)]\big)$.
Then $\varphi: M_1 \rightarrow M_2$ is an exact lcs diffeomorphism
with action $S$ with respect to $(\eta_1, \omega_1, \lambda_1)$
and $(\eta_2, \omega_2, \lambda_2)$
if and only if it is an exact conformal symplectomorphism
with action $e^{-\mu_1} \, S$
with respect to $(e^{-\mu_1}\omega_1, e^{-\mu_1} \lambda_1)$
and $(e^{-\mu_2}\omega_2, e^{-\mu_2} \lambda_2)$.
\end{example}

Similarly,
we have the following result.

\begin{prop}\label{proposition: lift of exact isotopy to symplectic cover}
Consider two Liouville lcs manifolds
$\big(M_1 , [(\eta_1, \omega_1, \lambda_1)]\big)$
and $\big(M_2 , [(\eta_2, \omega_2, \lambda_2)]\big)$,
and their symplectic covers
$\big(\bar{M}_1, [(e^{-\mu_1}\pi_1^* \omega_1, e^{-\mu_1}\pi_1^* \lambda_1)]\big)$
and $\big(\bar{M}_2, [(e^{-\mu_2}\pi_2^* \omega_2, e^{-\mu_2}\pi_2^* \lambda_2)]\big)$.
Suppose that $\varphi: M_1 \rightarrow M_2$
and $\bar{\varphi}: \bar{M}_1 \rightarrow \bar{M}_2$
are a lcs diffeomorphism
and a conformal symplectomorphism
such that $\pi_2 \circ \bar{\varphi} = \varphi \circ \pi_1$.
Then $\varphi$ is exact with action $S$
with respect to $(\eta_1, \omega_1, \lambda_1)$
and $(\eta_2, \omega_2, \lambda_2)$
if and only if $\bar{\varphi}$ is exact
with action $e^{-\mu_1} \, (S \circ \pi_1)$
with respect to $(e^{-\mu_1} \, \pi_1^{\ast}\omega_1, e^{-\mu_1} \, \pi_1^{\ast}\lambda_1)$
and $(e^{-\mu_2} \, \pi_2^{\ast}\omega_2, e^{-\mu_2} \, \pi_2^{\ast}\lambda_2)$.
\end{prop}

\begin{proof}
We denote $\bar{\omega}_j = e^{-\mu_j} \, \pi_j^{\ast}\omega_j$
and $\bar{\lambda}_j = e^{-\mu_j} \, \pi_j^{\ast}\lambda_j$.
Let $c$ be the conformal factor of $\bar{\varphi}$
with respect to $\bar{\omega}_1$
and $\bar{\omega}_2$,
and $g$ the conformal factor of $\varphi$
with respect to $(\eta_1, \omega_1)$
and $(\eta_2, \omega_2)$.
By Proposition \ref{proposition: lift of lcs diffeo to symplectic cover}
we have $- \mu_2 \circ \bar{\varphi} + g \circ \pi_1 + \mu_1 = c$,
thus
\[
\bar{\lambda}_1
- e^{-c} \, \bar{\varphi}^{\ast} \bar{\lambda}_2
= e^{-\mu_1} \, \pi_1^{\ast}\lambda_1 
- e^{-c -\mu_2 \circ \bar{\varphi}} \, \pi_1^{\ast} (\varphi^{\ast} \lambda_2)
= e^{-\mu_1} \, \pi_1^{\ast} (\lambda_1 - e^{-g} \varphi^{\ast}\lambda_2 ) \,.
\]
If $\varphi$ is exact with action $S$
with respect to $(\eta_1, \omega_1, \lambda_1)$
and $(\eta_2, \omega_2, \lambda_2)$
then
\[
\bar{\lambda}_1
- e^{-c} \, \bar{\varphi}^{\ast} \bar{\lambda}_2
= e^{-\mu_1} \, \pi_1^{\ast} (d_{\eta_1}S )
= e^{-\mu_1} \, d_{d\mu_1} (S \circ \pi_1)
= d \big(e^{-\mu_1} (S \circ \pi_1)\big) \,,
\]
thus $\bar{\varphi}$ is exact
with action $e^{-\mu_1} (S \circ \pi_1)$
with respect to $(\bar{\omega}_1, \bar{\lambda}_1)$
and $(\bar{\omega}_2, \bar{\lambda}_2)$.
Conversely,
if $\bar{\varphi}$ is exact
with action $e^{-\mu_1} (S \circ \pi_1)$
with respect to $(\bar{\omega}_1, \bar{\lambda}_1)$
and $(\bar{\omega}_2, \bar{\lambda}_2)$
then
\[
e^{-\mu_1} \, \pi_1^{\ast} (d_{\eta_1}S )
= e^{-\mu_1} \, d_{d\mu_1} (S \circ \pi_1)
= d \big(e^{-\mu_1} (S \circ \pi_1)\big)
= \bar{\lambda}_1
- e^{-c} \, \bar{\varphi}^{\ast} \bar{\lambda}_2
= e^{-\mu_1} \, \pi_1^{\ast} (\lambda_1 - e^{-g} \varphi^{\ast}\lambda_2 ) \,.
\]
This implies that 
$\lambda_1 - e^{-g} \varphi^{\ast} \lambda_2 = d_{\eta_1} S$,
thus $\varphi$ is exact
with action $S$ with respect to $(\eta_1, \omega_1, \lambda_1)$
and $(\eta_2, \omega_2, \lambda_2)$.
\end{proof}

\begin{example}\label{example: lift of maps to symplectization}
Let $\phi: Y_1 \rightarrow Y_2$ be a contactomorphism
between two contact manifolds $\big(Y_1, \xi_1 = \ker (\alpha_1)\big)$
and $\big(Y_2, \xi_2 = \ker (\alpha_2)\big)$,
with $\phi^{\ast}\alpha_2 = e^g\alpha_1$.
Then
\[
\widetilde{\phi}: S^1 \times Y_1 \rightarrow S^1 \times Y_2 \,,\;
\widetilde{\phi}(\theta, p) = \big( \theta - g(p) \,,\, \phi(p) \big)  
\]
is an exact lcs diffeomorphism
between the locally conformal symplectizations
$\big( S^1 \times Y_1 \,,\, [(-d\theta_1 , d_{-d\theta_1}\alpha_1 , \alpha_1)] \big)$
and $\big( S^1 \times Y_2 \,,\, [(-d\theta_2 , d_{-d\theta_2}\alpha_2 , \alpha_2)] \big)$,
with conformal factor $g \circ \pr_2$,
where $\pr_2: S^1 \times Y_1 \rightarrow Y_1$
denotes the projection to the second factor,
and action zero
with respect to $(-d\theta_1 , d_{-d\theta_1} \alpha_1 , \alpha_1)$
and $(-d\theta_2 , d_{-d\theta_2} \alpha_2 , \alpha_2)$.
\end{example}

The following lemma,
whose proof is immediate,
will be needed later.
Recall that the $1$-jet bundle of a manifold $B$
is the manifold $J^1B = T^{\ast}B \times \mathbb{R}$.
We endow it with the contact structure
given by the kernel of the contact form $\lambda - dz$,
where $\lambda$ is the tautological $1$-form on $T^{\ast}B$
and $z$ the coordinate on $\mathbb{R}$.

\begin{lemma}\label{lemma: identification twisted cotangent bundle with conformal symplectization}
For any manifold $B$,
the map
\[
T_{-d\theta}^{\ast} (S^1 \times B) \rightarrow S^1 \times J^1B \,,\;
(\theta, q, z, p) \mapsto (\theta, q, p, z)
\]
is a strict and exact lcs diffeomorphism
from the twisted cotangent bundle
$T^{\ast}_{- d\theta}(S^1 \times B)$
to the locally conformal symplectization 
\[
\big( S^1 \times J^1B \,,\,
[(-d\theta \,,\, d_{-d\theta} (\lambda - dz) \,,\, \lambda - dz)] \big)
\]
of the 1-jet bundle $J^1B$,
of action $S(\theta, q, z, p) = z$.
\end{lemma}

\subsection{$\mathcal{C}^0$-rigidity of lcs diffeomorphisms}\label{section: C0 rigidity}

The group of lcs diffeomorphisms
of a lcs manifold $(M, [(\eta, \omega)])$
is $\mathcal C^0$-closed
in the group of diffeomorphisms of $M$.
As explained hereafter,
this is a direct consequence of the Eliashberg--Gromov
$\mathcal C^0$-rigidity of symplectomorphisms \cite{Eliashberg}.
Let us fix an auxiliary Riemannian metric $g$ on $M$.

\begin{prop}\label{proposition:rigidity}
Let $(\varphi_k)_{k \in \N}$ be a sequence of lcs diffeomorphisms
of a lcs manifold $(M, [(\eta, \omega)])$
that converges in the $\mathcal C^0$-topology
to a diffeomorphism $\varphi$ of $M$.
Then $\varphi$ is a lcs diffeomorphism of $(M, [(\eta, \omega)])$. 
\end{prop}

\begin{proof}
Let $p$ be a point of $M$,
and suppose first that $\varphi(p) = p$ and $\varphi_k(p) = p$ for all $k \in \N$.
Let $\mathcal{U}$ be a contractible open neighborhood of $p$ in $M$
whose closure is compact
and whose boundary is a smooth submanifold of $M$.
Set $\mathcal{V}_k = \varphi_k (\mathcal{U})$
and $\mathcal{V} = \varphi (\mathcal{U})$.
Let $\mathcal{W}$ denote a contractible open neighborhood of $p$
that contains $\mathcal{V} \cup_{k} \mathcal{V}_k$
and whose closure is compact as well.
Let $f \in \mathcal{C}^\infty (\mathcal{W})$
be such that $\eta|_{\mathcal{W}} = df$
and set $\overline{\omega} = e^{-f}\omega$.
We may assume without loss of generality
that $\mathcal{W}$ is the domain of a Darboux chart
$(\mathcal{W}, \Psi)$
with $\Psi (\mathcal{W}) = B^{2n}(R)$,
$\Psi (\mathcal{U}) = B^{2n}(r)$
(and $\Psi^*\omega_0 = \overline{\omega}$).
Now, 
we have $\varphi_k^*(\overline{\omega}) = e^{c_k} \overline{\omega}$
for some $c_k \in \R$.
Notice that the $\mathcal{C}^0$-convergence of $(\varphi_k)_{k \in \N}$ towards $\varphi$
implies that the symplectic volume of $\varphi_k (\mathcal{U})$
converges towards the symplectic volume of $\varphi(\mathcal{U})$,
or in other words that 
\begin{equation}\label{equation}
\lim_{k \to \infty} \int_{\varphi_k (\mathcal{U})} \overline{\omega}^n
= \int_{\varphi (\mathcal{U})}\overline{\omega}^n.
\end{equation}
Indeed, for $\varepsilon > 0$,
let $N_\varepsilon \subset \mathcal{W}$ denote
the image of a tubular neighborhood of $\partial \varphi(\mathcal{U})$
such that 
$$
\int_{N_\varepsilon} \overline{\omega}^n < \varepsilon \,.
$$
This exists because $\partial \varphi(\mathcal{U})$
is a closed embedded submanifold of $\mathcal{W}$,
whose Lebesgue measure is $0$.
Besides, there exists $K \in \N$ such that for all $k \geq K$,
$\varphi_k (\partial \mathcal{U}) \subset N_\varepsilon$,
which implies that 
$$
\left|\int_{\varphi_k (\mathcal{U})} \overline{\omega}^n
- \int_{\varphi (\mathcal{U})} \overline{\omega}^n\right|
< \varepsilon \,.
$$
The relation (\ref{equation}) is equivalent to 
$$
\int_{\mathcal{U}} \varphi^*(\overline{\omega}^n)
= \lim_{k\to \infty} \int_{\mathcal{U}} \varphi_k^*(\overline{\omega}^n)
= \lim_{k\to \infty}  e^{nc_k} \int_{\mathcal{U}} \overline{\omega}^n \,.
$$
Hence the sequence $(c_k)_{k \in \N}$ has a limit $c$ in $\R$
(if the group of periods of $\eta$ had rank one, or, more generally, was discrete,
we could argue that there exists a $K \in \N$ such that $c_k = c_K$ for all $k \geq K$,
but this is not necessarily the case).
Set 
$$
\psi_k : \R^{2n} \to \R^{2n} : x \mapsto e^{-\frac12 c_k} x
$$
and 
$$
\psi : \mathcal{W} \to \mathcal{W} : x \mapsto e^{-\frac12 c} x \,.
$$ 
This says that
$\psi_k \circ \varphi_k : (\mathcal{U}, \overline{\omega}|_{\mathcal{U}})
\to (\mathcal{W}, \overline{\omega})$
is a sequence of symplectic embeddings
that $C^0$-converges to the embedding
$\psi \circ \varphi : (\mathcal{U}, \overline{\omega})
\to (\mathcal{W}, \overline{\omega})$.
The $C^0$-rigidity of symplectic embeddings implies that
$\psi \circ \varphi$ is a symplectic embedding,
hence that $\varphi$ is a conformally symplectic embedding.

In the general case,
the idea is to first consider a compactly supported Hamiltonian diffeomorphism
$\psi$ of $\big(M, [(\eta, \omega)]\big)$
such that $\psi(\varphi(p)) = p$
and to replace $\varphi$ with $\psi \circ \varphi$
and each $\varphi_k$ with $\psi \circ \varphi_k$.
We can now assume that $\varphi$ fixes $p$.
Next, we replace each $\varphi_k$ by $\psi_k \circ \varphi_k$,
where $\psi_k$ is a compactly supported Hamiltonian diffeomorphisms
of $\big(M, [(\eta, \omega)]\big)$
such that $\psi_k (\varphi_k(p)) = p$
and where the sequence $(\psi_k)_{k \in \N}$
$\mathcal C^0$-converges to the identity.
In order to prove existence of $\psi$
and of the sequence $(\psi_k)$,
we argue that the group $\Ham^c \big(M, [(\eta, \omega)]\big)$
acts transitively on $M$ and, moreover,
that it is possible to choose,
for any two points $p$ and $q$ in $M$,
a diffeomorphism $\psi \in \Ham^c \big(M, [(\eta, \omega)]\big)$
with $\psi(p) = q$ and whose $\mathcal C^0$-norm is bounded above by $2d(p,q)$. 
Indeed, let us pick a simple path $\gamma : [0,1] \to M$
starting at $p$ and ending at $q$
and whose length is bounded above by $\frac32 d(p,q)$,
and a contractible open neighborhood $\mathcal{U}$ of $\gamma([0,1])$.
On $\mathcal{U}$, the Lee form is exact
and we can construct a compactly supported Hamiltonian vector field
of the symplectic manifold $(\mathcal{U}, \overline{\omega})$,
one of whose trajectories is $\gamma$,
and whose norm is bounded above by $2 \sup\{||\dot{\gamma}(t)||\}$.
The time-one map of its flow extends to a compactly supported lcs Hamiltonian diffeomorphism
of $(M, [(\eta, \omega)])$ which maps $p$ to $q$
and whose $\mathcal C^0$-norm is bounded above by $2d(p,q)$. 
\end{proof}

\subsection{Lcs vector fields and isotopies}

A locally conformal symplectic (lcs) vector field
on a lcs manifold $\big(M, [(\eta, \omega)]\big)$
is a vector field $X$ such that for any representative $(\eta, \omega)$
there is a function $h: M \rightarrow \mathbb{R}$,
called the conformal factor of $X$ with respect to $(\eta, \omega)$,
satisfying $\mathcal{L}_X \omega = h \, \omega$.
This relation implies that
$\mathcal{L}_X \eta = dh$.

\begin{lemma}\label{lemma: conformal vector field 1}
A vector field $X$ on a lcs manifold $\big(M, [(\eta, \omega)]\big)$
is lcs if and only if there is a real number $c$
such that
$$
d_{\eta} (\iota_X\omega) = c \, \omega
$$
for some (hence any) representative $(\eta, \omega)$.
In this case,
the conformal factor of $X$ with respect to $(\eta, \omega)$
is $h = c + \eta(X)$.
\end{lemma}

\begin{proof}
Suppose that $X$ is lcs,
thus $\mathcal{L}_X \omega = h \, \omega$ and $\mathcal{L}_X \eta = dh$
for some $h: M \rightarrow \mathbb{R}$.
Then
\[
d_{\eta} (\iota_X\omega) = \mathcal{L}_X \omega - \iota_X (d_{\eta}\omega) - \eta(X) \, \omega
= \big( h - \eta(X) \big) \, \omega
\]
and
\[
dh = \mathcal{L}_X \eta = d (\iota_X\eta) + \iota_X d\eta = d (\iota_X\eta) \,.
\]
Thus $h - \eta(X) = c$ for some $c \in \mathbb{R}$,
and $d_{\eta} (\iota_X\omega) = c \, \omega$.
Conversely,
if $d_{\eta} (\iota_X\omega) = c \, \omega$
for some $c \in \mathbb{R}$
then
\[
\mathcal{L}_X \omega = d_{\eta} (\iota_X\omega) + \iota_X (d_{\eta}\omega) + \eta(X) \, \omega
= \big( c + \eta(X) \big) \, \omega \,,
\]
thus $X$ is lcs
with conformal factor $h = c + \eta(X)$.
\end{proof}

A locally conformal symplectic (lcs) isotopy
on a lcs manifold $\big(M, [(\eta, \omega)]\big)$
is an isotopy $\{\varphi_t\}$ such that
each $\varphi_t$ is a lcs diffeomorphism.
We say that a time dependent vector field $X_t$ is lcs
if each $X_t$ is lcs.

\begin{lemma}\label{lemma: conformal vector field}
Let $X_t$ be a time dependent vector field
on a lcs manifold $\big(M, [(\eta, \omega)]\big)$,
with flow $\{\varphi_t\}$.
Then $X_t$ is a lcs vector field
if and only if $\{\varphi_t\}$ is a lcs isotopy.
In this case, if $g_t$ and $h_t$ are the conformal factors of $\varphi_t$ and $X_t$
with respect to $(\eta, \omega)$
then
\[
h_t = \frac{dg_t}{dt} \circ \varphi_t^{-1} \,.
\]
\end{lemma}

\begin{proof}
Suppose that $\{\varphi_t\}$ is a lcs isotopy
with conformal factors $g_t$.
Then
\[
\varphi_t^{\ast} (\mathcal{L}_{X_t} \omega)
= \frac{d}{dt} \, \varphi_t^{\ast} \omega
= \frac{d}{dt} \, (e^{g_t} \omega)
= \frac{dg_t}{dt} \, e^{g_t} \omega
= \frac{dg_t}{dt} \, \varphi_t^{\ast} \omega \,,
\]
thus
\[
\mathcal{L}_{X_t} \omega
= \Big( \frac{dg_t}{dt} \circ \varphi_t^{-1} \Big) \, \omega
\]
and so $X_t$ is a lcs vector field
with conformal factors $h_t = \frac{dg_t}{dt} \circ \varphi_t^{-1}$.
Conversely, suppose that $X_t$ is a lcs vector field
with conformal factors $h_t$.
Set $g_t = \int_0^t h_s \circ \varphi_s \, ds$.
Then 
\[
\frac{d}{dt} \, \varphi_t^{\ast} \omega
= \varphi_t^{\ast} (\mathcal{L}_{X_t} \omega)
= \varphi_t^{\ast} (h_t \, \omega)
= (h_t \circ \varphi_t) \, \varphi_t^{\ast} \omega
= \frac{dg_t}{dt} \, \varphi_t^{\ast} \omega \,.
\]
This implies that $\varphi_t^{\ast} \omega = e^{g_t} \omega$,
thus $\{\varphi_t\}$ is a lcs isotopy with conformal factors $g_t$.
\end{proof}

We say that a lcs vector field $X_t$
on a lcs manifold $\big(M, [(\eta, \omega)]\big)$
is divergence free if $d_{\eta} (\iota_{X_t} \omega) = 0$ for all $t$,
for some (hence any) representative $(\eta, \omega)$.
We say that a lcs isotopy $\{\varphi_t\}$
is divergence free if each $\varphi_t$ is divergence free.

\begin{lemma}\label{lemma: lift of isotopy to symplectic cover}
Let $X_t$ be a lcs vector field
on a lcs manifold $\big(M, [(\eta, \omega)]\big)$,
and $\{\varphi_t\}$ the lcs isotopy generated by $X_t$.
Then $X_t$ is divergence free if and only if
$\{\varphi_t\}$ is divergence free.
\end{lemma}

\begin{proof}
Let $\{\bar{\varphi}_t\}$ be the lift of $\{\varphi_t\}$
to the symplectic cover $(\bar{M}, [e^{-\mu} \pi^{\ast}\eta])$
starting at the identity,
and $g_t$ the conformal factors of $\{\varphi_t\}$.
Then
\[
d (\mu \circ \bar{\varphi}_t)
= \bar{\varphi}_t^{\ast} \, \pi^{\ast} \eta
= \pi^{\ast} \varphi_t^{\ast} \eta
= \pi^{\ast} (\eta + dg_t)
= d (\mu + g_t \circ \pi) \,,
\]
thus
\[
\mu \circ \bar{\varphi}_t = \mu + g_t \circ \pi + a_t
\]
for a $1$-parameter family of real numbers $a_t$.
Denote by $\bar{X}_t$
the vector field generating $\{\bar{\varphi}_t\}$.
Using Lemmas \ref{lemma: conformal vector field 1}
and \ref{lemma: conformal vector field}
we obtain
\[
\frac{da_t}{dt}
= \frac{d}{dt} \, \big( \mu \circ \bar{\varphi}_t - g_t \circ \pi \big)
= d\mu \, (\bar{X}_t \circ \bar{\varphi}_t) - \frac{dg_t}{dt} \circ \pi
\]
\[
= \Big( \eta (X_t \circ \varphi_t) - \frac{dg_t}{dt} \Big) \circ \pi
= \Big(h_t \circ \varphi_t - c_t - \frac{dg_t}{dt} \Big) \circ \pi
= -c_t \,,
\]
where $c_t$ is the $1$-parameter family of real numbers
such that $d_{\eta} (\iota_{X_t}\omega) = c_t \, \omega$.
The vector field $X_t$ is divergence free
if and only if $c_t = 0$ for all $t$,
thus if and only if $\frac{d a_t}{dt} = 0$ for all $t$.
Since $\bar{\varphi}_0$ is the identity,
we have $a_0 = 0$.
We thus conclude that $X_t$ is divergence free
if and only if $a_t = 0$ for all $t$,
and so
\[
\mu \circ \bar{\varphi}_t - \mu = g_t \circ \pi
\]
for all $t$.
By Corollary \ref{corollary: symplectic lift},
this happens if and only if
$\{\bar{\varphi}_t\}$ is a symplectic isotopy,
and so if and only if $\{\varphi_t\}$ is divergence free.
\end{proof}

\begin{example}\label{example: lift of isotopies to symplectization}
If $X_t$ is a contact vector field
on a contact manifold $\big(Y, \xi = \ker (\alpha)\big)$
with $\mathcal{L}_{X_t} \alpha = h_t \, \alpha$
then
\[
\widetilde{X_t} (\theta, p) = X_t(p) - h_t(p) \, \frac{\partial}{\partial \theta}
\]
is a divergence free lcs vector field
on the locally conformal symplectization
$\big(S^1 \times Y \,,\, [(-d\theta \,,\, d_{-d\theta} \alpha)] \big)$.
The isotopy generated by $\widetilde{X_t}$
is the lift,
as defined in Example \ref{example: lift of maps to symplectization},
of the contact isotopy generated by $X_t$.
\end{example}

\subsection{Lcs Hamiltonian isotopies}

A divergence free lcs vector field $X_t$ on a lcs manifold
$\big(M, [(\eta, \omega)]\big)$
is said to be Hamiltonian
if for some (hence any) representative $(\eta, \omega)$
the $\eta$-closed $1$-form $\iota_{X_t}\omega$ is $\eta$-exact.
We then say that $X_t$ has Hamiltonian function $H_t: M \rightarrow \mathbb{R}$
with respect to $(\eta, \omega)$ if
$$
\iota_{X_t}\omega = - d_{\eta} H_t \,.
$$
By Remark \ref{remark: unique solution},
if $M$ is connected and $\eta$ is not exact
then such $H_t$ is unique
(otherwise it is just determined
up to an additive constant
on each connected component).
By Lemma \ref{lemma: lift of isotopy to symplectic cover},
the flow of $X_t$ is a divergence free lcs isotopy.
We say that it is a lcs Hamiltonian isotopy,
with Hamiltonian function $H_t$ with respect to $(\eta, \omega)$.
A lcs diffeomorphism
is said to be Hamiltonian
if it is the time-1 map of a lcs Hamiltonian isotopy.

\begin{example}\label{example: Hamiltonian conformal}
We say that a symplectic isotopy $\{\varphi_t\}$
of a conformal symplectic manifold $(M, [\omega])$
is Hamiltonian
if the closed $1$-form $\iota_{X_t}\omega$ is exact
for some (hence any) representative $\omega$,
where $X_t$ is the vector field generating $\{\varphi_t\}$.
We then say that $\{\varphi_t\}$
has Hamiltonian function $H_t$ with respect to $\omega$
if $\iota_{X_t}\omega = -dH_t$.
Identifying $(M, [\omega])$
with the lcs manifold $\big(M , [(0, \omega)]\big)$,
$\{\varphi_t\}$ is a lcs Hamiltonian isotopy
with Hamiltonian function $H_t$ with respect to $(0, \omega)$
if and only if it is a Hamiltonian isotopy
with Hamiltonian function $H_t$ with respect to $\omega$.
\end{example}

The description of Example \ref{example: Hamiltonian conformal}
always applies locally.
Let $\big( M , [(\eta, \omega)] \big)$ be a lcs manifold,
and $\{\mathcal{U}_i\}$ an open cover
such that $\left. \omega \right\lvert_{\mathcal{U}_i} = e^{\mu_i} \, \omega_i$
with $\omega_i$ symplectic.
If $X_t$ is a lcs Hamiltonian vector field
with Hamiltonian function $H_t$
with respect $(\eta, \omega)$
then $\left. X \right\lvert_{\mathcal{U}_i}$
is a Hamiltonian vector field on $(\mathcal{U}_i, \omega_i)$
with Hamiltonian function $e^{-\mu_i}H_t$.
Conversely,
given a time dependent function $H_t: M \rightarrow \mathbb{R}$
we can define the associated lcs vector field as follows.
Consider on every $\mathcal{U}_i$
the time dependent function
$H_t^{(i)} = e^{-\mu_i} \left. H_t \right\lvert_{\mathcal{U}_i}$,
and let $X_t^{(i)}$ be the associated Hamiltonian vector field.
Since the equation $\omega_i (X_t^{(i)} ,\,\cdot\,) = - dH_t^{(i)}$
is equivalent to
$\omega (X_t^{(i)} ,\,\cdot\,) = - d_{\eta} \left. H_t \right\lvert_{\mathcal{U}_i}$,
the vector fields $X_t^{(i)}$ glue to a global vector field $X_t$,
satisfying $\omega (X_t ,\,\cdot\,) = - d_{\eta} H_t$.

\begin{example}\label{Example:Lift of contact Hamiltonians}
Recall that for any contact isotopy $\{\phi_t\}$
of a contact manifold $\big(Y, \xi = \ker (\alpha)\big)$
there is a unique time dependent function $H_t$,
called the Hamiltonian function of $\{\phi_t\}$
with respect to $\alpha$,
such that $\alpha(X_t) = H_t$
and $\iota_{X_t} d\alpha = dH_t (R_{\alpha}) \alpha - dH_t$, where $X_t$ is the the vector field generating the isotopy.
Recall also that $\mathcal{L}_{X_t}\alpha = h_t \alpha$
for $h_t = dH_t (R_{\alpha})$.
The lift $\{\widetilde{\phi_t}\}$ of $\{\phi_t\}$
to the locally conformal symplectization
$\big( S^1 \times Y \,,\, [(-d\theta \,,\, d_{-d\theta}\alpha )]\big)$
is a lcs Hamiltonian isotopy
with Hamiltonian function
\[
\widetilde{H_t}: S^1 \times Y \rightarrow \mathbb{R} \,,\;
\widetilde{H_t} (\theta, p) = H_t (p)
\]
with respect to $(-d\theta \,,\, d_{-d\theta}\alpha)$.
Indeed,
by Example \ref{example: lift of isotopies to symplectization}
the lcs isotopy $\{\widetilde{\phi_t}\}$ is generated by the vector field
$\widetilde{X_t} = X_t - h_t \frac{\partial}{\partial \theta}$.
Thus
\[
\iota_{\widetilde{X_t}} (d_{-d\theta}\alpha)
= \iota_{X_t}d\alpha - h_t\alpha - \alpha(X_t)d\theta
= dH_t (R_{\alpha}) \alpha - dH_t - h_t\alpha - H_t d\theta
= - d_{-d\theta} \widetilde{H_t} \,.
\]
\end{example}

We denote by $\Ham \big(M, [(\eta, \omega)]\big)$
the space of lcs Hamiltonian diffeomorphisms
of $\big(M, [(\eta, \omega)]\big)$.
By the following lemma, whose proof is left to the reader,
$\Ham \big(M, [(\eta, \omega)]\big)$ is a group.

\begin{lemma}\label{lemma: composition law Hamiltonians}
Let $\{\varphi_t\}$, $\{\varphi_t^{(1)}\}$ and $\{\varphi_t^{(2)}\}$
be lcs Hamiltonian isotopies on a lcs manifold
$\big( M , [(\eta, \omega)]\big)$,
with Hamiltonian functions $H_t$, $H_t^{(1)}$ and $H_t^{(2)}$
and conformal factors $g_t, g_t^{(1)}$ and $g_t^{(2)}$
with respect to $(\eta, \omega)$.
Then
\renewcommand{\theenumi}{\roman{enumi}}
\begin{enumerate}
\item $\{\varphi_t^{-1}\}$ is a lcs Hamiltonian isotopy
with Hamiltonian function $- e^{-g_t} (H_t \circ \varphi_t)$
with respect to $(\eta, \omega)$;
\item $\{ \varphi_t^{(1)} \circ \, \varphi_t^{(2)} \}$
is a lcs Hamiltonian isotopy
with Hamiltonian function
$$
H_t^{(1)} + e^{g_t^{(1)} \circ (\varphi_t^{(1)})^{-1}} \,
H_t^{(2)} \circ (\varphi_t^{(1)})^{-1}
$$
with respect to $(\eta, \omega)$.
\end{enumerate}
\end{lemma}

We will also need the following lemma,
whose proof is again left to the reader.

\begin{lemma}\label{lemma: Hamiltonian conjugation}
Let $\{\varphi_t\}$ be a lcs Hamiltonian isotopy
of a lcs manifold $\big(M, [(\eta, \omega)]\big)$
with Hamiltonian function $H_t$ with respect to $(\eta, \omega)$,
and $\psi: M' \rightarrow M$ a lcs diffeomorphism
from another lcs manifold $\big(M', [(\eta', \omega')]\big)$
to $\big(M, [(\eta, \omega)]\big)$.
Then $\{ \psi^{-1} \circ \varphi_t \circ \psi \}$
is a lcs Hamiltonian isotopy of $\big(M', [(\eta', \omega')]\big)$
with Hamiltonian function $e^{-f} (H_t \circ \psi)$
with respect to $(\eta', \omega')$,
where $f$ is the conformal factor of $\psi$
with respect to $(\eta', \omega')$ and $(\eta, \omega)$.
\end{lemma}

Example \ref{example: Hamiltonian conformal}
can be reformulated as follows.
Let $\big( M, [(\eta, \omega)] \big)$
be a lcs manifold with exact Lee form $\eta = d\mu$,
and consider the associated conformal symplectic manifold
$\big( M, [ e^{-\mu} \omega ] \big)$.
Then $\{\varphi_t\}$ is a lcs Hamiltonian isotopy
with Hamiltonian function $H_t$ with respect to $(\eta, \omega)$
if and only if it is a Hamiltonian isotopy
with Hamiltonian function $e^{-\mu}H_t$
with respect to $e^{-\mu}\omega$.
Similarly,
we have the following result.

\begin{prop}\label{proposition: lift of Hamiltonian isotopy to symplectic cover}
Let $\{\varphi_t\}$ be a divergence free lcs isotopy
of a lcs manifold $\big(M, [(\eta, \omega)]\big)$
starting at the identity,
and $\{ \bar{\varphi}_t \}$ its lift
to the symplectic cover $(\bar{M}, [e^{-\mu} \pi^{\ast}\omega])$.
Then $\{\varphi_t\}$ is a lcs Hamiltonian isotopy
if and only if $\{ \bar{\varphi}_t \}$ is a Hamiltonian isotopy.
In this case,
the Hamiltonian function of $\{ \bar{\varphi}_t \}$
with respect to $e^{-\mu} \pi^{\ast}\omega$
is $e^{-\mu} (H_t \circ \pi)$,
where $H_t$ is the Hamiltonian function of $\{\varphi_t\}$
with respect to $(\eta, \omega)$.
\end{prop}

\begin{proof}
Let $\bar{X}_t$ be the vector field generating $\{ \bar{\varphi}_t \}$.
Then $\pi_{\ast} (\bar{X}_t) = X_t$,
where $X_t$ is the vector field generating $\{\varphi_t\}$.
Suppose that $\{\varphi_t\}$ is a lcs Hamiltonian isotopy
with Hamiltonian function $H_t$
with respect to $(\eta, \omega)$.
For every vector field $Y$ on $\bar{M}$
we then have
\[
e^{-\mu} \pi^{\ast}\omega (\bar{X}_t, Y)
= e^{-\mu} \, \omega \big( X_t \,,\, \pi_{\ast}(Y) \big)
= - e^{-\mu} \, d_{\eta} H_t \big(\pi_{\ast}(Y)\big)
\]
\[
= - e^{-\mu} \, d_{d\mu} (H_t \circ \pi) (Y)
= - d \big( e^{-\mu} (H_t \circ \pi ) \big) (Y) \,,
\]
thus $\{ \bar{\varphi}_t \}$ is a Hamiltonian isotopy
with Hamiltonian function $e^{-\mu} (H_t \circ \pi)$
with respect to $e^{-\mu} \pi^{\ast}\omega$.
Conversely,
suppose that $\{ \bar{\varphi}_t \}$ is a Hamiltonian isotopy
with Hamiltonian function $\bar{H}_t$
with respect to $e^{-\mu} \pi^{\ast}\omega$.
Then
\[
e^{-\mu} \, \omega \big(X_t, \pi_{\ast}(Y)\big)
= e^{-\mu} \, \pi^{\ast} \omega (\bar{X}_t, Y)
= - d \bar{H}_t (Y)
\]
for every vector field $Y$ on $\bar{M}$,
thus
\[
\pi^{\ast} (\iota_{X_t}\omega) = - e^{\mu} \, d \bar{H}_t
= - d_{d\mu} (e^{\mu} \bar{H}_t)
= - d_{\pi^{\ast}\eta} (e^{\mu} \bar{H}_t) \,.
\]
This implies that $\bar{H}_t = e^{-\mu} (H_t \circ \pi)$
for some $H_t: M \rightarrow \mathbb{R}$
and $\iota_{X_t} \omega = - d_{\eta}H_t$,
thus $\{\varphi_t\}$ is a lcs Hamiltonian isotopy
with Hamiltonian function $H_t$
with respect to $(\eta, \omega)$.
\end{proof}

Proposition \ref{proposition: lift of Hamiltonian isotopy to symplectic cover}
allows to derive the following result
from the corresponding one in the symplectic case,
which is due to Banyaga \cite{Banyaga}
(see also \cite[Proposition 10.2.12]{McDS}).

\begin{thm}\label{theorem: Banyaga}
Let $\{\varphi_t\}$ be a lcs isotopy
on a lcs manifold $\big( M , [(\eta, \omega)] \big)$.
Suppose that each $\varphi_t$ is a lcs Hamiltonian diffeomorphism.
Then $\{\varphi_t\}$ is a lcs Hamiltonian isotopy.
\end{thm}

\begin{proof}
Since each $\varphi_t$ is Hamiltonian,
hence divergence free,
$\{\varphi_t\}$ is a divergence free lcs isotopy.
By Proposition \ref{proposition: lift of Hamiltonian isotopy to symplectic cover},
its lift $\{\bar{\varphi}_t\}$ to the symplectic cover
is a symplectic isotopy such that each $\bar{\varphi}_t$ is Hamiltonian.
Thus $\{\bar{\varphi}_t\}$ is a Hamiltonian isotopy.
By Proposition \ref{proposition: lift of Hamiltonian isotopy to symplectic cover},
$\{\varphi_t\}$ is thus a lcs Hamiltonian isotopy.
\end{proof}

We end this subsection by showing that
lcs Hamiltonian diffeomorphisms
on Liouville lcs manifolds
are exact.

\begin{prop}\label{proposition: Hamiltonian vs exact}
Let $\{\varphi_t\}$ be a divergence free lcs isotopy
of a Liouville lcs manifold $\big(M, [(\eta, \omega, \lambda)]\big)$.
Then $\{\varphi_t\}$ is a lcs Hamiltonian isotopy
if and only if $\varphi_t$ is exact for every $t$.
In this case,
if $H_t$ and $g_t$ are the Hamiltonian function
and the conformal factors of $\{\varphi_t\}$
with respect to $(\eta, \omega)$,
then the action of $\{\varphi_t\}$
with respect to $(\eta, \omega, \lambda)$
is given by
\begin{equation}\label{equation: Hamiltonian function vs action}
S_t = \int_0^t e^{-g_s} \Big( \big(H_s - \lambda (X_s) \big) \circ \varphi_s \Big) \, ds \,,
\end{equation}
where $X_t$ is the vector field generating $\{\varphi_t\}$.
\end{prop}

\begin{proof}
By Lemma \ref{lemma: lift of isotopy to symplectic cover},
$X_t$ is divergence free
and so $d_{\eta} (\iota_{X_t}\omega) = 0$ for all $t$.
By Lemmas \ref{lemma: conformal vector field 1}
and \ref{lemma: conformal vector field}
we thus have
$$
\eta (X_t) = \frac{dg_t}{dt} \circ \varphi_t^{-1} \,,
$$
and so
\[
\frac{d}{dt} \, \big(\lambda - e^{-g_t} \, \varphi_t^{\ast}\lambda \big)
= \frac{dg_t}{dt} \, e^{-g_t} \, \varphi_t^{\ast}\lambda - e^{-g_t} \, \frac{d}{dt} \, \varphi_t^{\ast}\lambda
= \varphi_t^{\ast} \, \Big( \Big(\frac{dg_t}{dt} \circ \varphi_t^{-1}\Big) \,
e^{-g_t \circ \varphi_t^{-1}} \lambda
- e^{-g_t \circ \varphi_t^{-1}} \mathcal{L}_{X_t}\lambda \Big)
\]
\[
= \varphi_t^{\ast} \, \Big( \eta (X_t) \, e^{-g_t \circ \varphi_t^{-1}} \lambda
- e^{-g_t \circ \varphi_t^{-1}}
\big( d_{\eta} (\iota_{X_t}\lambda) + \iota_{X_t}d_{\eta}\lambda + \eta(X_t) \lambda \big) \Big)
\]
\[
= \varphi_t^{\ast} \, \Big( - e^{-g_t \circ \varphi_t^{-1}}
\big( d_{\eta} (\iota_{X_t}\lambda) + \iota_{X_t}\omega \big) \Big)
\]
\[
= - e^{-g_t} \,
\varphi_t^{\ast} \big( d_{\eta} (\iota_{X_t}\lambda) + \iota_{X_t}\omega \big) \,.
\]
Suppose that $\{\varphi_t\}$ is a lcs Hamiltonian isotopy
with Hamiltonian function $H_t$ with respect to $(\eta, \omega)$,
and define $S_t$ by \eqref{equation: Hamiltonian function vs action}.
Then
\[
\frac{d}{dt} \, \big(\lambda - e^{-g_t} \, \varphi_t^{\ast}\lambda \big)
= - e^{-g_t} \, \varphi_t^{\ast} \big( d_{\eta} (\iota_{X_t}\lambda - H_t) \big)
= - e^{-g_t} \, d_{\varphi_t^{\ast}\eta} \big((\iota_{X_t}\lambda - H_t) \circ \varphi_t \big)
\]
\[
= - e^{-g_t} \, d_{\eta + dg_t} \big((\iota_{X_t}\lambda - H_t) \circ \varphi_t \big)
= d_{\eta} \Big( e^{-g_t} \big((H_t - \iota_{X_t}\lambda) \circ \varphi_t \big) \Big)
\]
\[
= d_{\eta} \, \frac{dS_t}{dt}
= \frac{d}{dt} \, d_{\eta}S_t \,.
\]
Since $\lambda - e^{-g_0} \, \varphi_0^{\ast}\lambda = 0$
and $d_{\eta}S_0 = 0$,
this implies that
$\lambda - e^{-g_t} \, \varphi_t^{\ast}\lambda = d_{\eta}S_t$,
and so each $\varphi_t$ is exact with action $S_t$
with respect to $(\eta, \omega, \lambda)$.
Conversely,
suppose that each $\varphi_t$ is exact with action $S_t$
with respect to $(\eta, \omega, \lambda)$,
and define $H_t$ by \eqref{equation: Hamiltonian function vs action}.
Then
\[
d_{\eta} \Big( e^{-g_t} \big( \big(H_t - \lambda (X_t)\big) \circ \varphi_t \big) \Big)
= d_{\eta} \, \frac{dS_t}{dt}
= \frac{d}{dt} \, d_{\eta}S_t
= \frac{d}{dt} \, \big(\lambda - e^{-g_t} \, \varphi_t^{\ast}\lambda \big)
= - e^{-g_t} \,
\varphi_t^{\ast} \big( d_{\eta} (\iota_{X_t}\lambda) + \iota_{X_t}\omega \big) \,,
\]
thus
\[
\varphi_t^{\ast} \big( d_{\eta} (\iota_{X_t}\lambda) + \iota_{X_t}\omega \big)
= - e^{g_t} \, d_{\eta} \Big(e^{-g_t} \big((H_t - \iota_{X_t}\lambda) \circ \varphi_t \big)\Big)
= - d_{\eta + dg_t} \big((H_t - \iota_{X_t}\lambda) \circ \varphi_t \big)
\]
\[
= - d_{\varphi_t^{\ast}\eta} \big((H_t - \iota_{X_t}\lambda) \circ \varphi_t \big)
= \varphi_t^{\ast} \, d_{\eta} (\iota_{X_t}\lambda - H_t) \,.
\]
This implies that $\iota_{X_t}\omega = - d_{\eta}H_t$,
thus $\{\varphi_t\}$ is a lcs Hamiltonian isotopy
with Hamiltonian function $H_t$ with respect to $(\eta, \omega)$.
Alternatively,
we can apply the corresponding result in symplectic topology
\cite[Proposition 9.3.1]{McDS}
to the lift of $\{\varphi_t\}$
to the symplectic cover,
and use Corollary \ref{corollary: symplectic lift},
Proposition \ref{proposition: lift of exact isotopy to symplectic cover}
and Proposition \ref{proposition: lift of Hamiltonian isotopy to symplectic cover}.
\end{proof}

\subsection{Lee flow}

The Lee vector field of a lcs manifold $\big(M, [(\eta, \omega)]\big)$
with respect to $(\eta, \omega)$
is the vector field $R_{\omega}$ that satisfies
$$
\omega (R_{\omega} \,,\, \cdot\,) = \eta \,,
$$
in other words the lcs Hamiltonian vector field
with Hamiltonian function $H_t \equiv 1$
with respect to $(\eta, \omega)$.
In particular,
$R_{\omega}$ is divergence free.
By Lemma \ref{lemma: conformal vector field 1}
it thus has conformal factor
$h = \eta (R_{\omega}) = 0$ with respect to $(\eta, \omega)$,
and so it satisfies
$\mathcal{L}_{R_{\omega}} \omega = 0$
and $\mathcal{L}_{R_{\omega}} \eta = 0$.
The Lee flow $\{\varphi_t^{\omega}\}$ of $(\eta, \omega)$
is the flow of $R_{\omega}$.
It satisfies
$(\varphi_t^{\omega})^{\ast}\omega = \omega$
and $(\varphi_t^{\omega})^{\ast}\eta = \eta$.

\begin{example}\label{example: Lee flow symplectic}
The Lee flow of a conformal symplectic manifold
$(M, [\omega]) = \big(M , [(0, \omega)]\big)$
with respect to $(0, \omega)$ is the identity.
More generally,
the Lee flow with respect to $(d\mu, e^{\mu}\omega)$
is the Hamiltonian flow of $e^{-\mu}$
with respect to the symplectic form $\omega$.
In particular,
the Lee flow with respect to $(d\mu, e^{\mu}\omega)$
preserves $\mu$.
\end{example}

Similarly,
we have the following result,
whose proof is immediate
using Example \ref{example: Lee flow symplectic}.

\begin{lemma}\label{lemma: Lee flow symplectic cover}
Let $\big( M , [(\eta, \omega)] \big)$ be a lcs manifold,
and $(\bar{M}, [e^{-\mu} \pi^{\ast}\omega])$
its symplectic cover.
The Lee flow on $\bar{M}$ with respect to $(\pi^*\eta, \pi^*\omega)$
is the Hamiltonian flow of $e^{-\mu}$
with respect to $e^{-\mu}\pi^{\ast}\omega$.
It preserves $\mu$ and commutes with the projection $\pi: \bar{M} \rightarrow M$.
The induced isotopy on $M$
is the Lee flow with respect to $(\eta, \omega)$.
\end{lemma}

\begin{example}
The Lee vector field of the locally conformal symplectization
$\big( S^1 \times Y \,,\, [(- d\theta \,,\, d_{-d\theta}\alpha)]\big)$
of a contact manifold $\big(Y, \xi = \ker (\alpha)\big)$
with respect to $(- d\theta \,,\, d_{-d\theta}\alpha)$
is $(0, R_{\alpha})$,
where $R_{\alpha}$ is the Reeb vector field
with respect to $\alpha$.
\end{example}

\begin{example}\label{example: Lee flow on twisted cotangent bundle}
The Lee vector field of the twisted cotangent bundle $T^{\ast}_{\beta}B$
with respect to $(\pi^{\ast}\beta \,,\, d_{\pi^{\ast}\beta} \lambda)$
is the vector field that is dual to $\pi^{\ast}\beta$
with respect to $d_{\pi^{\ast}\beta} \lambda$.
The Lee flow $\{\varphi_t\}$ is given by
$$
\varphi_t (\sigma_q) = \sigma_q + t \beta (q) \,.
$$
\end{example}

\subsection{Lagrangian submanifolds}

An immersion $i: L \rightarrow M$
into a lcs manifold $\big(M, [(\eta, \omega)] \big)$
is said to be Lagrangian
if the dimension of $L$ is half the dimension of $M$
and $i^{\ast} \omega = 0$
for some (hence any) representative $(\eta, \omega)$.
A submanifold $L$ of $M$ is said to be Lagrangian
if the inclusion $L \hookrightarrow M$ is Lagrangian.
If $i: L \rightarrow M$ is a Lagrangian immersion
into a Liouville lcs manifold $\big(M, [(\eta, \omega, \lambda)] \big)$
then for any representative $(\eta, \omega, \lambda)$
the $1$-form $i^{\ast}\lambda$ is $i^{\ast}\eta$-closed.
We say that $i: L \rightarrow M$ is an exact Lagrangian immersion
if for some (hence any) representative $(\eta, \omega, \lambda)$
the $1$-form $i^{\ast}\lambda$ is $i^{\ast}\eta$-exact.
We then say that the exact Lagrangian immersion $i: L \rightarrow M$
has action $S: L \rightarrow \mathbb{R}$ with respect to $(\eta, \omega, \lambda)$
(or just action $S$ when the representative we consider
is clear from the context)
if $i^{\ast}\lambda = d_{i^{\ast}\eta} S$.
By Remark \ref{remark: unique solution},
if $L$ is connected and $i^{\ast} \eta$ is not exact
then the action $S$ is unique
(otherwise it is just determined up to an additive constant
on each connected component).
A Lagrangian submanifold $L$ of $M$ is said to be exact
and have action $S$
if the inclusion $L \hookrightarrow M$ is exact
and has action $S$.

\begin{example}\label{example: Lagrangians in twisted cotangent bundle}
For any function $f: B \rightarrow \mathbb{R}$,
the $\beta$-twisted differential
\[
d_{\beta}f: B \rightarrow T^{\ast}B \,,\;
d_{\beta}f(q)
= df(q) - f(q) \beta(q)
\]
is an exact Lagrangian embedding
into the twisted cotangent bundle $T_{\beta}^{\ast}B$,
of action $f$ with respect to
$(\pi^{\ast}\beta \,,\, d_{\pi^{\ast}\beta} \lambda \,,\, \lambda )$.
Indeed,
by the tautological property of $\lambda$
we have
\[
(d_{\beta} f)^{\ast} \lambda = d_{\beta} f \,.
\]
We denote by $L_{f, \beta}$ the image of $d_{f, \beta}$.
It is an exact Lagrangian submanifold of $T^{\ast}_{\beta}B$
of action
\[
S_{f, \beta}: L_{f, \beta} \rightarrow \mathbb{R} \,,\;
S_{f, \beta} \big(d_{\beta}f(q)\big) = f(q) \,.
\]
In particular,
the zero section $L_0 = L_{0, \beta}$ of $T_{\beta}^{\ast}B$
is an exact Lagrangian submanifold of action zero.
Any exact Lagrangian submanifold of $T^{\ast}_{\beta}B$
that is a section with respect to the projection
$T^{\ast}_{\beta}B \rightarrow B$
is equal to $L_{f, \beta}$
for some function $f: B \rightarrow \mathbb{R}$.
\end{example}

\begin{example}\label{example: lift of Legendrians to symplectization}
If $\Lambda$ is a Legendrian submanifold
of a contact manifold $\big(Y, \xi = \ker (\alpha) \big)$
then $S^1 \times \Lambda$ is an exact Lagrangian submanifold
of the locally conformal symplectization
$\big(S^1 \times Y \,,\, [(-d\theta \,,\, d_{-d\theta}\alpha \,,\, \alpha)]\big)$,
of action zero with respect to $(-d\theta \,,\, d_{-d\theta}\alpha \,,\, \alpha)$.
We say that $S^1 \times \Lambda$
is the lift of $\Lambda$
to the locally conformal symplectization.
Similarly,
the lift of a Legendrian immersion
$i: \Lambda \rightarrow Y$
is the exact Lagrangian immersion
\[
\widetilde{i}: S^1 \times \Lambda \rightarrow S^1 \times Y \,,\;
(\theta, x) \mapsto \big( \theta, i(x) \big) \,.
\]
\end{example}

The next proposition follows easily from the definitions.

\begin{prop}\label{proposition: image of exact by exact}
Let $\big( M, [(\eta, \omega, \lambda)]\big)$ be a Liouville lcs manifold,
$i: L \rightarrow M$ an exact Lagrangian immersion
and $\varphi: M \rightarrow M$ an exact lcs diffeomorphism.
Then $\varphi \circ i: L \rightarrow M$
is an exact Lagrangian immersion.
\end{prop}

In particular,
we obtain the following corollary.

\begin{cor}\label{corollary: image of zero section is exact}
Any Lagrangian submanifold of a twisted cotangent bundle $T^{\ast}_{\beta}B$
that is the image of the zero section by a lcs Hamiltonian diffeomorphism
is exact.
\end{cor}

\begin{proof}
By Proposition \ref{proposition: Hamiltonian vs exact},
any lcs Hamiltonian diffeomorphism is exact.
Since the zero section is an exact Lagrangian submanifold of $T^{\ast}_{\beta}B$,
the result thus follows from Proposition \ref{proposition: image of exact by exact}.
\end{proof}

We have the following lcs analogue
of the Weinstein neighborhood theorem for Lagrangians submanifolds.
We refer to \cite[Theorem 3.2]{OS} or \cite[Theorem 2.11]{CM}
for a proof.

\begin{thm}\label{theorem: Weinstein}
Let $L$ be a compact Lagrangian submanifold
of a lcs manifold $\big( M, [(\eta, \omega)] \big)$.
Denote by $\beta$ the pullback of $\eta$
by the inclusion $L \hookrightarrow M$.
Then there exist a neighborhood $\mathcal{U}$ of $L$ in $M$,
a neighborhood $\mathcal{V}$ of the zero section of $T^{\ast}_{\beta}L$
and a strict lcs diffeomorphism $\psi: \mathcal{U} \rightarrow \mathcal{V}$
with respect to 
$( \left. \eta \right\lvert_{\mathcal{U}} , \left. \omega \right\lvert_{\mathcal{U}} )$
and $( \left. (\pi^{\ast}\beta) \right\lvert_{\mathcal{V}} ,
\left. (d_{\pi^{\ast}\beta}\lambda) \right\lvert_{\mathcal{V}} )$
such that $\psi(L)$ is the zero section.
\end{thm}

If $L$ is an exact compact Lagrangian submanifold
of a Liouville lcs manifold $\big( M, [(\eta, \omega, \lambda)] \big)$
then the strict lcs diffeomorphism $\psi: \mathcal{U} \rightarrow \mathcal{V}$
of Theorem \ref{theorem: Weinstein}
is actually exact.
To see this,
we need the following lemma.

\begin{lemma}\label{lemma: Poincare}
Consider a vector bundle $\pi: E \rightarrow B$,
and a closed $1$-form $\beta$ on $B$
that is not exact.
Identify $B$ with the zero section,
and denote by $i: B \hookrightarrow E$ the inclusion.
Let $\mathcal{U} \subset E$ be a neighborhood of $B$
that is fibrewise starshaped,
i.e.\ $\mathcal{U} \cap E_x \subset E_x$ is starshaped
for all fibres $E_x = \pi^{-1} (x)$.
If $\sigma$ is a $\pi^{\ast}\beta$-closed form on $\mathcal{U}$
such that $i^{\ast} \sigma$ is $\beta$-exact
then $\sigma$ is $\pi^{\ast}\beta$-exact.
\end{lemma}

\begin{proof}
Let $h_t: \mathcal{U} \rightarrow \mathcal{U}$
be a homotopy between $h_0 := \left. \id \right\lvert_{\mathcal{U}}$
and $h_1 := i \circ \pi$,
obtained by contracting radially along the fibres.
Consider the total map
$$
h: \mathcal{U} \times \mathbb{R} \rightarrow \mathcal{U} \,,\; h (x,t) = h_t (x) \,,
$$
and the operator
$$
I: \Omega^k (\mathcal{U} \times \mathbb{R}) \rightarrow \Omega^{k-1} (\mathcal{U}) \,,\;
I (\zeta) = \int_0^1 j_t^{\,\ast} (\iota_{\partial / \partial t} \zeta) \, dt \,,
$$
where, for every $t$,
$j_t: \mathcal{U} \rightarrow \mathcal{U} \times \mathbb{R}$
is the map $j_t (x) = (x, t)$.
Then $H := I \circ h^{\ast}:
\Omega^k (\mathcal{U}) \rightarrow \Omega^{k-1} (\mathcal{U})$
satisfies
$$
h_1^{\,\ast} - h_0^{\,\ast} = d_{\pi^{\ast}\beta} \circ H + H \circ d_{\pi^{\ast}\beta} \,.
$$
Indeed, 
consider the maps
$$
\phi_t: \mathcal{U} \times \mathbb{R} \rightarrow \mathcal{U} \times \mathbb{R} \,,\;
\phi_t (x, s) = (x, s + t) \,.
$$
Then $h_t = h \circ j_t = h \circ \phi_t \circ j_0$,
thus
\begin{align*}
h_1^{\,\ast} \zeta - h_0^{\,\ast} \zeta
&= \int_0^1 \frac{d}{dt} \; h_t^{\,\ast} \zeta \, dt
= \int_0^1 j_0^{\,\ast} \frac{d}{dt} \; \phi_t^{\,\ast}  (h^{\ast}\zeta) \, dt
=  \int_0^1 j_0^{\,\ast} \phi_t^{\,\ast} (\mathcal{L}_{\partial/\partial t} h^{\ast}\zeta) \, dt \\
&=  \int_0^1 j_0^{\,\ast} \phi_t^{\,\ast}
\Big( d_{\pi^{\ast}\beta}(\iota_{\partial/\partial t} h^{\ast}\zeta)
+ \iota_{\partial/\partial t} d_{\pi^{\ast}\beta} h^{\ast}\zeta
+ \pi^{\ast} \beta \big(\frac{\partial}{\partial t}\big) \, h^{\ast}\zeta
\Big) \, dt \\
&=  \int_0^1 j_0^{\,\ast} \phi_t^{\,\ast}
\big( d_{\pi^{\ast}\beta}(\iota_{\partial/\partial t} h^{\ast}\zeta)
+ \iota_{\partial/\partial t} d_{\pi^{\ast}\beta} h^{\ast}\zeta
\big) \, dt \\
&=  d_{\pi^{\ast}\beta} \int_0^1 j_t^{\,\ast} (\iota_{\partial/\partial t} h^{\ast}\zeta) \, dt
+  \int_0^1 j_t^{\ast} \big(\iota_{\partial/\partial t} d_{\pi^{\ast}\beta} h^{\ast}\zeta \big)  \, dt \\
&= d_{\pi^{\ast}\beta} \circ I \circ h^{\ast} (\zeta)
+ I \circ h^{\ast} \circ d_{\pi^{\ast}\beta} (\zeta) \\
&= d_{\pi^{\ast}\beta} \circ H \, (\zeta) + H \circ d_{\pi^{\ast}\beta} \, (\zeta) \,,
\end{align*}
where in the fifth equality we have used
that $d_{\pi^{\ast}\beta} h^{\ast} \zeta = h^{\ast} \circ d_{\pi^{\ast}\beta} (\zeta)$,
since $\pi \circ h = \pi$.

Since $i^{\ast} \sigma$ is $\beta$-exact,
the pullback $\pi^{\ast} (i^{\ast} \sigma)$
is $\pi^{\ast}\beta$-exact.
Since moreover $\sigma$ is $\pi^{\ast}\beta$-closed,
we conclude that
$$
\sigma = h_0^{\, \ast} (\sigma)
= h_1^{\, \ast} (\sigma) - d_{\pi^{\ast}\beta} \big(H(\sigma)\big) - H (d_{\pi^{\ast}\beta}\sigma)
= \pi^{\ast} (i^{\ast} \sigma) - d_{\pi^{\ast}\beta} \big(H(\sigma)\big)
$$
is $\pi^{\ast}\beta$-exact.
\end{proof}

\begin{remark}\label{remark: exact Weinstein}
Without loss of generality
we can assume that the neighborhood $\mathcal{V}$
of the zero section of $T^{\ast}_{\beta}L$
in Theorem \ref{theorem: Weinstein} is fibrewise starshaped.
If $L$ is an exact compact Lagrangian submanifold
of a Liouville lcs manifold $\big( M, [(\eta, \omega, \lambda)] \big)$,
the strict lcs diffeomorphism $\psi: \mathcal{U} \rightarrow \mathcal{V}$
of Theorem \ref{theorem: Weinstein}
is then exact.
Indeed,
denoting now the tautological $1$-form on $T^{\ast}L$ by $\lambda_0$,
$(\psi^{-1})^{\ast}\lambda - \lambda_0$ is a $\pi^{\ast}\beta$-closed $1$-form
whose pullback by the inclusion of $L$ as the zero section
is $\pi^{\ast}\beta$-exact.
Using Lemma \ref{lemma: Poincare}
we conclude that $(\psi^{-1})^{\ast}\lambda - \lambda_0$
is $\pi^{\ast}\beta$-exact,
and so $\psi$ is exact.
\end{remark}

\subsection{Contactization}

The contactization of a Liouville lcs manifold
$\big(M \,,\, [(\eta, \omega, \lambda)]\big)$
with respect to $(\eta, \omega, \lambda)$
is the contact manifold
$\big( M \times \mathbb{R} \,,\, \xi_{\lambda} = \ker (\alpha_{\lambda}) \big)$
with $\alpha_{\lambda} = \lambda - d_{\eta}z$,
where $z$ denotes the coordinate on $\mathbb{R}$
and where we still denote by $\lambda$ and $\eta$
their pullbacks by the projection
$M \times \mathbb{R} \rightarrow M$.

Let $i: L \rightarrow M$ be an exact Lagrangian immersion
into a Liouville lcs manifold $\big(M, [(\eta, \omega, \lambda)]\big)$,
of action $S$ with respect to $(\eta, \omega,\lambda)$.
Then
\[
\widetilde{i}: L \rightarrow M \times \mathbb{R} \,,\;
\widetilde{i} (q) = \big( i(q) , S(q) \big)
\]
is a Legendrian immersion
into $(M \times \mathbb{R} \,,\, \xi_{\lambda})$.
We say that $\widetilde{i}$
is the lift of $i$
to the contactization.

Let $\varphi$ be an exact lcs diffeomorphism
of a Liouville lcs manifold $\big( M , [(\eta, \omega, \lambda)] \big)$,
with conformal factor $g$ and action $S$
with respect to $(\eta, \omega, \lambda)$.
Then the diffeomorphism $\widetilde{\varphi}$
of the contactization $(M \times \mathbb{R} \,,\, \xi_{\lambda})$
defined by
\[
\widetilde{\varphi} \, (p, z)
= \Big( \varphi(p) \,,\, e^{g(p)} \big( z - S(p) \big) \Big)
\]
is a contactomorphism,
with conformal factor $(p, z) \mapsto g(p)$
with respect to $\alpha_{\lambda}$.
We say that $\widetilde{\varphi}$ is the lift of $\varphi$
to the contactization.

\begin{prop}\label{proposition: lift of lcs Hamiltonian isotopy to the contactization}
Let $\{\varphi_t\}$ be a lcs Hamiltonian isotopy
of a Liouville lcs manifold $\big( M , [(\eta, \omega, \lambda)] \big)$
with Hamiltonian function $H_t$
with respect to $(\eta, \omega)$.
Then its lift $\{\widetilde{\varphi_t}\}$
to the contactization $( M \times \mathbb{R} \,,\, \xi_{\lambda})$
is a contact isotopy with Hamiltonian function
$(p, z) \mapsto H_t(p)$
with respect to $\alpha_{\lambda}$.
\end{prop}

\begin{proof}
The contact isotopy $\{\widetilde{\varphi_t}\}$
is generated by the vector field $\widetilde{X_t}$
defined by
$$
\widetilde{X_t} \big( \widetilde{\varphi_t}(p) \big)
= \Big( X_t \big( \varphi_t (p)\big) \,,\,
\frac{d}{dt} \Big( e^{g_t(p)} \big( z - S_t(p) \big) \Big)  \Big) \,,
$$
where $S_t$ is the action of $\{\varphi_t\}$
and $X_t$ the vector field generating $\{\varphi_t\}$.
By Proposition \ref{proposition: Hamiltonian vs exact}
we have
$$
\frac{dS_t}{dt} =
e^{-g_t} \Big( \big( H_t - \lambda (X_t) \big) \circ \varphi_t \Big) \,.
$$
Moreover,
by Lemmas \ref{lemma: conformal vector field 1}
and \ref{lemma: conformal vector field}
we have $\frac{dg_t}{dt} = \eta (X_t) \circ \varphi_t$.
We thus obtain
\[
\frac{d}{dt} \Big( e^{g_t(p)} \big( z - S_t(p) \big) \Big)
= - e^{g_t(p)} \, \frac{d S_t}{dt} (p) + e^{g_t(p)} \, \frac{d g_t}{dt} (p) \, \big( z - S_t(p) \big)
\]
\[
= \lambda \Big( X_t \big( \varphi_t (p) \big) \Big) - H_t \big( \varphi_t (p) \big)
+ e^{g_t(p)} \, \eta \Big( X_t \big( \varphi_t (p) \big) \Big) \, \big( z - S_t(p) \big) \,,
\]
and so
\[
\alpha_{\lambda} \Big( \widetilde{X_t} \big( \widetilde{\varphi_t}(p) \big) \Big)
= (\lambda - dz + z \eta)
\Big( \widetilde{X_t} \big( \widetilde{\varphi_t}(p) \big) \Big)
\]
\[
= \lambda \Big( X_t \big( \varphi_t (p) \big) \Big) - \lambda \Big( X_t \big( \varphi_t (p) \big) \Big)
+ H_t \big( \varphi_t (p) \big)
- e^{g_t(p)} \, \eta \Big( X_t \big( \varphi_t (p) \big) \Big) \, \big( z - S_t(p) \big)
\]
\[
+ e^{g_t(p)} \, \big( z - S_t(p) \big) \, \eta \Big( X_t \big( \varphi_t (p) \big) \Big)
= H_t \big( \varphi_t (p) \big) \,.
\]
This shows that the contact isotopy $\{\widetilde{\varphi_t}\}$
has Hamiltonian function $(p, z) \mapsto H_t(p)$
with respect to $\alpha_{\lambda}$.
\end{proof}

\begin{example}\label{example: twisted 1-jet}
The contactization of the twisted cotangent bundle $T_{\beta}^{\ast}B$
with respect to $(\pi^{\ast}\beta \,,\, d_{\pi^{\ast}\beta}\lambda \,,\, \lambda)$
is the twisted $1$-jet bundle
\[
J^1_{\beta} B := \big( \, T^{\ast}B \times \mathbb{R} \,,\,
\ker (\lambda - d_{\pi^{\ast}\beta}z) \, \big) \,.
\]
For any function $f: B \rightarrow \mathbb{R}$,
the twisted $1$-jet
\[
j^1_{\beta}f: B \rightarrow J^1_{\beta}B \,,\;
j^1_{\beta}f (q) = \big( d_{\beta}f (q) \,,\, f(q) \big)
\]
is a Legendrian embedding,
which is the lift of the exact Lagrangian embedding
$d_{\beta}f: B \rightarrow T_{\beta}^{\ast}B$.
Following \cite{CM},
we consider the strict contactomorphism
\begin{equation}\label{equation: contactomorphism 1 jet}
U_{\beta}: J^1_{\beta}B \rightarrow J^1B \,,\;
U_{\beta}(\sigma_q, z) = ( \sigma_q + z \beta(q), z ) \,.
\end{equation}
This map sends $\beta$-twisted 1-jets to usual 1-jets:
\[
U_{\beta} \circ j^1_{\beta}f = j^1f \,.
\] 
In particular,
$U_{\beta}$ sends the zero section to the zero section.
We call it the untwisting map from $J^1_{\beta}B$ to $J^1B$.
\end{example}

The untwisting map $U_{\beta}$ of Example \ref{example: twisted 1-jet}
will play an important role in Sections \ref{section: gf}
and \ref{section: lcs product etc}.
In particular,
we will need the following lemma,
whose proof follows directly from the definitions
and Example \ref{example: lift of Legendrians to symplectization}.

\begin{lemma}\label{lemma: untwisting of lift}
Let $\Lambda$ be a Legendriann submanifold
of a $1$-jet bundle $J^1B$,
and $S^1 \times \Lambda$ its lift
to the locally conformal symplectization $S^1 \times J^1B$.
Identify $S^1 \times J^1B$ with $T^{\ast}_{-d\theta} (S^1 \times B)$
by the strict lcs diffeomorphism
of Lemma \ref{lemma: identification twisted cotangent bundle with conformal symplectization},
and consider the lift $\widetilde{S^1 \times \Lambda}$
of $S^1 \times \Lambda \subset T^{\ast}_{-d\theta} (S^1 \times B)$
to the contactization $J^1_{-d\theta} (S^1 \times B)$.
Then
\[
U_{-d\theta} (\widetilde{S^1 \times \Lambda})
= 0_{S^1} \times \Lambda \,,
\]
where $U_{-d\theta}$ is the untwisting map \eqref{equation: contactomorphism 1 jet}
from $J^1_{-d\theta} (S^1 \times B)$ to $J^1 (S^1 \times B)$
and where we identify $J^1 (S^1 \times B)$ with $T^{\ast}S^1 \times J^1B$. 
\end{lemma}

We will also use the untwisting map in the next section
in order to derive (strong) orderability of twisted cotangent bundles
from the corresponding result in the contact case
for $1$-jet bundles.

\subsection{Orderability}\label{section: orderability}

Let $\big( M, [(\eta, \omega)] \big)$ be a connected lcs manifold
such that $\eta$ is not exact.
By Remark \ref{remark: unique solution},
the Hamiltonian function $H_t$
of a lcs Hamiltonian isotopy $\{\varphi_t\}$
with respect to a representative $(\eta, \omega)$
is then uniquely defined.
Moreover,
it is non-negative (respectively positive)
if and only if the same is true
for the Hamiltonian function with respect to any other representative
(indeed,
the Hamiltonian function of $\{\varphi_t\}$
with respect to another representative $(\eta + df, e^f\omega)$ is $e^f H_t$).
We say that a lcs Hamiltonian isotopy $\{\varphi_t\}$
is non-negative (respectively positive)
if its Hamiltonian function
with respect to some (hence any) representative $(\eta, \omega)$
is non-negative (respectively positive)\footnote{If $\eta$ is exact
then, by Example \ref{example: Hamiltonian conformal},
the lcs Hamiltonian isotopies of $\big( M, [(\eta, \omega)] \big)$
coincide with the Hamiltonian isotopies
of the symplectic manifold $(M, e^{-\mu}\omega)$,
where $\mu$ is any primitive of $\eta$.
The Hamiltonian function of a Hamiltonian isotopy
is in this case only defined up to addition of a constant.
If $M$ is not compact
then the Hamiltonian functions of compactly supported Hamiltonian isotopies
can be normalized by requiring them to have compact support,
and we can define a compactly supported Hamiltonian isotopy
to be non-negative (respectively positive)
if its normalized Hamiltonian function
is non-negative (respectively positive).
However,
we do not consider this case in our discussion
of orderability
because it does not add anything
to what is already known for symplectic manifolds.}.

\begin{example}\label{example: orderabiliity cf contact}
Recall from \cite{EP}
that a contact isotopy $\{\phi_t\}$
on a co-oriented contact manifold $(Y, \xi)$
is said to be non-negative (respectively positive)
if its contact Hamiltonian function $H_t$
with respect to some (hence any) contact form $\alpha$
is non-negative (respectively positive),
equivalently if it moves every point
in a direction positively transverse or tangent to $\xi$
(respectively positively transverse to $\xi$).
By Example \ref{Example:Lift of contact Hamiltonians},
$\{\phi_t\}$ is a non-negative (respectively positive) contact isotopy
if and only if its lift $\{\widetilde{\phi_t}\}$
to the locally conformal symplectization
is a non-negative (respectively positive) lcs Hamiltonian isotopy.
\end{example}

Recall that a bi-invariant partial order on a group $G$
is a partial relation $\leq$ that satisfies the following properties:
\renewcommand{\theenumi}{\roman{enumi}}
\begin{enumerate}
\item \label{properties po: reflexivity} \emph{Reflexivity:}
$g \leq g$.
\item \label{properties po: transitivity} \emph{Transitivity:}
if $g_1 \leq g_2$ and $g_2 \leq g_3$ then $g_1 \leq g_3$.
\item \label{properties po: anti-symmetry} \emph{Anti-symmetry:}
if $g_1 \leq g_2$ and $g_2 \leq g_1$ then $g_1 = g_2$.
\item \label{properties po: bi-invariance} \emph{Bi-invariance:}
if $g_1 \leq g_2$ then $h g_1 \leq h g_2$ and $g_1 h \leq g_2 h$.
\end{enumerate}

Recall also from \cite{EP}
that a compact co-oriented contact manifold $(Y, \xi)$
is said to be orderable
if the partial relation $\leq$
on the universal cover of the identity component
of the contactomorphism group
defined by posing $[\{\phi_t^{(1)}\}] \leq [\{\phi_t^{(2)}\}]$
if $[\{\phi_t^{(2)}\}] \cdot [\{\phi_t^{(1)}\}]^{-1}$
can be represented by a non-negative contact isotopy
is a bi-invariant partial order.
If $(Y, \xi)$ is not compact,
we consider the partial relation $\leq$
on the universal cover of the identity component
of the group of compactly supported contactomorphisms
defined in the same way,
and say that $(Y, \xi)$ is orderable
if $\leq$ is a partial order
(see \cite{FPR}).
Moreover,
a contact manifold $(Y, \xi)$
is said to be strongly orderable
if the partial relation $\leq$
on the group of compactly supported contactomorphisms
defined by posing $\phi_1 \leq \phi_2$
if $\phi_2 \circ \phi_1^{-1}$
is the time-$1$ map of a non-negative
compactly supported contact isotopy
is a partial order.

Similarly,
for a connected lcs manifold $\big( M, [(\eta, \omega)] \big)$
such that $\eta$ is not exact
we consider the partial relation $\leq$
on the universal cover of the group
of compactly supported lcs Hamiltonian diffeomorphisms
defined by posing $[\{\varphi_t^{(1)}\}] \leq [\{\varphi_t^{(2)}\}]$
if $[\{\varphi_t^{(2)}\}] \cdot [\{\varphi_t^{(1)}\}]^{-1}$
can be represented by a non-negative lcs Hamiltonian isotopy.
We say that $\big( M, [(\eta, \omega)] \big)$ is orderable
if $\leq$ is a bi-invariant partial order.
As in the contact case,
reflexivity, transitivity and bi-invariance of $\leq$
are always true,
and can be easily proved
using Lemmas \ref{lemma: composition law Hamiltonians}
and \ref{lemma: Hamiltonian conjugation}.
On the other hand,
anti-symmetry is equivalent to the non-existence
of a non-negative non-constant contractible loop
of compactly supported lcs Hamiltonian diffeomorphisms,
and can fail in general.
For instance,
by Example \ref{Example:Lift of contact Hamiltonians}
any non-negative non-constant contractile loop of contactomorphisms
on a contact manifold lifts to a non-negative non-constant contractile loop
of lcs Hamiltonian diffeomorphisms of the locally conformal symplectization.
Thus,
the locally conformal symplectization
of a contact manifold that is not orderable
(for instance the standard contact sphere \cite{EKP})
is also not orderable.
The converse implication is in general not clear,
but we will see in Section \ref{section: order and metric}
that it holds for the locally conformal symplectizations
of $(\mathbb{R}^{2n+1}, \xi_0)$ and $(\mathbb{R}^{2n} \times S^1, \xi_0)$.

Consider now the partial relation $\leq$
on the group of compactly supported
lcs Hamiltonian diffeomorphisms of $\big( M, [(\eta, \omega)] \big)$
defined by posing $\varphi_1 \leq \varphi_2$
if $\varphi_2 \circ \varphi_1^{-1}$
is the time-$1$ map of a non-negative lcs Hamiltonian isotopy.
We say that $\big( M, [(\eta, \omega)] \big)$ is strongly orderable
if $\leq$ is a bi-invariant partial order
(equivalently,
if it does not admit
any non-negative non-constant loop
of compactly supported lcs Hamiltonian diffeomorphisms).
In particular,
if $\big( M, [(\eta, \omega)] \big)$ is strongly orderable
then it is orderable.
Again,
the locally conformal symplectization of a contact manifold
that is not strongly orderable
(for instance any contact manifold
with periodic Reeb flow)
is also not strongly orderable.
We will see in Section \ref{section: order and metric}
that the locally conformal symplectizations
of $(\mathbb{R}^{2n+1}, \xi_0)$ and $(\mathbb{R}^{2n} \times S^1, \xi_0)$
are in fact strongly orderable.
Using the untwisting map \eqref{equation: contactomorphism 1 jet}
of Example \ref{example: twisted 1-jet},
we now derive strong orderability
of twisted cotangent bundles with compact base
from the corresponding result for $1$-jet bundles
\cite{CFP, CN}. 

\begin{proof}[Proof of Theorem \ref{theorem: orderability - intro} (i)]
Let $B$ be a compact connected manifold,
and $\beta$ a closed non-exact $1$-form.
We want to show that the twisted cotangent bundle $T_{\beta}^{\ast}B$
is strongly orderable,
thus that it does not admit any
non-constant non-negative loop $\{\varphi_t\}$
of compactly supported lcs Hamiltonian diffeomorphisms.
Suppose by contradiction that such $\{\varphi_t\}$ exists.
By Proposition \ref{proposition: lift of lcs Hamiltonian isotopy to the contactization},
the lift $\{\widetilde{\varphi}_t\}$ to the contactization $J^1_{\beta}B$
is a non-constant non-negative loop of compactly supported contactomorphisms.
By Lemma \ref{lemma: Hamiltonian conjugation},
the conjugation $\{U_{\beta} \circ \widetilde{\varphi}_t \circ U_{\beta}^{-1}\}$
by the untwisting map \eqref{equation: contactomorphism 1 jet}
is thus a non-constant non-negative loop
of compactly supported contactomorphisms of $J^1B$.
By \cite[Proposition 2.12]{FPR},
$\{U_{\beta} \circ \widetilde{\varphi}_t \circ U_{\beta}^{-1}\}$
can be deformed to a loop of contactomorphisms
whose restriction to the zero section is positive.
But this contradicts \cite[Theorem 1]{CFP}
and \cite[Corollary 5.5]{CN}.
\end{proof}

\subsection{Lcs product and graph}\label{section: lcs product}

In this section we describe the lcs product
of a lcs manifold
and the lcs graph of a divergence free lcs diffeomorphism,
referring to \cite{CS}
for a more formal point of view on these notions.

Let $\big( M, [(\eta, \omega)] \big)$
be a lcs manifold.
Consider the homomorphism
\[
\Delta_{\eta}: \pi_1 (M) \times \pi_1 (M) \rightarrow \mathbb{R} \,,\;
( [\gamma_1], [\gamma_2] ) \mapsto
\langle [\eta], [\gamma_1] \rangle - \langle [\eta], [\gamma_2] \rangle \,,
\]
and let
\[
M \boxtimes M \rightarrow M \times M
\]
be the cover corresponding to the subgroup
$\ker (\Delta_{\eta})$
of $\pi_1 (M \times M) \cong \pi_1 (M) \times \pi_1 (M)$.
In other words,
if we denote by $\pi: \bar{M} \rightarrow M$ the symplectic cover
then the manifold $M \boxtimes M$ is the quotient of $\bar{M} \times \bar{M}$
by the equivalence relation
\begin{equation}\label{equation: equivalence relation lcs product}
(x_1, x_2) \sim (y_1, y_2) \quad \text{ if }
\pi (x_1) = \pi (y_1) \,,\; \pi (x_2) = \pi (y_2)
\text{ and } \mu (x_1) - \mu (y_1) = \mu (x_2) - \mu (y_2) \,,
\end{equation}
where $\mu$ is any primitive of $\pi^{\ast}\eta$.
Consider the symplectic form
$\bar{\omega} = e^{-\mu} \, \pi^{\ast} \omega$
on $\bar{M}$.
Then the symplectic form
$- \pr_1^{\ast} \bar{\omega} + \pr_2^{\ast} \bar{\omega}$
on $\bar{M} \times \bar{M}$,
where $\pr_1$ and $\pr_2$ denote the projections of $\bar{M} \times \bar{M}$
to the first and second factors,
can be written as
\[
- \pr_1^{\ast} \bar{\omega} + \pr_2^{\ast} \bar{\omega}
= - e^{-\mu \circ \pr_1} \, (\pi\circ \pr_1)^{\ast} \omega
+ e^{-\mu \circ \pr_2} \, (\pi \circ \pr_2)^{\ast} \omega
\]
\[
= e^{-\mu \circ \pr_1} \big( - (\pi \circ \pr_1)^{\ast} \omega
+ e^{\mu \circ \pr_1 - \mu \circ \pr_2} \, (\pi \circ \pr_2)^{\ast} \omega \big) \,.
\]
Therefore
\[
\big(0 \,,\, - \pr_1^{\ast} \bar{\omega} + \pr_2^{\ast} \bar{\omega} \big) \sim
\big( d(\mu \circ \pr_1) \,,\, - (\pi \circ \pr_1)^{\ast} \omega
+ e^{\mu \circ \pr_1 - \mu \circ \pr_2} \, (\pi \circ \pr_2)^{\ast} \omega \big)
\]
\[
= \big( (\pi \circ \pr_1)^{\ast} \eta \,,\,
- (\pi \circ \pr_1)^{\ast} \omega
+ e^{\mu \circ \pr_1 - \mu \circ \pr_2} \, (\pi \circ \pr_2)^{\ast} \omega \big) \,.
\]
Since the function $\mu \circ \pr_1 - \mu \circ \pr_2$
is invariant by the equivalence relation
\eqref{equation: equivalence relation lcs product}, the forms 
$(\pi \circ \pr_1)^{\ast} \eta$
and
$$
- (\pi \circ \pr_1)^{\ast} \omega
+ e^{\mu \circ \pr_1 - \mu \circ \pr_2} \, (\pi \circ \pr_2)^{\ast} \omega
$$
descend to forms $\eta \boxtimes \eta$
and $\omega \boxtimes \omega$ on $M \boxtimes M$,
defining a lcs structure.
We say that
$$
\big(M \boxtimes M \,,\,
[(\eta \boxtimes \eta, \omega \boxtimes \omega)] \big)
$$
is the lcs product of the lcs manifold
$\big( M, [(\eta, \omega)] \big)$
with respect to $(\eta, \omega)$.
By construction,
its symplectic cover is the conformal class
$(\bar{M} \times \bar{M} \,,\,
[- \pr_1^{\ast} \bar{\omega} + \pr_2^{\ast}\bar{\omega}] )$
of the symplectic product of the symplectic cover.
We denote by
\[
\pi \boxtimes \pi: \bar{M} \times \bar{M}
\rightarrow M \boxtimes M
\]
the symplectic cover projection.
The following diagram of covers
summarizes the situation:
\[
\begin{tikzcd}
\widetilde{M} \times \widetilde{M}
\arrow[rdd, "\ker(\eta) \times \ker(\eta)"]
\arrow[rddd, "\ker(\Delta_{\eta})"' near end]
\arrow[dddd, "\pi_1 (M) \times \pi_1 (M)" ' near end]
& \\
& \\
& \bar{M} \times \bar{M} \arrow{d} \\
& M \boxtimes M \arrow{ld} \\
M \times M \,.
\end{tikzcd}
\]

For a Liouville lcs manifold
$\big( M, [(\eta, \omega, \lambda)] \big)$
we have
\[
- (\pi \circ \pr_1)^{\ast} \omega
+ e^{\mu \circ \pr_1 - \mu \circ \pr_2} \, (\pi \circ \pr_2)^{\ast} \omega
= d_{(\pi \circ \pr_1)^{\ast} \eta}
( - (\pi \circ \pr_1)^{\ast} \lambda
+ e^{\mu \circ \pr_1 - \mu \circ \pr_2} \, (\pi \circ \pr_2)^{\ast} \lambda ) \,.
\]
The $1$-form $- (\pi \circ \pr_1)^{\ast}\lambda
+ e^{\mu \circ \pr_1 - \mu \circ \pr_2} \, (\pi \circ \pr_2)^{\ast} \lambda$
on $\bar{M} \times \bar{M}$
descends to a $1$-form $\lambda \boxtimes \lambda$
on $M \boxtimes M$
with $\omega \boxtimes \omega
= d_{\eta \boxtimes \eta} \lambda \boxtimes \lambda$.
The lcs structure $[(\eta \boxtimes \eta, \omega \boxtimes \omega)]$
on $M \boxtimes M$ is thus exact,
with Liouville form $\lambda \boxtimes \lambda$.
We say that
$$
\big(M \boxtimes M \,,\,
[(\eta \boxtimes \eta, \omega \boxtimes \omega, \lambda \boxtimes \lambda)] \big)
$$
is the lcs product of the Liouville lcs manifold
$\big( M, [(\eta, \omega, \lambda)] \big)$
with respect to $(\eta, \omega, \lambda)$.

\begin{lemma}\label{lemma: Lee flow lcs product}
Let $\{\bar{\varphi}_t^{\,\omega}\}$
be the lift to $\bar{M}$ of the Lee flow on $M$
with respect to $(\eta, \omega)$.
The isotopy
\[
(x_1, x_2) \mapsto \big( \bar{\varphi}_{-t}^{\,\omega} (x_1) \,,\, x_2 \big) 
\]
on $\bar{M} \times \bar{M}$
commutes with $\pi \boxtimes \pi: \bar{M} \times \bar{M} \rightarrow M \boxtimes M$,
and the induced isotopy on $M \boxtimes M$
is the Lee flow with respect to
$(\eta \boxtimes \eta, \omega \boxtimes \omega)$.
\end{lemma}

\begin{proof}
By Lemma \ref{lemma: Lee flow symplectic cover},
$\{\bar{\varphi}_t^{\,\omega}\}$ is the Hamiltonian flow
of $e^{-\mu}$ with respect to $\bar{\omega}$.
The isotopy
$(x_1, x_2) \mapsto \big( \bar{\varphi}_{-t}^{\,\omega} (x_1) , x_2 \big)$
is thus the Hamiltonian flow of $e^{-\mu \circ \pr_1}$
with respect to 
$- \pr_1^{\ast}\bar{\omega} + \pr_2^{\ast}\bar{\omega}$.
Since $\mu \circ \pr_1$ is a primitive
of the pullback of $\eta \boxtimes \eta$
by $\pi \boxtimes \pi$,
and $- \pr_1^{\ast}\bar{\omega} + \pr_2^{\ast}\bar{\omega}$
is the pullback of $\omega \boxtimes \omega$
by $\pi \boxtimes \pi$,
by Lemma \ref{lemma: Lee flow symplectic cover}
the isotopy $(x_1, x_2)
\mapsto \big( \bar{\varphi}_{-t}^{\,\omega} (x_1) , x_2 \big)$
commutes with $\pi \boxtimes \pi$
and induces the Lee flow on $M \boxtimes M$
with respect to
$(\eta \boxtimes \eta, \omega \boxtimes \omega)$.
\end{proof}

We define the contact product\footnote{
In \cite{Bhupal,San10, San11, FSZ}
the contact product of $\big(Y, \xi = \ker(\alpha)\big)$
with respect to $\alpha$
is defined to be
$\big( Y \times Y \times \mathbb{R} \,,\,
\ker ( \pr_2^{\ast} \alpha - e^{\rho} \pr_1^{\ast} \alpha) \big)$.
The formula we use in the present article
works better for our purposes,
and could also have been used
in \cite{Bhupal,San10, San11, FSZ}
(with a consequently adapted formula
for the indentification
of $\mathbb{R}^{2n+1} \times \mathbb{R}^{2n+1} \times \mathbb{R}$
with $J^1\mathbb{R}^{2n+1}$,
as in Section \ref{section: lcs product etc})
without any inconvenience.}
of a contact manifold
$\big(Y, \xi = \ker(\alpha)\big)$
with respect to the contact form $\alpha$
to be the contact manifold
\[
\big( Y \times Y \times \mathbb{R} \,,\,
\ker ( e^{-\rho} \pr_2^{\ast} \alpha - \pr_1^{\ast} \alpha) \big) \,,
\]
where $\rho$ denotes the coordinate on $\mathbb{R}$
and $\pr_1$ and $\pr_2$ the projections
to the first and second factors.
The lcs product and the contact product
are related as follows.
Consider the symplectic cover
$\pi: \mathbb{R} \times Y \rightarrow S^1 \times Y$
of the conformal symplectization,
and the product $(\mathbb{R} \times Y) \times (\mathbb{R} \times Y)$.
Denote by $\theta_1$ and $\theta_2$
the coordinates on the first and second $\mathbb{R}$-factors
and by $\pr_j: (\mathbb{R} \times Y) \times (\mathbb{R} \times Y)
\rightarrow \mathbb{R} \times Y$
for $j = 1, 2$
the projections on the first and second factors.
The lcs product of the locally conformal symplectization
$\big(S^1 \times Y \,,\, [(-d\theta \,,\, d_{-d\theta}\alpha \,,\, \alpha)] \big)$
is
\[
\Big( (\mathbb{R} \times Y \times \mathbb{R} \times Y)/\sim \,,\, 
\big[\big(-d\theta_1 \,,\,
d_{-d\theta_1} ( e^{\theta_2 - \theta_1} \pr_2^{\ast}\alpha - \pr_1^{\ast}\alpha ) \,,\,
e^{\theta_2 - \theta_1} \pr_2^{\ast}\alpha - \pr_1^{\ast}\alpha \big)\big]
\Big) \,,
\]
where $\sim$ is the equivalence relation generated by
\[
(\theta_1, p_1, \theta_2, p_2) \sim (\theta_1 + 1, p_1, \theta_2 + 1, p_2) \,.
\]
The next result is then immediate.

\begin{prop}\label{proposition: lcs product}
For any contact manifold
$\big(Y, \xi = \ker(\alpha)\big)$,
the map
\[
I: (S^1 \times Y) \boxtimes (S^1 \times Y)
\rightarrow S^1 \times (Y \times Y \times \mathbb{R})
\]
defined by
\[
I \, \big( [(\theta_1, p_1, \theta_2, p_2)] \big) = ([\theta_1], p_1, p_2, \theta_1 - \theta_2) \,,
\]
where we identify as above $(S^1 \times Y) \boxtimes (S^1 \times Y)$
with $(\mathbb{R} \times Y \times \mathbb{R} \times Y)/\sim$,
is a strict lcs diffeomorphism
between the lcs product of the locally conformal symplectization
$\big(S^1 \times Y \,,\, [(-d\theta \,,\, d_{-d\theta}\alpha \,,\, \alpha)] \big)$
and the locally conformal symplectization
\[
\big( \, S^1 \times (Y \times Y \times \mathbb{R}) \,,\,
[( -d\theta \,,\,
d_{-d\theta} ( e^{-\rho} \pr_2^{\ast} \alpha - \pr_1^{\ast} \alpha) \,,\,
e^{-\rho} \pr_2^{\ast} \alpha - \pr_1^{\ast} \alpha
)] \, \big)
\]
of the contact product
$\big( Y \times Y \times \mathbb{R} \,,\,
\ker ( e^{-\rho} \pr_2^{\ast} \alpha - \pr_1^{\ast} \alpha) \big)$,
which is exact with action zero.
\end{prop}

Let $\varphi$ be a divergence free lcs diffeomorphism
of a lcs manifold $\big( M, [(\eta, \omega)] \big)$,
and $\bar{\varphi}: \bar{M} \rightarrow \bar{M}$
its lift to the symplectic cover
$(\bar{M}, [e^{-\mu} \pi^{\ast}\omega])$.
By Corollary \ref{corollary: symplectic lift},
for every primitive $\mu$ of $\pi^{\ast}\eta$
and $x \in \bar{M}$
we have
\[
\mu \big(\bar{\varphi}(x)\big) - \mu (x) = g \big(\pi(x)\big) \,,
\]
where $g$ is the conformal factor of $\varphi$
with respect to $(\eta, \omega)$.
This implies that the symplectic graph
\[
\gr (\bar{\varphi}): \bar{M} \rightarrow \bar{M} \times \bar{M} \,,\;
x \mapsto \big( x , \bar{\varphi}(x)\big)
\]
descends to a map
\[
\gr_{\lcs} (\varphi): M \rightarrow M \boxtimes M \,.
\]
Indeed,
for $x, y \in \bar{M}$ with $\pi (x) = \pi (y)$
we have
\[
\mu \big(\bar{\varphi}(x)\big) - \mu (x)
= g \big(\pi (x)\big)
= g \big(\pi (y)\big)
= \mu \big(\bar{\varphi}(y)\big) - \mu (y) \,,
\]
and so $\big(x, \bar{\varphi}(x)\big) \sim \big(y, \bar{\varphi}(y)\big)$.
The map $\gr_{\lcs} (\varphi): M \rightarrow M \boxtimes M$
is a Lagrangian embedding,
which we call the lcs graph
of the divergence free lcs diffeomorphism $\varphi$.

\begin{lemma}\label{lemma: graph exact}
If $\varphi$ is an exact divergence free
lcs diffeomorphism of a Liouville lcs manifold
$\big( M,$ $ [(\eta, \omega, \lambda)] \big)$
with action $S$
with respect to $(\eta, \omega, \lambda)$
then $\gr_{\lcs} (\varphi): M \rightarrow M \boxtimes M$
is an exact Lagrangian embedding
with action $- S$
with respect to
$(\eta \boxtimes \eta, \omega \boxtimes \omega, \lambda \boxtimes \lambda)$.
\end{lemma}

\begin{proof}
Let $\bar{\varphi}$ be the lift of $\varphi$
to the symplectic cover.
Then
\[
\gr (\bar{\varphi})^{\ast}
( - (\pi \circ \pr_1)^{\ast} \lambda
+ e^{\mu \circ \pr_1 - \mu \circ \pr_2} (\pi \circ \pr_2)^{\ast}\lambda)
= - \pi^{\ast} \lambda
+ e^{\mu - \mu \circ \bar{\varphi}} \, \bar{\varphi}^{\ast} \pi^{\ast} \lambda
\]
\[
= - \pi^{\ast} \lambda + e^{-g \circ \pi} \, \pi^{\ast} \varphi^{\ast} \lambda
= - \pi^{\ast} \lambda + e^{-g \circ \pi} \, \pi^{\ast} \big(e^g (\lambda - d_{\eta}S)\big)
= d_{\pi^{\ast} \eta} (- S \circ \pi) \,,
\]
thus $\gr_{\lcs} (\varphi)^{\ast} (\lambda \boxtimes \lambda)
= d_{\eta} (- S)
= d_{\gr_{\lcs}(\varphi)^{\ast} (\eta \boxtimes \eta) } (-S)$.
\end{proof}

\begin{lemma}\label{lemma: product of lcs diffeomorphisms}
Let $\varphi_1$ and $\varphi_2$
be divergence free lcs diffeomorphisms
a lcs manifold $\big( M , [(\eta, \omega)] \big)$,
and $\bar{\varphi}_1$ and $\bar{\varphi}_2$
their lifts to the symplectic cover
$(\bar{M}, [e^{-\mu} \pi^{\ast}\omega])$.
Then the map
\[
\varphi_1 \boxtimes \varphi_2:
M \boxtimes M \rightarrow M \boxtimes M \,,\;
[(x_1, x_2)] \mapsto \big[\big( \bar{\varphi}_1 (x_1) \,,\, \bar{\varphi}_2 (x_2) \big)\big]
\]
is well-defined,
and is a divergence free lcs diffeomorphism.
Moreover,
the following holds:
\begin{itemize}
\item[(i)] For every divergence free lcs diffeomorphism $\varphi$
we have
\[
\gr_{\lcs} (\varphi) = (\id \boxtimes\, \varphi) \circ \gr_{\lcs} (\id) \,.
\]
\item[(ii)] If $\{\varphi_t^{(1)}\}$ and $\{\varphi_t^{(2)}\}$
are lcs Hamiltonian isotopies
with Hamiltonian functions $H_t^{(1)}$ and $H_t^{(2)}$
with respect to $(\eta, \omega)$,
then $\big\{ \varphi_t^{(1)} \boxtimes \varphi_t^{(2)} \big\}$
is a lcs Hamiltonian isotopy
with Hamiltonian function
\[
H_t^{(1)} \boxplus H_t^{(2)} \big( [(x, y)] \big)
= - H_t^{(1)} \big(\pi(x)\big) + e^{\mu(x) - \mu (y)} \, H_t^{(2)} \big(\pi(y)\big)
\]
with respect to $(\eta \boxtimes \eta, \omega \boxtimes \omega)$.
\end{itemize}
\end{lemma}

\begin{proof}
If $(x_1, x_2) \sim (y_1, y_2)$ then
$\big( \bar{\varphi}_1 (x_1) , \bar{\varphi}_2 (x_2) \big)
\sim \big( \bar{\varphi}_1 (y_1) , \bar{\varphi}_2 (y_2) \big)$. 
Indeed, by Corollary \ref{corollary: symplectic lift}
we have
$\mu \big( \bar{\varphi}_j (x_j) \big) = \mu (x_j) + g \big( \pi (x_j) \big)$
and $\mu \big( \bar{\varphi}_j (y_j) \big) = \mu (y_j) + g_j \big( \pi (y_j) \big)$.
This shows that $\varphi_1 \boxtimes \varphi_2$
is well-defined.
Since it is covered by the symplectomorphism
$\bar{\varphi}_1 \times \bar{\varphi}_2$ of the symplectic cover,
it is divergence free.

Statement (i) follows from the fact that
$\gr (\bar{\varphi}) = (\id \times \bar{\varphi}) \circ \gr (\id)$.

For (ii),
by Proposition \ref{proposition: lift of Hamiltonian isotopy to symplectic cover}
the lifts $\{\bar{\varphi}_t^{(j)}\}$ are Hamiltonian isotopies
with Hamiltonian functions
$\bar{H}_t^{(j)} = e^{-\mu} (H_t^{(j)} \circ \pi)$.
Thus $\{ \bar{\varphi}_t^{(1)} \times \bar{\varphi}_t^{(2)} \}$
is a Hamiltonian isotopy of $\bar{M} \times \bar{M}$
with Hamiltonian function
\[
(x, y) \mapsto
- e^{-\mu(x)} H_t^{(1)} \big(\pi(x)\big) + e^{-\mu(y)} H_t^{(2)} \big(\pi(y)\big)
= e^{-\mu(x)} \Big( - H_t^{(1)} \big(\pi(x)\big)
+ e^{\mu(x) - \mu(y)} H_t^{(2)} \big(\pi(y)\big) \Big) \,.
\]
By Proposition \ref{proposition: lift of Hamiltonian isotopy to symplectic cover}
we conclude that $\big\{ \varphi_t^{(1)} \boxtimes \varphi_t^{(2)} \big\}$
is a lcs Hamiltonian isotopy with Hamiltonian function
\[
[(x, y)] \mapsto
- H_t^{(1)} \big(\pi(x)\big) + e^{\mu(x) - \mu (y)} \, H_t^{(2)} \big(\pi(y)\big) \,.
\]
\end{proof}

The contact graph of a contactomorphism $\phi$
of a contact manifold $\big( Y , \xi = \ker (\alpha) \big)$
with respect to $\alpha$
is the Legendrian embedding
\[
\gr (\phi): Y \rightarrow Y \times Y \times \mathbb{R} \,,\;
x \mapsto \big( x, \phi(x), g(x) \big)
\]
into the contact product
$\big(Y \times Y \times \mathbb{R} \,,\,
\ker (e^{-\rho} \pr_2^{\ast}\alpha - \pr_1^{\ast}\alpha)\big)$,
where $g$ is the conformal factor of $\phi$
with respect to $\alpha$.
Recall from Example \ref{example: lift of maps to symplectization}
that the lift of $\phi$ to the locally conformal symplectization
$\big( S^1 \times Y \,,\, [(-d\theta, d_{-d\theta} \alpha)] \big)$
is the lcs diffeomorphism
$\widetilde{\phi}: S^1 \times Y \rightarrow S^1 \times Y$
defined by
$\widetilde{\phi} (\theta, p) = \big( \theta - g(p) , \phi(p) \big)$.
This lcs diffeomorphism is divergence free,
since it is covered by the symplectomorphism
of the symplectic cover
$\big(\mathbb{R} \times Y \,,\, [d(e^{\theta}\alpha)]\big)$
defined by the same formula.
Moreover,
we have the following relation
between the contact graph of $\phi$
and the lcs graph of its lift $\widetilde{\phi}$.

\begin{prop}\label{proposition: relation graphs lcs contact}
Let $\phi$ be a contactomorphism
of a contact manifold $\big( Y , \xi = \ker (\alpha) \big)$,
and $\widetilde{\phi}$ its lift
to the locally conformal symplectization
$\big( S^1 \times Y \,,\, [(-d\theta, d_{-d\theta} \alpha)] \big)$. 
Consider the contact graph
$\gr (\phi): Y \rightarrow Y \times Y \times \mathbb{R}$
of $\phi$,
and its lift
\[
\widetilde{\gr(\phi)}:
S^1 \times Y \rightarrow S^1 \times (Y \times Y \times \mathbb{R})
\]
to the locally conformal symplectization of the contact product.
Then
\[
\widetilde{\gr(\phi)} = I \circ \gr_{\lcs} (\widetilde{\phi}) \,,
\]
where $I: (S^1 \times Y) \boxtimes (S^1 \times Y)
\rightarrow S^1 \times (Y \times Y \times \mathbb{R})$
is the lcs diffeomorphism
defined in Proposition \ref{proposition: lcs product}.
\end{prop}

\begin{proof}
We have
\[
\gr_{\lcs} (\widetilde{\phi}) \, ([\theta], x)
= \big[ \big( \theta, x, \theta - g(x) , \phi(x)\big) \big] \,,
\]
and so
\[
I \circ \gr_{\lcs} (\widetilde{\phi}) \, ([\theta], x)
= \big( [\theta], x, \phi(x), g(x) \big)
= \widetilde{\gr(\phi)} \, ([\theta], x) \,.
\]
\end{proof}

\section{Essential Lee chords and translated points}\label{section: Lee chords and translated points}

Let $p_1$ and $p_2$ be two points
of a lcs manifold $\big( M, [(\eta, \omega)] \big)$.
A Lee chord with respect to $(\eta, \omega)$
from $p_1$ to $p_2$
is a pair $(T, \gamma)$ with $T$ a real number
and $\gamma: [0, 1] \rightarrow M$
a curve such that $\gamma(0) = p_1$,
$\gamma(1) = p_2$
and $\dot{\gamma} (t) = T \, R_{\omega} \big(\gamma(t)\big)$
for all $t$. 
We say that $T$ is the time-shift
of the Lee chord $(T, \gamma)$.
Observe that either $R_{\omega} \big(\gamma(t)\big) = 0$ for all $t$
or $R_{\omega} \big(\gamma(t)\big) \neq 0$ for all $t$.
In the second case the time-shift $T$
is determined by $\gamma$,
and we simply say that $\gamma$ is a Lee chord
from $p_1$ to $p_2$ of time-shift $T$.
If $R_{\omega} \big(\gamma(t)\big) = 0$ for all $t$
then $\gamma$ is the constant curve
$\gamma(t) = p_1$
and, for any other $T'$, the pair
$(T', \gamma)$ is also a Lee chord from $p_1$ to $p_2 = p_1$.
It is pertinent to regard $(T, \gamma)$ and $(T', \gamma)$
as different Lee chords,
indeed in the first case we regard $\gamma$
as a (constant) orbit of the flow $\{\varphi^{\omega}_{Tt}\}_{t \in [0,1]}$
while in the second case we regard it as a (constant) orbit
of the flow $\{\varphi^{\omega}_{T't}\}_{t \in [0,1]}$.
Although $\varphi^{\omega}_{Tt} (p_1) = \varphi^{\omega}_{T't} (p_1) = p_1$ for all $t$,
the two flows do not necessarily coincide
on a neighborhood of $p_1$.
In particular,
their differentials at $p_1$
might be different
(this will play a role for instance
in the definition of transverse Lee chords
between Lagrangian submanifolds below).

Let now $L_1$ and $L_2$ be two Lagrangian submanifolds
of $\big(M, [(\eta, \omega)]\big)$.
A Lee chord with respect to $(\eta, \omega)$
from $L_1$ to $L_2$
is a Lee chord $(T, \gamma)$
from $\gamma(0) \in L_1$ to $\gamma(1) \in L_2$.
We say that it is transverse if
\[
(\varphi_T^{\omega})_{\ast} \big( T_{\gamma(0)}L_1 \big) + T_{\gamma(1)}L_2
= T_{\gamma(1)} M \,.
\]
If $L_1$ and $L_2$ are connected exact Lagrangian submanifolds
of a Liouville lcs manifold $\big(M, [(\eta, \omega, \lambda)]\big)$
and if the pullback of $\eta$ by the inclusion of $L_1$ and $L_2$ into $M$
is not exact
then we define the (total) action
with respect to $(\eta, \omega, \lambda)$
of a Lee chord $(T, \gamma)$
from $L_1$ to $L_2$
by
\[
\mathcal{A}_{L_1, L_2} (T, \gamma)
= S_{L_1} \big(\gamma(0)\big) - S_{L_2} \big(\gamma(1)\big) + \int_{\gamma} \lambda \,,
\]
where $S_{L_1}$ and $S_{L_2}$ are the actions of $L_1$ and $L_2$
with respect to $(\eta, \omega, \lambda)$.
We say that $S_{L_1} \big(\gamma(0)\big) - S_{L_2} \big(\gamma(1)\big)$
and $\int_{\gamma} \lambda$ are respectively
the Hamiltonian and Lee actions
of the Lee chord $(T, \gamma)$.
We say that the Lee chord $(T, \gamma)$ from $L_1$ to $L_2$
is essential if the action $\mathcal{A}_{L_1, L_2} (T, \gamma)$
is equal to the time-shift $T$.

\begin{example}\label{example: action Lee chors symplectic}
Let $L_1$ and $L_2$ be Lagrangian submanifolds
of a conformal symplectic manifold
$\big(M, [\omega]\big) = \big(M, [(0, \omega)]\big)$.
Since the Lee flow $\{\varphi_t^{\omega}\}$ with respect to $(0, \omega)$
is the identity,
$(T, \gamma)$ is a Lee chord from $L_1$ to $L_2$
with respect to $(0, \omega)$
if and only if $\gamma$ is a constant curve
at an intersection point of $L_1$ and $L_2$.
In this case there is no reason
to distinguish between Lee chords
$(T, \gamma)$ and $(T', \gamma)$ for $T \neq T'$,
so we omit the time-shift from the notation.
If $L_1$ and $L_2$ are exact Lagrangian submanifolds
of a conformal Liouville manifold
$\big(M, [(\omega, \lambda)]\big) = \big(M, [(0, \omega, \lambda)]\big)$,
then the Lee action with respect to $(0, \omega, \lambda)$
of any Lee chord vanishes.
The Hamiltonian action is not well-defined,
because $S_{L_1}$ and $S_{L_2}$ are only defined
up to the addition of a constant.
However,
if we restrict to the set $\mathcal{L}_{p_0}$
of Lagrangian submanifolds that contain a certain point $p_0$
then we can normalize the action $S_L$
of any $L \in \mathcal{L}_{p_0}$
by asking that it vanishes at $p_0$.
For $L_1$ and $L_2$ in $\mathcal{L}_{p_0}$,
the total action of a Lee chord $\gamma$ from $L_1$ to $L_2$,
i.e.\ an intersection at $\gamma(t) \equiv p$ of $L_1$ and $L_2$,
is equal to the Hamiltonian action $S_{L_1} (p) - S_{L_2} (p)$
and coincides with the usual (normalized) action of $p$
as a Lagrangian intersection.
\end{example}

\begin{example}\label{example: Lee chords lifts of Legendrians}
Let $\Lambda_1$ and $\Lambda_2$ be Legendrian submanifolds
of a contact manifold $\big( Y , \xi = \ker (\alpha) \big)$,
and consider their lifts
$S^1 \times \Lambda_1$ and $S^1 \times \Lambda_2$
to the locally conformal symplectization
$\big( S^1 \times Y \,,\, [( -d\theta \,,\, d_{-d\theta}\alpha )] \big)$.
Since the Lee vector field
with respect to $( -d\theta \,,\, d_{-d\theta}\alpha )$
never vanishes,
the time-shift of any Lee chord $(T, \gamma)$
is determined by $\gamma$
and so we omit it from the notation.
The Lee chords with respect to $( -d\theta \,,\, d_{-d\theta}\alpha )$
from $S^1 \times \Lambda_1$ to $S^1 \times \Lambda_2$
are of the form $\gamma = (\theta, \gamma_Y)$
with $\theta \in S^1$
and $\gamma_Y$ a Reeb chord with respect to $\alpha$
from $\Lambda_1$ to $\Lambda_2$,
i.e.\ a curve $\gamma_Y: [0,1] \rightarrow M$
such that $\gamma_Y (0) \in \Lambda_1$,
$\gamma_Y (1) \in \Lambda_2$
and $\dot{\gamma}_Y(t) = T R_{\alpha} \big(\gamma_Y(t)\big)$ for all $t$.
By Example \ref{example: lift of Legendrians to symplectization},
the Lagrangian submanifolds $S^1 \times \Lambda_1$ and $S^1 \times \Lambda_2$
have action zero
with respect to $(-d\theta \,,\, d_{-d\theta}\alpha \,,\, \alpha)$.
Thus,
the Hamiltonian action
with respect to $(-d\theta \,,\, d_{-d\theta}\alpha \,,\, \alpha)$
of any Lee chord from $S^1 \times \Lambda_1$ to $S^1 \times \Lambda_2$
vanishes,
and so the total action is given by the Lee action,
which coincides with the contact action $\int_{\gamma_Y}\alpha$
of the Reeb chord $\gamma_Y$.
Since $\lambda (R_{\omega}) = \alpha (R_{\alpha}) = 1$,
the time-shift of $\gamma$ is equal to the Lee action,
hence to the total action.
In particular,
all the Lee chords from $S^1 \times \Lambda_1$ to $S^1 \times \Lambda_2$
are essential.
\end{example}

\begin{example}\label{example: action cotangent bundle}
Let $L_1$ and $L_2$ be Lagrangian submanifolds
of a twisted cotangent bundle
$T_{\beta}^{\ast}B$.
A Lee chord
with respect to $(\pi^{\ast}\beta \,,\, d_{\pi^{\ast}\beta} \lambda)$
from $L_1$ to $L_2$
is a pair $(T, \gamma)$
with $\gamma: [0, 1] \rightarrow T_{\beta}^{\ast}B$
of the form
\[
\gamma(t) = \sigma_q + t \, T \, \beta(q)
\]
with $\gamma (0) = \sigma_q \in L_1$
and $\gamma (1) = \sigma_q + T \, \beta(q) \in L_2$.
Suppose now that the Lagrangian submanifolds $L_1$ and $L_2$
are connected and exact,
and that the pullback of $\pi^{\ast}\beta$
by the inclusions of $L_1$ and $L_2$ into $T_{\beta}^{\ast}B$
are not exact.
The (total) action
with respect to $(\pi^{\ast}\beta \,,\, d_{\pi^{\ast}\beta} \lambda, \lambda)$
is then equal to the Hamiltonian action,
indeed the Lee action vanishes
(since $\lambda$ vanishes
on vectors tangent to the fibres).
The time-shift is in general not equal to the total action
(nor to the Lee action).
\end{example}

In the next lemma and in the rest of the article,
the Lee chords between Lagrangian submanifolds
of a twisted cotangent bundle $T^{\ast}_{\beta}B$
are meant to be Lee chords
with respect to $(\pi^{\ast}\beta \,,\, d_{\pi^{\ast}\beta} \lambda)$,
and all the actions are with respect to 
$(\pi^{\ast}\beta \,,\, d_{\pi^{\ast}\beta} \lambda, \lambda)$.
Recall from Example \ref{example: Lagrangians in twisted cotangent bundle}
that for any function $f: B \rightarrow \mathbb{R}$
we denote by $L_{f, \beta}$
the image of the Lagrangian embedding
$d_{\beta}f: B \rightarrow T_{\beta}^{\ast}B$.
It is an exact Lagrangian submanifold
of action $S_{f, \beta}$ given by
$S_{f, \beta} \big( d_{\beta}f (q) \big) = f(q)$.

\begin{lemma}\label{lemma: bijection critical points f}
For any function $f: B \rightarrow \mathbb{R}$,
there is a bijection between the critical points of $f$
and the essential Lee chords from the Lagrangian submanifold
$L_{f, \beta}$ of $T_{\beta}^{\ast}B$
to the zero section.
The action (hence the time-shift)
of any such Lee chord
is equal to the critical value
of the corresponding critical point of $f$.
Moreover,
the bijection sends non-degenerate critical points
to transverse Lee chords.
\end{lemma}

\begin{proof}
Suppose that $q \in B$ is a critical point of $f$.
Then
\[
d_{\beta}f (q) = df(q) - f(q) \beta(q)
= -f(q) \beta(q) \,,
\]
thus $\big( f(q), \gamma_q \big)$
with
\[
\gamma_q: [0, 1] \rightarrow T_{\beta}^{\ast}B \,,\;
\gamma_q (t) = - f(q) \, \beta(q) + t f(q) \, \beta(q)
\]
is a Lee chord from the point
$- f(q) \, \beta(q)$ of $L_{f, \beta}$
to the point $0_q$ of the zero section.
The action
\[
S_{f, \beta} \big(\gamma(0)\big)
- S_{0, \beta}\big(\gamma(1)\big) + \int_{\gamma_q} \lambda 
= S_{f, \beta} \big(\gamma(0)\big) = f(q)
\]
of the Lee chord $(f(q),\gamma_q)$
is equal to the time-shift,
and to the critical value
of the critical point $q$.
In particular,
the Lee chord $(f(q),\gamma_q)$ is essential.
The map
\begin{equation}\label{equation: bijection critical points f}
q \mapsto \big(f(q),\gamma_q\big)    
\end{equation}
from the set of critical points of $f$
to the set of essential Lee chords
from $L_{f, \beta}$ to the zero section
is injective by definition.
We now show that it is also surjective.
Suppose that $(T, \gamma)$ with
\[
\gamma: [0, 1] \rightarrow T_{\beta}^{\ast}B \,,\;
\gamma (t) = d_{\beta}f (q) + t \, T \, \beta(q)
\]
is an essential Lee chord from $L_{f, \beta}$ to the zero section.
Since the time-shift $T$ is equal to the action $f(q)$,
we have
$$
0_q = \gamma(1) = df(q) - f(q)\beta(q) + T\beta(q) = df(q) \,.
$$ 
Thus $q$ is a critical point of $f$,
and $\gamma = \gamma_q$.
This shows that the map
\eqref{equation: bijection critical points f}
is surjective,
hence bijective.

Finally we prove that a critical point $q$ is non-degenerate
if and only if the Lee chord
$\big(f(q),\gamma_q\big)$
is transverse.
By definition,
$\big(f(q),\gamma_q\big)$ is not transverse
if and only if
\[
(\varphi_{f(q)}^{\omega})_{\ast} \big(T_{d_{\beta}f(q)} L_{f, \beta} \big)
+ T_{0_q} 0_B \neq
T_{0_q} \, T^{\ast}B \,,
\]
thus if and only if
$(\varphi_{f(q)}^{\omega})_{\ast} \big(T_{d_{\beta}f(q)} L_{f, \beta} \big)
\cap T_{0_q} 0_B \neq \{0\}$,
and so if and only if there is $Y \in T_qB$ such that
$\big(\varphi_{f(q)}^{\omega} \circ d_{\beta}f\big)_{\ast} \, (Y) \in T_{0_q}0_B$.
Write $Y = \dot{\delta}(t)$
for a path $\delta: (-\varepsilon, \varepsilon) \to B$
with $\delta (0) = q$.
Then
\begin{align*}
\big(\varphi^\omega_{f(q)} \circ d_\beta f\big)_*(Y) = &
\left. \frac{d}{dt} \right\lvert_{t=0} \;
\varphi^\omega_{f(q)} \circ d_\beta f \, \big(\delta(t)\big) \\
= & \left. \frac{d}{dt} \right\lvert_{t=0} \;
df \big(\delta(t)\big) + \Big( f(q) - f \big(\delta(t)\big) \Big) \, \beta \big(\delta(t)\big) \\
= & \; Y + d^2f \, (Y, \;\cdot\;) \,,
\end{align*}
where the symbol $+$ in the last equality
refers to the splitting of $T_{0_q}T^*B$
into horizontal and vertical subspaces,
and where we have used that $q$ is a critical point of $f$.
This shows that $\big(f(q),\gamma_q\big)$ is not transverse
if and only if the Hessian of $f$ at $q$ is degenerate,
thus if and only if $q$ is a degenerate critical point of $f$.
\end{proof}

We will see in Section \ref{section: gf}
that a similar bijection holds
for all Lagrangian submanifolds of $T_{\beta}^{\ast}B$
Hamiltonian isotopic to the zero section
and their generating functions.
We now show that exact strict lcs diffeomorphisms
send essential Lee chords to essential Lee chords.

\begin{prop}\label{proposition: action invariant by strict exact}
Let $\psi: M_1 \rightarrow M_2$ be an exact lcs diffeomorphism
between two Liouville lcs manifolds
$\big( M_1, [(\eta_1, \omega_1, \lambda_1)] \big)$
and $\big( M_2, [(\eta_2, \omega_2, \lambda_2)] \big)$
that is strict with respect to
$(\eta_1, \omega_1)$ and $(\eta_2, \omega_2)$.
Let $L_1^{(j)}$ and $L_2^{(j)}$
be connected exact Lagrangian submanifolds
of $\big( M_j, [(\eta_j, \omega_j, \lambda_j)] \big)$
with $\psi (L_1^{(1)}) = L_1^{(2)}$ and $\psi (L_2^{(1)}) = L_2^{(2)}$,
and suppose that the pullback of $\eta_j$
by the inclusion of $L_i^{(j)}$ into $M_j$ is not exact.
Let $(T, \gamma)$ be a Lee chord with respect to $(\eta_1, \omega_1)$
from $L_1^{(1)}$ to $L_2^{(1)}$.
Then $( T, \psi \circ \gamma )$ is a Lee chord
with respect to $(\eta_2, \omega_2)$
from $L_1^{(2)}$ to $L_2^{(2)}$,
and the action of $( T, \psi \circ \gamma )$
with respect to $(\eta_2, \omega_2, \lambda_2)$
is equal to the action of $(T, \gamma)$
with respect to $(\eta_1, \omega_1, \lambda_1)$.
In particular,
$\psi$ sends essential Lee chords from $L_1^{(1)}$ to $L_2^{(1)}$
to essential Lee chords from $L_1^{(2)}$ to $L_2^{(2)}$.
\end{prop}

\begin{proof}
The fact that $( T, \psi \circ \gamma )$ is a Lee chord
with respect to $(\eta_2, \omega_2)$
from $L_1^{(2)}$ to $L_2^{(2)}$
just follows from the fact that $\psi^{\ast}\omega_2 = \omega_1$
and $\psi^{\ast}\eta_2 = \eta_1$
(since $\psi$ is strict).
Let $S_{\psi}$ be the action of $\psi$
with respect to $(\eta_1, \omega_1, \lambda_1)$,
thus $\psi^{\ast}\lambda_2 = \lambda_1 - d_{\eta_1} S_{\psi}$.
Then
\[
\int_{\gamma} \lambda_1
= \int_{\gamma} \psi^{\ast}\lambda_2 + \int_{\gamma} d_{\eta_1}S_{\psi}
= \int_{\psi \circ \gamma} \lambda_2 + \int_{\gamma} dS_{\psi}
= \int_{\psi \circ \gamma} \lambda_2 + S_{\psi} \big(\gamma(1)\big) - S_{\psi} \big(\gamma(0)\big) \,,
\]
where in the second equality
we have used that $\int_{\gamma} S_{\psi} \eta_1 = 0$,
because $\dot{\gamma} (t) = T \, R_{\omega_1} \big(\gamma(t)\big)$
and $\eta_1 (R_{\omega_1}) = 0$.
Moreover,
denoting by $S_{L_i^{(j)}}$ the action of $L_i^{(j)}$
with respect to $(\eta_j, \omega_j, \lambda_j)$,
we have
\[
S_{L_j^{(1)}} (p) = S_{L_j^{(2)}} \big(\psi(p)\big) + S_{\psi} (p)
\]
for all $p \in L_j^{(1)}$.
We thus conclude that
\[
\mathcal{A}_{L_1^{(1)}, L_2^{(1)}} (T, \gamma)
= S_{L_1^{(1)}} \big(\gamma(0)\big) - S_{L_2^{(1)}} \big(\gamma(1)\big) + \int_{\gamma}\lambda_1
\]
\[
= S_{L_1^{(2)}} \big( \psi \circ \gamma \, (0) \big)
- S_{L_2^{(2)}} \big( \psi \circ \gamma \, (1) \big)
+ \int_{\psi \circ \gamma}\lambda_2
= \mathcal{A}_{L_1^{(2)}, L_2^{(2)}} \big( T, \psi \circ \gamma\big) \,.
\]
\end{proof}

\begin{remark}\label{remark: AK}
Abouzaid and Kragh \cite{AK}
proved that
if $L$ is a compact exact Lagrangian submanifold
of the cotangent bundle $T^{\ast}B$
of a compact manifold $B$
then the restriction to $L$
of the projection $T^{\ast}B \rightarrow B$
is a simple homotopy equivalence.
This result implies that
if $\Lambda$ is a compact Legendrian submanifold of $J^1B$
that has no Reeb chords
(and so the projection of $\Lambda$ to $T^{\ast}B$
is an exact embedded Lagrangian submanifold)
then the restriction to $\Lambda$
of the projection $J^1B \rightarrow B$
is a simple homotopy equivalence.
A similar result also holds for compact exact Lagrangian submanifolds
of twisted cotangent bundles
that have no essential Lee chords.
Indeed,
let $L$ be an exact Lagrangian submanifold of $T_{\beta}^{\ast}B$,
with $i^{\ast}\lambda = d_{i^{\ast}(\pi^{\ast}\beta)} S$
for $S: L \rightarrow \mathbb{R}$,
where $i: L \hookrightarrow T^{\ast}_{\beta}B$ denotes the inclusion.
A Lee chord of $L$ is a pair $(T, \gamma)$
with $T \in \mathbb{R}$
and $\gamma: [0,1] \rightarrow T^{\ast}_{\beta}B$
a curve of the form $\gamma (t) = \sigma_q + tT \beta(q)$
with $\gamma(0)$ and $\gamma(1)$ in $L$.
We say that the Lee chord $(T, \gamma)$ is essential
if its action
\[
\mathcal{A}_L (T, \gamma)
= S \big(\gamma(0)\big) - S \big(\gamma(1)\big) + \int_{\gamma}\lambda
= S \big(\gamma(0)\big) - S \big(\gamma(1)\big)
\]
is equal to the time-shift $T$.
Consider the Legendrian submanifold
$\widetilde{L} = \{\, \big( \sigma_q , S(\sigma_q) \big) \,,\, \sigma_q \in L \,\}$
of $J^1_{\beta}B$,
and the Legendrian submanifold $U_{\beta} (\widetilde{L})$ of $J^1B$,
where $U_{\beta}$ is the untwisting map \eqref{equation: contactomorphism 1 jet}.
The map $(T, \gamma) \mapsto (T, \gamma')$
with
\[
\gamma': [0, 1] \rightarrow J^1B \,,\;
\gamma'(t) = \Big( \gamma(0) + S \big(\gamma(0)\big) \, \beta \big( \pi \circ \gamma (0) \big) \,,\,
S \big(\gamma(0)\big) - T + tT \Big)
\]
is then a bijection
between the essential Lee chords of $L$
and the Reeb chords of $U_{\beta} (\widetilde{L})$.
Thus,
if $L$ is compact and has no essential Lee chords
then $U_{\beta} (\widetilde{L})$ is a compact Legendrian submanifold of $J^1B$
that has no Reeb chords.
We thus conclude that the restriction to $U_{\beta} (\widetilde{L})$
of the projection $J^1B \rightarrow B$
is a simple homotopy equivalence,
and so the restriction to $L$
of the projection $T^{\ast}_{\beta}B \rightarrow B$
is also a simple homotopy equivalence.
This observation should be compared
to a similar result of Currier \cite[Theorem 1.3]{Currier}
for essential Liouville chords.
\end{remark}

Let $\varphi$ be a lcs diffeomorphism
of a lcs manifold $\big(M, \, [(\eta, \omega)]\big)$,
with conformal factor $g$ with respect to $(\eta, \omega)$.
We say that a point $p$ of $M$
is a translated point of $\varphi$
with respect to $(\eta, \omega)$
if $p$ and $\varphi(p)$ are in the same orbit
of the Lee flow of $(\eta, \omega)$
and $g(p) = 0$.
If there is only one Lee chord $(T, \gamma)$
from $\varphi(p)$ to $p$
then we say that the translated point $p$ of $\varphi$
has time-shift $T$.

\begin{example}\label{example: translated points symplectic}
Let $\varphi$ be a symplectomorphism
of a conformal symplectic manifold $(M, [\omega]) = \big(M, [(0, \omega)]\big)$. 
Since the conformal factor of $\varphi$ is zero
and the Lee flow of $(0, \omega)$ is the identity,
the translated points of $\varphi$ with respect to $(0, \omega)$
are the fixed points of $\varphi$.
More generally,
since the Lee flow of $(d\mu, e^{\mu} \omega)$
is the Hamiltonian flow of $e^{-\mu}$ with respect to $\omega$,
the translated points of $\varphi$
with respect to $(d\mu \,,\, e^{\mu}\omega)$
are the leafwise intersections of $\varphi$
with respect to $e^{-\mu}$,
i.e. the points $p$ such that $p$ and $\varphi(p)$
belong to the same orbit
of the Hamiltonian flow of $e^{-\mu}$. 
On the other hand,
if $\varphi$ is a conformal symplectomorphism
with $\varphi^{\ast} \omega = e^c \omega$
for $c \neq 0$
then $\varphi$ has no translated points
with respect to $(0, \omega)$,
because the conformal factor never vanishes.
Moreover,
$\varphi$ also has no translated points
with respect to any other representative
$(d\mu \,,\, e^{\mu}\omega)$.
Indeed,
since the conformal factor of $\varphi$
with respect to $(d\mu \,,\, e^{\mu}\omega)$
is $\mu \circ \varphi - \mu + c$,
if $p$ is a translated point of $\varphi$
with respect to $(d\mu \,,\, e^{\mu}\omega)$
then $p$ and $\varphi(p)$ are in the same Lee orbit
and $\mu \big(\varphi(p)\big) - \mu (p) = -c$.
But this is impossible,
because the Lee flow of $(d\mu \,,\, e^{\mu}\omega)$
preserves $\mu$.
\end{example}

\begin{example}\label{example: translated points contact}
Recall that a point $p$
of a contact manifold $\big(Y , \xi = \ker (\alpha)\big)$
is said to be a translated point of a contactomorphism $\phi$ with respect to $\alpha$
if $p$ and $\phi(p)$ are in the same orbit
of the Reeb flow of $\alpha$
and the conformal factor of $\phi$ with respect to $\alpha$
vanishes at $p$.
Let $\varphi = \widetilde{\phi}$ be the lift of $\phi$
to the locally conformal symplectization
$\big( S^1 \times Y \,,\, [( -d\theta \,,\, d_{-d\theta}\alpha )] \big)$.
Then $(\theta, p) \in S^1 \times Y$
is a translated point of ${\varphi}$
with respect to $( -d\theta \,,\, d_{-d\theta}\alpha )$
if and only if $p$ is a translated point
of $\phi$ with respect to $\alpha$.
In this case,
the Lee chords from ${\varphi} (\theta, p)$ to $(\theta, p)$
are of the form $\gamma = (\theta, \gamma_Y)$
with $\gamma_Y$ a Reeb chord from $\phi (p)$ to $p$.
\end{example}

Suppose now that $\varphi$ is a lcs Hamiltonian diffeomorphism
of a connected Liouville lcs manifold $\big(M , [(\eta, \omega, \lambda)]\big)$
with $\eta$ not exact.
Let $p$ be a translated point of $\varphi$,
and $(T, \gamma)$ a Lee chord from $\varphi(p)$ to $p$.
We define the (total) action
with respect to $(\eta, \omega, \lambda)$
of $(T, \gamma)$ by
\[
\mathcal{A}_{\varphi} (T, \gamma) = - S_{\varphi}(p) + \int_{\gamma} \lambda \,,
\]
where $S_{\varphi}$ is the action of $\varphi$
with respect to $(\eta, \omega, \lambda)$.
We say that $- S_{\varphi}(p)$ and $\int_{\gamma} \lambda$
are respectively the Hamiltonian and Lee actions of $(T, \gamma)$.
If there is only one Lee chord $(T, \gamma)$
from $\varphi (p)$ to $p$
then we say that $\mathcal{A}_\varphi (\gamma)$,
$- S_{\varphi}(p)$ and $\int_{\gamma} \lambda $
are the total, Hamiltonian and Lee actions
of the translated point $p$.
The Hamiltonian action always just depends
on the translated point $p$,
and not on the Lee chord from $\varphi (p)$ to $p$.
By Proposition \ref{proposition: Hamiltonian vs exact},
it can also be written as
\[
- S_{\varphi} (p) = - \int_0^1 e^{-g_t(p)} \Big( H_t \big(\varphi_t(p)\big)
- \lambda \big( X_{H_t} \big(\varphi_t(p)\big) \big) \Big) \, dt \,,
\]
where $\{\varphi_t\}_{t \in [0, 1]}$ is any lcs Hamiltonian isotopy
with $\varphi_1 = \varphi$
and $g_t$ and $H_t$ its conformal factors
and Hamiltonian function
with respect to $(\eta, \omega)$.
We say that a Lee chord for a translated point is essential
if it has time-shift equal to the action,
and that a translated point is essential
if it admits an essential Lee chord.

\begin{example}\label{example: action translated point symplectic}
Let $\varphi$ be a Hamiltonian symplectomorphism
of a conformal Liouville manifold
$\big(M, [(\omega, \lambda)] \big)$
$= \big(M, [(0, \omega, \lambda)]\big)$.
The action with respect to $(0, \omega, \lambda)$
of a translated (hence fixed) point $p$ of $\varphi$
is not well-defined,
because $S_{\varphi}$ is only defined up to the addition of a constant.
However,
if $\varphi$ is compactly supported
then we can normalize $S_{\varphi}$
by asking that it is also compactly supported.
The action of $p$,
which is equal to the Hamiltonian action $- S_{\varphi} (p)$,
coincides then with the usual (normalized) action of $p$
as a fixed point of $\varphi$.
\end{example}

\begin{example}
Let $\phi$ be a contactomorphism
of a contact manifold $\big(Y, \xi = \ker (\alpha)\big)$,
and $\varphi = \widetilde{\phi}$ its lift
to the locally conformal symplectization
$\big( S^1 \times Y \,,\, [( -d\theta \,,\, d_{-d\theta}\alpha \,,\, \alpha )] \big)$.
The Hamiltonian action of any translated point of ${\varphi}$ vanishes
because, by Example \ref{example: lift of maps to symplectization},
${\varphi}$ has action zero.
Thus,
for any translated point $(\theta, p)$ of ${\varphi}$,
the action of any Lee chord $\gamma = (\theta, \gamma_Y)$
from ${\varphi} (\theta, p) = \big(\theta, \phi(p)\big)$ to $(\theta, p)$
is equal to the Lee action,
which coincides with the action and time-shift
of the Reeb chord $\gamma_Y$ from $\phi(p)$ to $p$,
and to the time-shift of $\gamma$.
In particular,
all the Lee chords for the translated points of $\varphi$
are essential.
\end{example}

\begin{example}\label{example: action translated point in conformal symplectization}
Let $p$ be a translated point
of a lcs Hamiltonian diffeomorphism $\varphi$
of a locally conformal symplectization
$\big( S^1 \times Y \,,\, [( -d\theta \,,\, d_{-d\theta}\alpha \,,\, \alpha )] \big)$.
The time-shift of any Lee chord
from $\varphi(p)$ to $p$
is equal to the Lee action,
because $\lambda (R_{\omega}) = \alpha (R_{\alpha}) = 1$.
Thus, a Lee chord for $p$ is essential
if and only if $S_{\varphi} (p) = 0$.
\end{example}

We remark for completeness the following analogue
of Proposition \ref{proposition: action invariant by strict exact}
(even though we will not need it).

\begin{prop}\label{proposition: action invariant by conjugaiton by strict exact}
Let $\psi: M_1 \rightarrow M_2$ be an exact lcs diffeomorphism
between two Liouville lcs manifolds
$\big( M_1, [(\eta_1, \omega_1, \lambda_1)] \big)$
and $\big( M_2, [(\eta_2, \omega_2, \lambda_2)] \big)$
that is strict with respect to
$(\eta_1, \omega_1)$ and $(\eta_2, \omega_2)$,
and suppose that the Lee forms $\eta_1$ and $\eta_2$
are not exact.
Let $\varphi$ be a lcs Hamiltonian diffeomorphism
of $\big( M_1, [(\eta_1, \omega_1, \lambda_1)] \big)$,
$p$ a translated point of $\varphi$ with respect to $(\eta_1, \omega_1)$
and $(T, \gamma)$ a Lee chord for $p$.
Then $\psi(p)$ is a translated point of $\psi \circ \varphi \circ \psi^{-1}$
with respect to $(\eta_2, \omega_2)$,
and $(T, \psi \circ \gamma)$ is a Lee chord for $\psi(p)$
with
$$
\mathcal{A}_{\varphi} (T, \gamma)
= \mathcal{A}_{\psi \circ \varphi \circ \psi^{-1}} (T, \psi \circ \gamma) \,.
$$
In particular,
$p$ is an essential translated point of $\varphi$
if and only if $\psi (p)$ is an essential translated point
of $\psi \circ \varphi \circ \psi^{-1}$.
\end{prop}

\begin{proof}
The facts that $\psi(p)$ is a translated point
of $\psi \circ \varphi \circ \psi^{-1}$
with respect to $(\eta_2, \omega_2)$
and that $(T, \psi \circ \gamma)$ is a Lee chord for $\psi(p)$
just follow from the fact that $\psi^{\ast}\omega_2 = \omega_1$
and $\psi^{\ast}\eta_2 = \eta_1$
(since $\psi$ is strict).
Let $S_{\psi}$ be the action of $\psi$,
thus $\psi^{\ast}\lambda_2 = \lambda_1 - d_{\eta_1} S_{\psi}$.
As in the proof
of Proposition \ref{proposition: action invariant by strict exact},
\[
\int_{\gamma} \lambda_1
= \int_{\psi \circ \gamma} \lambda_2
+ S_{\psi} \big(\gamma(1)\big) - S_{\psi} \big(\gamma(0)\big)
= \int_{\psi \circ \gamma} \lambda_2 + S_{\psi} (p) - S_{\psi} \big(\varphi(p)\big) \,.
\]
Moreover,
a direct calculation shows that
$\psi \circ \varphi \circ \psi^{-1}$ has action
\[
S_{\psi \circ \varphi \circ \psi^{-1}}
= (S_{\varphi} - S_{\psi} + e^{-g} S_{\psi} \circ \varphi) \circ \psi^{-1} \,,
\]
where $S_{\varphi}$ and $g$ are the action and conformal factor of $\varphi$,
thus
\[
S_{\psi \circ \varphi \circ \psi^{-1}} \big( \psi(p) \big)
= S_{\varphi} (p) - S_{\psi} (p) + S_{\psi} \big(\varphi(p)\big) \,.
\]
We thus conclude that
\[
\mathcal{A}_{\varphi} (T, \gamma)
= - S_{\varphi} (p) + \int_{\gamma} \lambda_1
= - S_{\psi \circ \varphi \circ \psi^{-1}} \big( \psi(p) \big)
+ \int_{\psi \circ \gamma} \lambda_2
= \mathcal{A}_{\psi \circ \varphi \circ \psi^{-1}} (T, \psi \circ \gamma) \,.
\]
\end{proof}

Let $p$ be a translated point
of a lcs diffeomorphism $\varphi$
of a lcs manifold $\big( M , [(\eta, \omega)] \big)$,
and $(T, \gamma)$ a Lee chord
with respect to $(\eta, \omega)$ from $\varphi (p)$ to $p$.
We say that $(T, \gamma)$ is a transverse Lee chord
for the translated point $p$
if there is no non-zero vector $Y \in T_pM$
such that $(\varphi_T^{\omega} \circ \varphi)_{\ast} \, (Y) = Y$
and $dg (Y) = 0$,
where $g$ is the conformal factor of $\varphi$.
If there is only one Lee chord $(T, \gamma)$ from $\varphi (p)$ to $p$
then we say that the translated point $p$ is non-degenerate
if $(T, \gamma)$ is a transverse Lee chord.

\begin{prop}\label{proposition: bijection translated points Lee chords graph}
Let $\varphi$ be a divergence free lcs diffeomorphism
of a lcs manifold $\big( M , [(\eta, \omega)] \big)$.
Then $p$ is a translated point
of $\varphi$ with respect to $(\eta, \omega)$
if and only if there is a Lee chord in $M \boxtimes M$
with respect to $(\eta \boxtimes \eta, \omega \boxtimes \omega)$
from $\gr_{\lcs} (\varphi) (p)$
to $\gr_{\lcs} (\id) \big(\varphi(p)\big)$.
In this case,
there is a bijection between
the Lee chords in $M$ from $\varphi (p)$ to $p$
and the Lee chords in $M \boxtimes M$
from $\gr_{\lcs} (\varphi) (p)$
to $\gr_{\lcs} (\id) \big(\varphi(p)\big)$.
This bijection preserves the time-shift,
and sends transverse Lee chords for the translated point $p$
to transverse Lee chords from $\gr_{\lcs} (\varphi)$
to $\gr_{\lcs} (\id)$.
For Liouville lcs manifolds
and lcs Hamiltonian diffeomorphisms,
the bijection preserves moreover
the Hamiltonian and Lee actions,
and thus sends essential Lee chords for $p$
to essential Lee chords from $\gr_{\lcs} (\varphi) (p)$
to $\gr_{\lcs} (\id) \big(\varphi(p)\big)$.
\end{prop}

\begin{proof}
By Lemma \ref{lemma: Lee flow lcs product},
the Lee flow on $M \boxtimes M$
with respect to $(\eta \boxtimes \eta, \omega \boxtimes \omega)$
is the projection of the isotopy
$(x_1, x_2) \mapsto
\big( \bar{\varphi}^{\, \omega}_{-t} (x_1), x_2 \big)$
on $\bar{M} \times \bar{M}$,
where $\{\bar{\varphi}^{\, \omega}_t\}$ is the lift to $\bar{M}$
of the Lee flow $\{\varphi^{\, \omega}_t\}$ on $M$
with respect to $(\eta, \omega)$. Thus,
since $\gr_{\lcs} (\varphi): M \rightarrow M \boxtimes M$
is the map induced by the symplectic graph
$\gr (\bar{\varphi}): \bar{M} \rightarrow \bar{M} \times \bar{M}$,
there can be a Lee chord
from $\gr_{\lcs} (\varphi) (p)$ to $\gr_{\lcs} (\id) (x)$
only if $x = \varphi(p)$.
Since moreover $\{\bar{\varphi}^{\, \omega}_t\}$
is the Lee flow on $\bar{M}$
with respect to $(\pi^{\ast}\eta, \pi^{\ast}\omega)$,
for every $p \in M$ and $\bar{p} \in \bar{M}$
with $\pi (\bar{p}) = p$
there is a bijection between the Lee chords in $M \boxtimes M$
from $\gr_{\lcs} (\varphi) (p)$ to $\gr_{\lcs} (\id) \big(\varphi(p)\big)$
and the Lee chords in $\bar{M}$
with respect to $(\pi^{\ast}\eta, \pi^{\ast}\omega)$
from $\bar{\varphi} (\bar{p})$ to $\bar{p}$.
By Example \ref{example: Lee flow symplectic},
the Lee flow on $\bar{M}$
with respect to $(\pi^{\ast}\eta, \pi^{\ast}\omega)$
preserves $\mu$.
Thus, by Corollary \ref{corollary: symplectic lift},
there is a Lee chord with respect to $(\pi^{\ast}\eta, \pi^{\ast}\omega)$
from $\bar{\varphi} (\bar{p})$ to $\bar{p}$
if and only if there is a Lee chord with respect to $(\eta, \omega)$
from $\varphi (p)$ to $p$ and $g(p) = 0$,
where $g$ is the conformal factor of $\varphi$,
and so if and only if $p$ is a translated point of $\varphi$.
In this case,
there is a bijection between the Lee chords
with respect to $(\eta, \omega)$
from $\varphi (p)$ to $p$
and the Lee chords with respect to $(\pi^{\ast}\eta, \pi^{\ast}\omega)$
from $\bar{\varphi} (\bar{p})$ to $\bar{p}$,
hence with the Lee chords
with respect to $(\eta \boxtimes \eta, \omega \boxtimes \omega)$
from $\gr_{\lcs} (\varphi) (p)$ to $\gr_{\lcs} (\id) \big(\varphi(p)\big)$.
This bijection preserves the time-shift.
Suppose now that $\varphi$ is a lcs Hamiltonian diffeomorphism
of a Liouville lcs manifold $\big( M, [(\eta, \omega, \lambda)] \big)$,
and consider the lcs product $\big( M \boxtimes M,
[(\eta \boxtimes \eta, \omega \boxtimes \omega, \lambda \boxtimes \lambda)] \big)$.
It is immediate to see that the above bijection preserves the Lee actions.
The fact that it also preserves the Hamiltonian actions
follows from Lemma \ref{lemma: graph exact}.

Finally,
a Lee chord $(T,\gamma)$ in $M \boxtimes M$
from $\gr_{\lcs} (\varphi)(p)$ to $\gr_{\lcs} (\id) (\varphi(p))$ is transverse
if and only if the  corresponding Lee chord in $\bar{M} \times \bar{M}$
from $\gr (\bar{\varphi}) (\bar{p})$ to $\gr (\id) (\bar{\varphi} (\bar{p}))$
is transverse for some (hence any) $\bar{p}$ in $\pi^{-1}(p)$. 
This happens if and only if there is no non-zero vector
tangent to $\gr (\bar{\varphi})$
at $\gr (\bar{\varphi})(\bar{p}) = (\bar{p}, \bar{\varphi}(\bar{p}))$
whose image by $(\bar{\varphi}^\omega_{-T} \times \id)_{\ast}$
is tangent to $\gr (\id)$
at $\gr (\id)(\bar{\varphi}(\bar{p})) = (\bar{\varphi}(\bar{p}), \bar{\varphi}(\bar{p}))$,
thus no non-zero vector $X \in T_{\bar{p}}\bar{M}$ such that
$$
(\bar{\varphi}^\omega_{-T} \times \id)_* (X, \bar{\varphi}_* (X))
\in T_{\bar{\varphi}(\bar{p})}\gr(\id) \,.
$$
This is equivalent to
$(\bar{\varphi}^\omega_{-T})_* (X) =  \bar{\varphi}_* (X)$,
and so to $( \bar{\varphi}_T^{\omega} \circ \bar{\varphi})_{\ast} \, (X) = X$.
But $(\bar{\varphi}_T^{\omega} \circ \bar{\varphi})_{\ast} \, (X) = X$
if and only if
$(\varphi_T^{\omega} \circ \varphi)_{\ast} \, \big(\pi_{\ast}(X)\big) = \pi_{\ast} (X)$
and $d (\mu \circ \bar{\varphi}) (X) = d\mu (X)$.
The second equality is equivalent
to $dg \big(\pi_{\ast} (X)\big) = 0$,
because $\mu \circ \bar{\varphi} - \mu = g \circ \pi$
by Corollary \ref{corollary: symplectic lift}.
We conclude that $(T,\gamma)$ is a transverse Lee chord
from $\gr_{\lcs} (\varphi)(p)$ to $\gr_{\lcs} (\id)(\varphi(p))$
if and only if there is no non-zero vector $Y \in T_pM$
such that $(\varphi_T^{\omega} \circ \varphi)_{\ast} \, (Y) = Y$
and $dg (Y) = 0$,
and so if and only if the corresponding Lee chord from $\varphi (p)$ to $p$
is a transverse Lee chord for the translated point $p$ of $\varphi$.
\end{proof}

\begin{remark}\label{remark: cf bijection CS}
In \cite[Theorem 1.5 (3)]{CS} it is shown that
for any lcs Hamiltonian diffeomorphism $\varphi$
of a lcs manifold $\big(M, [(\eta, \omega)]\big)$
there is a bijection between
$\gr_{\lcs} (\varphi) \cap \gr_{\lcs} (\id)$
and the set of fixed points $p$ of $\varphi$
such that $\int_{\varphi_t(p)} \eta = 0$
for some (hence any) lcs Hamiltonian isotopy
$\{\varphi_t\}_{t \in [0, 1]}$ with $\varphi_1 = \varphi$.
Let $\{\bar{\varphi}_t\}$ be the lift of $\{\varphi_t\}$
to the symplectic cover $(\bar{M}, [e^{-\mu} \pi^{\ast}\omega] )$.
For any $\bar{p} \in \bar{M}$ with $\pi (\bar{p}) = p$
we have
$$
\int_{\varphi_t(p)} \eta
= \int_{\bar{\varphi}_t(\bar{p})} d\mu 
= \mu \circ \bar{\varphi} (\bar{p}) - \mu (\bar{p}) \,.
$$
By Corollary \ref{corollary: symplectic lift},
the condition $\int_{\varphi_t(p)} \eta = 0$
is thus equivalent to the vanishing
of the conformal factor of $\varphi$ at $p$.
We conclude that the bijection in \cite[Theorem 1.5 (3)]{CS}
can be reformulated
as a bijection between
$\gr_{\lcs} (\varphi) \cap \gr_{\lcs} (\id)$
and the set of translated points of $\varphi$
that are also fixed points,
and is thus consistent with
Proposition \ref{proposition: bijection translated points Lee chords graph}.
\end{remark}

The next proposition
will be generalized in Section \ref{section: gf}
to the $\mathcal{C}^0$-small case,
to obtain Theorem \ref{theorem: translated points - intro} (i).

\begin{prop}\label{proposition: existence C1 small}
Let $\varphi$ be the time-$1$ map
of a sufficiently $\mathcal{C}^1$-small
lcs Hamiltonian isotopy $\{\varphi_t\}$
of a compact Liouville lcs manifold $\big( M , [(\eta, \omega, \lambda)]\big)$
with $\eta$ not exact.
Then $\varphi$ has at least as many essential translated points
with respect to any representative $(\eta, \omega, \lambda)$
as the minimal number of critical points
of a function on $M$.
Moreover,
if all the essential Lee chords
for the essential translated points are non-degenerate
then the number of essential translated points
is greater than or equal to
the minimal number of critical points
of a Morse function on $M$.
\end{prop}

\begin{proof}
Let $\Delta \subset M \boxtimes M$
be the image of $\gr_{\lcs} (\id): M \rightarrow M \boxtimes M$,
and denote by $\beta$ the pullback of $\eta \boxtimes \eta$
by the inclusion $\Delta \hookrightarrow M \boxtimes M$.
By Lemma \ref{lemma: graph exact},
$\Delta$ is an exact Lagrangian submanifold
of $M \boxtimes M$.
By Theorem \ref{theorem: Weinstein}
and Remark \ref{remark: exact Weinstein},
there are a neighborhood $\mathcal{U}$
of $\Delta$ in $M \boxtimes M$,
a neighborhood $\mathcal{V}$ of the zero section in $T_{\beta}^{\ast}\Delta$
and an exact lcs diffeomorphism $\psi: \mathcal{U} \rightarrow \mathcal{V}$
sending $\Delta$ to the zero section
that is strict  with respect to $(\eta \boxtimes \eta, \omega \boxtimes \omega)$
and $(\pi^{\ast}\beta, d_{\pi^{\ast}\beta} \lambda)$.
Since $\{\varphi_t\}$ is $\mathcal{C}^1$-small,
$L_{\varphi_t} := \im \big(\gr_{\lcs} (\varphi_t)\big)$
is contained in $\mathcal{U}$ for all $t$
(in fact, it would be enough here to assume
that $\{\varphi_t\}$ is sufficiently $\mathcal{C}^0$-small).
By Lemma \ref{lemma: product of lcs diffeomorphisms},
$L_{\varphi} := L_{\varphi_1}$ is the image of $\Delta$
by the time-$1$ map of a lcs Hamiltonian isotopy $\{\chi_t\}$.
Since $\chi_t$ sends $\Delta$ to $L_{\varphi_t}$
and $L_{\varphi_t}$ is contained in $\mathcal{U}$,
we can assume that $\{\chi_t\}$ is supported in $\mathcal{U}$.
By Lemma \ref{lemma: Hamiltonian conjugation},
the conjugation $\{\psi \circ \chi_t \circ \psi^{-1}\}$
is also a lcs Hamiltonian isotopy.
Since the time-$1$ map of this lcs Hamiltonian isotopy
sends the zero section to $\psi(L_{\varphi})$,
by Corollary \ref{corollary: image of zero section is exact}
we deduce that $\psi(L_{\varphi})$
is an exact Lagrangian submanifold of $T_{\beta}^{\ast}\Delta$.
By the $\mathcal{C}^1$-smallness assumption,
$\psi (L_{\varphi})$ is also a section of $T_{\beta}^{\ast}\Delta$,
and so, by Example \ref{example: Lagrangians in twisted cotangent bundle},
$\psi(L_{\varphi}) = L_{f, \beta}$
for some function $f$ on $\Delta$.
Since there are no closed Lee orbits in $\mathcal{U}$,
by Proposition \ref{proposition: bijection translated points Lee chords graph},
Lemma \ref{lemma: bijection critical points f}
and Proposition \ref{proposition: action invariant by strict exact}
we obtain an injection from the set of critical point of $f$
to the set of essential translated points of $\varphi$
with respect to $(\eta, \omega, \lambda)$.
Since $\Delta$ is diffeomorphic to $M$,
we conclude that $\varphi$
has at least as many essential translated points
with respect to $(\eta, \omega, \lambda)$
as the minimal number of critical points
of a function on $M$.
If all the essential Lee chords
for the essential translated points of $\varphi$
are transverse then,
by Proposition \ref{proposition: bijection translated points Lee chords graph}
and Lemma \ref{lemma: bijection critical points f},
all the critical points of $f$ are non-degenerate.
We thus conclude that in this case 
$\varphi$ has at least as many essential translated points
with respect to $(\eta, \omega, \lambda)$
as the minimal number of critical points
of a Morse function on $M$.
\end{proof}

\section{Generating functions for Lagrangian submanifolds of twisted cotangent bundles}\label{section: gf}

Before describing generating functions
for Lagrangian submanifolds of twisted cotangent bundles,
we recall the classical notions of generating functions
for Lagrangian and Legendrian submanifolds
of usual cotangent and $1$-jet bundles.
We follow \cite{Viterbo, Bhupal, San11},
and invite the reader to consult these sources
and the references therein for more details.

A function $F: E \rightarrow \mathbb{R}$
defined on the total space of a trivial vector bundle
$E = B \times \mathbb{R}^N \rightarrow B$
is a \emph{generating function}
if the differential $dF: E \rightarrow T^{\ast}E$
is transverse to the subbundle $N_E^{\ast}$ of $T^{\ast}E$
formed by the covectors that vanish on vertical vectors.
In this case, the space
$$
\Sigma_F = dF^{-1} \big( N_E^{\ast} \cap \im (dF) \big)
$$
of fibre critical points of $F$
is a submanifold of $E$
of dimension equal to the dimension of $B$.
The map $i_F: \Sigma_F \rightarrow T^{\ast}B$
that associates to $e = (q, \zeta) \in \Sigma_F$
the covector $i_F(e) \in T_q^{\ast}B$
defined by $i_F(e) (X) = dF (\hat{X})$,
where $\hat{X}$ is any vector in $T_eE$
projecting to $X$,
is an exact Lagrangian immersion,
with $i_F^{\,\ast} \, \lambda = d ( \left. F \right\lvert_{\Sigma_F} )$.
Its lift to the $1$-jet bundle $J^1B = T^{\ast}B \times \mathbb{R}$,
endowed with the contact structure $\xi = \ker (\lambda - dz)$,
is the Legendrian immersion
\[
j_F: \Sigma_F \rightarrow J^1B \,,\;
e \mapsto \big( i_F(e), F(e) \big) \,.
\]
If $i_F: \Sigma_F \rightarrow T^{\ast}B$ is an embedding
we say that $F$ is a generating function
for the Lagrangian submanifold $\im (i_F)$ of $T^{\ast}B$.
The map $i_F$ then induces a bijection
between the critical points of $F$
and the intersections of $\im (i_F)$ with the zero section.
Similarly, if $j_F: \Sigma_F \rightarrow J^1B$ is an embedding
we say that $F$ is a generating function
for the Legendrian submanifold $\im (j_F)$ of $J^1B$.
The map $j_F$ then induces a bijection
between the critical points of $F$
and the Reeb chords from the zero section to $\im (j_F)$.

Suppose now that $B$ is compact.
A generating function
$F: E = B \times \mathbb{R}^N \rightarrow \mathbb{R}$
is said to be \emph{quadratic at infinity}
if there exists a non-degenerate quadratic form $F_{\infty}$ on $E$,
i.e.\ a map $F_{\infty}: E \rightarrow \mathbb{R}$
whose restriction to every fibre
is a non-degenerate quadratic form,
such that
$\partial_v (F - F_{\infty}): E \rightarrow E^{\ast}$
is bounded,
where $\partial_v$ denotes the vertical derivative.
It has been proved by Sikorav \cite{Sikorav87}
that every Lagrangian submanifold of $T^{\ast}B$
Hamiltonian isotopic to the zero section
has a generating function quadratic at infinity,
and by Chaperon \cite{Chaperon}
and Chekanov \cite{Chekanov}
that every Legendrian submanifold of $J^1B$
contact isotopic to the zero section
has a generating function quadratic at infinity.
Moreover, we have the following uniqueness result.
If $F: B \times \mathbb{R}^N \rightarrow \mathbb{R}$
is a generating function quadratic at infinity
for a Legendrian submanifold $\Lambda$ of $J^1B$
then for any fibre preserving diffeomorphim $\Phi$ of $B \times \mathbb{R}^N$
and non-degenerate quadratic form $Q$ on $\mathbb{R}^{N'}$
the composition $F \circ \Phi: B \times \mathbb{R}^N \rightarrow \mathbb{R}$
and the stabilization
$F \oplus Q: B \times \mathbb{R}^N \times \mathbb{R}^{N'} \rightarrow \mathbb{R}$
are also generating functions quadratic at infinity for $\Lambda$.
We consider the equivalence relation
on the space of generating functions quadratic at infinity
generated by these two operations:
two generating functions quadratic at infinity $F_1$ and $F_2$
are equivalent if there are non-degenerate quadratic forms $Q_1$ and $Q_2$
and a fibre preserving diffeomorphism $\Phi$
such that $F_1 \oplus Q_1 = (F_2 \oplus Q_2) \circ \Phi$.
It has been proved by Viterbo and Th\'eret
\cite{Viterbo, Theret, Theret_thesis}
that if $\Lambda$ is a Legendrian submanifold of $J^1B$
contact isotopic to the zero section
then all the generating functions quadratic at infinity for $\Lambda$
are equivalent
(similarly,
all \emph{normalized} generating functions quadratic at infinity
for a Lagrangian submanifold of $T^{\ast}B$
Hamiltonian isotopic to the identity
are equivalent).
Finally,
we recall that the existence theorems
are special cases of the following results:
if $L$ is a Lagrangian submanifold of $T^{\ast}B$
that has a generating function quadratic at infinity $F$,
then for any Hamiltonian isotopy $\{\varphi_t\}_{t \in [0, 1]}$
of $T^{\ast}B$
there is a 1-parameter family $F_t$
of generating functions quadratic at infinity
for $\{\varphi_t (L)\}$
such that $F_0$ is a stabilization of $F$ \cite{Sikorav87},
and if $\Lambda$ is a Legendrian submanifold of $J^1B$
that has a generating function quadratic at infinity $F$,
then for any contact isotopy $\{\phi_t\}_{t \in [0, 1]}$ of $J^1B$
there is a 1-parameter family $F_t$
of generating functions quadratic at infinity
for $\{\phi_t (\Lambda)\}$
such that $F_0$ is a stabilization of $F$
\cite{Chaperon, Chekanov}.

Following Chantraine and Murphy \cite{CM},
we now describe the lcs analogues
of these notions and results.
Let $F: E \rightarrow \mathbb{R}$ be a function
defined on the total space
of a trivial vector bundle $E = B \times \mathbb{R}^N \rightarrow B$,
and $\beta$ a closed $1$-form on $B$.
Notice first that the twisted differential $d_{\beta}F: E \rightarrow T^{\ast}E$,
where we still denote by $\beta$ its pullback
by the projection $B \times \mathbb{R}^N \rightarrow B$,
is transverse to $N_E^{\ast}$
if and only if the usual differential $dF: E \rightarrow T^{\ast}E$
is transverse to $N_E^{\ast}$,
and that in this case
the space
\[
\Sigma_{F, \beta} = ( d_{\pi^{\ast}\beta}F )^{-1}
\big( N_E^{\ast} \cap \im (d_{\pi^{\ast}\beta}F) \big)
\]
of twisted fibre critical points
coincides with the space $\Sigma_F$
of usual fibre critical points.
Assume thus that $dF: E \rightarrow T^{\ast}E$ is transverse to $N_E^{\ast}$,
and so $\Sigma _F = \Sigma_{F, \beta}$
is a submanifold of $E$
of dimension equal to the dimension of $B$.
Consider the map
\[
i_{F, \beta}: \Sigma_F \rightarrow T^{\ast}B
\]
that associates to $e = (q, \zeta) \in \Sigma_F$
the covector $i_{F, \beta}(e) \in T_q^{\ast}B$
defined by
\[
i_{F, \beta}(e) (X)
= d_{\beta}F (\hat{X})
= dF (e) (\hat{X}) - F(e) \, \beta (q) (X) \,,
\]
where $\hat{X}$ is any vector in $T_eE$
projecting to $X$. Then $i_{F, \beta}$ is an exact Lagrangian immersion
into the twisted cotangent bundle $T_{\beta}^{\ast}B$,
of action $\left. F \right\lvert_{\Sigma_F}$.
Indeed, by the tautological property of $\lambda$
we have
\begin{equation}\label{equation: pullback of lambda}
i_{F, \beta}^{\ast} \, \lambda
= d_{\beta} \big( \left. F \right\lvert_{\Sigma_F} \big) \,.
\end{equation}
If $i_{F, \beta}: \Sigma_F \rightarrow T^{\ast}_{\beta}B$
is an embedding
we say that $F$ is a generating function
for the exact Lagrangian submanifold
$L_{F, \beta} := \im (i_{F, \beta})$
of $T_{\beta}^{\ast}B$.
The map $i_{F, \beta}$ then induces a bijection
between the twisted critical points of $F$
(the zeros of $d_{\beta}F$)
and the intersections of $L_{F, \beta}$
with the zero section.
Moreover we have the following bijection,
which generalizes Lemma \ref{lemma: bijection critical points f}.

\begin{lemma}\label{lemma: bijection critical points F}
Let $F: E \rightarrow \mathbb{R}$ be a generating function
for a Lagrangian submanifold $L$ of $T_{\beta}^{\ast}B$.
Then there is a bijection between the critical points of $F$
and the essential Lee chords from $L$ to the zero section.
The action (hence the time-shift) of any such Lee chord
is equal to the critical value
of the corresponding critical point of $F$.
Moreover,
the bijection sends non-degenerate critical points
to transverse Lee chords.
\end{lemma}

\begin{proof}
By definition,
$L$ is the image of the exact Lagrangian embedding
$i_{F, \beta}: \Sigma_F \rightarrow T^{\ast}_{\beta}B$.
By \eqref{equation: pullback of lambda},
it has action
\[
S: L \rightarrow \mathbb{R} \,,\;
S \big( i_{F, \beta} (e) \big) = F (e) \,.
\]
Suppose that $e = (q, \zeta) \in \Sigma_F$ is a critical point of $F$.
Then
\[
i_{F, \beta} (e) = - F(e) \, \beta (q) \,,
\]
thus $\big( F(e), \gamma_e \big)$ with
\[
\gamma_e: [0, 1] \rightarrow T_{\beta}^{\ast}B \,,\;
\gamma_e(t) = - F(e) \, \beta (q) + t \, F(e) \, \beta (q)
\]
is a Lee chord from the point $- F(e) \, \beta (q)$ of $L$
to the point $0_q$ of the zero section.
The action
\[
S \big(\gamma(0)\big) - S \big(\gamma(1)\big) + \int_{\gamma_e} \lambda
= S \big(\gamma(0)\big) = F(e)
\]
of the Lee chord $\big( F(e), \gamma_e \big)$
is equal to the time-shift,
and to the critical value of the critical point $e$.
In particular,
the Lee chord $\big( F(e), \gamma_e \big)$ is essential.
The map
\begin{equation}\label{equation: bijection critical point F}
e \mapsto \big( F(e), \gamma_e \big)    
\end{equation}
from the set of critical points of $F$
to the set of essential Lee chords
from $L$ to the zero section
is injective,
because $i_{F, \beta}$ is an embedding.
We now show that it is also surjective.
Suppose that $(T, \gamma)$ with
\[
\gamma: [0, 1] \mapsto T_{\beta}^{\ast}B \,,\;
\gamma (t) = \sigma_q + t \, T \, \beta (q)
\]
is an essential Lee chord from $L$ to the zero section.
Since $\sigma_q \in L$,
there is $e = (q, \zeta) \in \Sigma_F$
with $i_{F, \beta} (e) = \sigma_q$.
Thus
\[
dF (e) (\hat{X})
= \sigma_q (X) + F(e) \beta(q) (X)
\]
for all $X \in T_qB$ and $\hat{X} \in T_eE$ projecting to $X$.
Since $(T, \gamma)$ is essential,
the time-shift $T$
is equal to the action $F(e)$.
Since moreover $\gamma(1) = \sigma_q + T \beta(q)$ belongs to the zero section,
we thus have $dF(e) (\hat{X}) = 0$,
and so $e$ is a critical point of $F$.
Moreover, $\sigma_q = - T \beta(q) = - F(e) \beta(q)$
and so $(T, \gamma) = \big( F(e), \gamma_e \big)$.
This shows that the map \eqref{equation: bijection critical point F}
is surjective,
hence a bijection.

Finally,
we show that a critical point $e$ is non-degenerate
if and only if the Lee chord $\big( F(e), \gamma_e \big)$ is transverse.
By definition,
the Lee chord $\big( F(e), \gamma_e \big)$ is not transverse
if and only if
\[
(\varphi_{F(e)}^{\omega})_{\ast} \big(T_{i_{F, \beta}(e)} L \big)
+ T_{0_q} 0_B \neq T_{0_q} \, T^{\ast}B \,,
\]
thus if and only if
$(\varphi_{F(e)}^{\omega})_{\ast} \big(T_{i_{F, \beta}(e)} L \big)
\cap T_{0_q} 0_B \neq \{0\}$,
and so
if and only if there is $Y \in T_e \Sigma_F$
such that
\begin{equation}\label{equation: transversality}
(\varphi_{F(e)}^{\omega} \circ i_{F, \beta})_{\ast} \, (Y) \in T_{0_q} 0_B \,.
\end{equation}
Write $Y = \dot{\delta}(t)$
for a path $\delta: (-\varepsilon, \varepsilon) \to \Sigma_F$
with $\delta (0) = e$.
Set $\delta_o = \pr_1 \circ \,\delta$,
where $\pr_1 : \Sigma_F \subset E = B \times \mathbb{R}^N \to B$
is the projection to the first factor.
Observe that $N_E^*$ is canonically isomorphic
to the pullback bundle $\pr_1^{\,\ast}\, T^*B$,
and denote by $p: N_E^* \to T^*B$ the natural projection.
We can then write $i_{F,\beta} (e) = p \, (d_\beta F(e))$ and
\begin{align*}
(\varphi_{F(e)}^{\omega} \circ i_{F, \beta})_{\ast} \, (Y)
&=  \left. \frac{d}{dt} \right \lvert_{t=0}
p \, \Big( d_\beta F \big(\delta(t)\big) \Big) + F(e) \, \beta \big(\delta_o(t)\big) \\
&= \left.\frac{d}{dt} \right \lvert_{t=0}
p \, \Big( dF \big(\delta(t)\big) \Big)
+ \Big(F(e) - F \big(\delta(t)\big)\Big) \, \beta \big(\delta_o(t)\big) \,.
\end{align*}
We observe that $\delta_1 (t) := p \, \left(dF(\delta(t))\right)$
and $\delta_2 (t) := \big( F(e) - F(\delta(t)) \big) \, \beta(\delta_o(t))$
are two paths in $T^*B$ with $\delta_1 (0) = \delta_2(0) = 0_q$.
Therefore,
their tangent vectors at $t = 0$
both belong to $T_{0_q} T^*B$.
The latter vector space admits a canonical decomposition
as the direct sum of $T_q B$ -- the ``horizontal" part -- and $T^*_qB$ -- the ``vertical" part.
The decompositions of $\left. \frac{d}{dt} \right\lvert_{t=0} \delta_1(t)$
and $\left. \frac{d}{dt} \right\lvert_{t=0} \delta_2(t)$
into horizontal and vertical parts are
$$
\left.\frac{d}{dt}\right\lvert_{t=0} \delta_1 (t)
= p_{*} \big( Y + d^2F (Y, \,\cdot\,) \big)
= {\pr_1}_{*} (Y) + p_{\ast} \big( d^2F(Y, \,\cdot\,) \big)
$$
and
$$
\left.\frac{d}{dt}\right\lvert_{t=0} \delta_2 (t)
= {\pr_1}_{*} (Y) + \left.\frac{d}{dt}\right\lvert_{t=0}
\Big( F(e) - F \big(\delta(q)\big) \Big) \, \beta(q)
= {\pr_1}_{*} (Y) + 0_q \,,
$$
where we have used that $e$ is a critical point of $F$.
The condition \eqref{equation: transversality}
is equivalent to asking that the sum of the vertical parts
of $\left. \frac{d}{dt} \right\lvert_{t=0} \delta_1(t)$
and $\left. \frac{d}{dt} \right\lvert_{t=0} \delta_2(t)$
vanishes.
Since this sum
is $p_{\ast} \big(d^2F (Y, \,\cdot\,)\big)$,
we conclude that 
the Lee chord $\big( F(e), \gamma_e \big)$ is not transverse
if and only if there is $Y \in T_e\Sigma_F$
such that $p_{\ast} \big(d^2F (Y, \,\cdot\,)\big)$ vanishes,
hence $d^2F (Y, \,\cdot\,)$ vanishes,
and so if and only if $e$ is a degenerate critical point of $F$.
\end{proof}

By \eqref{equation: pullback of lambda},
the lift of $i_{F, \beta}: \Sigma_F \rightarrow T_{\beta}^{\ast}B$
to the contactization $J_{\beta}^1B$
is the Legendrian immersion
\[
\widetilde{i_{F, \beta}}: \Sigma_F \rightarrow J_{\beta}^1B \,,\;
\widetilde{i_{F, \beta}}(e) = \big( i_{F, \beta}(e) \,,\, F(e) \big) \,.
\]
We have
\[
U_{\beta} \circ \widetilde{i_{F, \beta}} = j_F \,,
\]
where $U_{\beta}$ is the untwisting map \eqref{equation: contactomorphism 1 jet}
from $J^1_{\beta}B$ to $J^1B$
and
\[
j_F: \Sigma_F \rightarrow J^1B \,,\;
j_F(e) = \big( i_F(e), F(e) \big)
\]
the Legendrian immersion generated (in the usual sense) by $F$.
This implies the following result.

\begin{prop}\label{proposition: relation gf 1 jet}
A function $F: B \times \mathbb{R}^N \rightarrow \mathbb{R}$
is a generating function for a Lagrangian submanifold $L$ of $T_{\beta}^{\ast}B$
if and only if it is a generating function for the Legendrian submanifold
$U_{\beta} (\widetilde{L})$ of $J^1B$,
where $\widetilde{L}$ is the lift of $L$
to the contactization $J_{\beta}^1B$
and $U_{\beta}$ the untwisting map \eqref{equation: contactomorphism 1 jet}
from $J^1_{\beta}B$ to $J^1B$.
\end{prop}

Using Proposition \ref{proposition: relation gf 1 jet}
(and still following \cite{CM}),
we can deduce the following existence and uniqueness results
from the corresponding ones on $1$-jet bundles.

\begin{thm}\label{theorem: gf}
Let $B$ be a compact manifold,
and $\beta$ a closed $1$-form on $B$.
\begin{itemize}
\item[(i)] If $L$ is a Lagrangian submanifold of $T^{\ast}_{\beta} B$
that is lcs Hamiltonian isotopic to the zero section
then $L$ has a generating function quadratic at infinity,
unique up to equivalence.
\item[(ii)] If $L$ is a Lagrangian submanifold of $T^{\ast}_{\beta} B$
that has a generating function quadratic at infinity $F$
then for any lcs Hamiltonian isotopy
$\{\varphi_t\}_{t \in [0,1]}$ of $T^{\ast}_{\beta} B$
there is a $1$-parameter family $F_t: E \rightarrow \mathbb{R}$
of generating functions quadratic at infinity for $\{\varphi_t(L)\}$
such that $F_0$ is a stabilization of $F$.
\end{itemize}
\end{thm}

\begin{proof}
Suppose that $L$ is a Lagrangian submanifold of $T^{\ast}_{\beta} B$
that is lcs Hamiltonian isotopic to the zero section.
By Corollary \ref{corollary: image of zero section is exact},
$L$ is exact.
Let $\widetilde{L}$ be its lift to $J_{\beta}^1B$,
and consider the image $U_{\beta} (\widetilde{L})$
by the untwisting map $U_{\beta}: J_{\beta}^1B \rightarrow J^1B$.
Then $U_{\beta} (\widetilde{L})$ is a Legendrian submanifold of $J^1B$,
contact isotopic to the zero section.
Indeed,
it is the image of the zero section
by the time-$1$ map of the contact isotopy
$\{ U_{\beta} \circ \widetilde{\varphi_t} \circ U_{\beta}^{-1} \}$,
where $\{\widetilde{\varphi_t}\}$
is the lift to $J^1_{\beta}B$
of the lcs Hamiltonian isotopy $\{\varphi_t\}$
whose time-$1$ map sends the zero section to $L$.
By \cite{Chaperon, Chekanov, Viterbo, Theret},
$U_{\beta} (\widetilde{L})$ has a generating function quadratic at infinity,
unique up to equivalence.
By Proposition \ref{proposition: relation gf 1 jet},
we thus conclude that $L$ has a generating function quadratic at infinity,
unique up equivalence.
This proves (i).

Suppose now that $L$ is a Lagrangian submanifold
of $T^{\ast}_{\beta} B$
that has a generating function quadratic at infinity $F$,
and that $\{\varphi_t\}_{t \in [0,1]}$
is a lcs Hamiltonian isotopy of $T^{\ast}_{\beta} B$.
Consider the lifts $\widetilde{L}$ and $\{\widetilde{\varphi_t}\}$
of $L$ and $\{\varphi_t\}$ to $J^1_{\beta}B$.
Then $\widetilde{\varphi_t} (\widetilde{L}) = \widetilde{\varphi_t(L)}$,
thus
\[
U_{\beta} \circ \widetilde{\varphi_t} \circ U_{\beta}^{-1}
\big( U_{\beta} (\widetilde{L}) \big)
= U_{\beta} (\widetilde{\varphi_t(L)}) \,.
\]
By Proposition \ref{proposition: relation gf 1 jet},
$F$ is a generating functions quadratic at infinity
for $U_{\beta} (\widetilde{L})$.
By \cite{Chaperon, Chekanov},
there is a $1$-parameter family
of generating functions quadratic at infinity
$F_t: E \rightarrow \mathbb{R}$
for $\{ U_{\beta} \big( \widetilde{\varphi_t (L)} \big) \}$
such that $F_0$ is a stabilization of $F$.
By Proposition \ref{proposition: relation gf 1 jet},
$F_t$ is then also a $1$-parameter family
of generating functions quadratic at infinity
for $\{\varphi_t(L)\}$.
This proves (ii).
\end{proof}

We can now prove Theorem \ref{theorem: lcs Arnold conjecture intro}
and Theorem \ref{theorem: translated points - intro} (i).

\begin{proof}[Proof of Theorem \ref{theorem: lcs Arnold conjecture intro}]
Let $B$ be a connected compact manifold,
$\beta$ a closed non-exact $1$-form on $B$,
and $L$ a Lagrangian submanifold
of the twisted cotangent bundle $T^{\ast}_{\beta}B$
that is lcs Hamiltonian isotopic to the zero section.
By Theorem \ref{theorem: gf} (i),
$L$ has a generating function quadratic at infinity
$F: B \times \mathbb{R}^N \rightarrow \mathbb{R}$,
and by Lemma \ref{lemma: bijection critical points F}
there is a bijection between the (non-degenerate) critical points of $F$
and the (transverse) essential Lee chords
from $L$ to the zero section.
Since the number of critical points
of a function quadratic at infinity over $B$
is greater than or equal to the cup-length of $B$,
and greater than or equal to the sum of the Betti numbers of $B$
if all the critical points are non-degenerate,
we conclude that
the number of essential Lee chords
from $L$ to the zero section
is greater than or equal to the cup-length of $B$,
and greater than or equal to the sum of the Betti numbers of $B$
if all essential Lee chords are transverse.
\end{proof}

\begin{proof}[Proof of Theorem \ref{theorem: translated points - intro} (i)]
Let $\varphi$ be the time-$1$ map
of a sufficiently $\mathcal{C}^0$-small
lcs Hamiltonian isotopy $\{\varphi_t\}$
of a compact Liouville lcs manifold $\big( M , [(\eta, \omega, \lambda)]\big)$
with $\eta$ not exact.
We need to show that the number of essential translated points of $\varphi$
with respect to any representative $(\eta, \omega)$
is greater than or equal to the cup-length of $M$,
and greater than or equal to the sum of the Betti numbers of $M$
if all the essential Lee chords for the essential translated points
are non-degenerate.
As in the proof of Proposition \ref{proposition: existence C1 small},
$L_{\varphi} = \im \big(\gr_{\lcs} (\varphi) \big)$ is contained
in a Weinstein neighborhood of $\Delta = \im \big(\gr_{\lcs} (\id) \big)$
in $M \boxtimes M$,
and its image by the strict exact lcs diffeomorphism
given by Theorem \ref{theorem: Weinstein}
and Remark \ref{remark: exact Weinstein}
is a Lagrangian submanifold of $T_{\beta}^{\ast} \Delta$,
where $\beta$ is the pullback of $\eta \boxtimes \eta$
by the inclusion $\Delta \hookrightarrow M \boxtimes M$,
lcs Hamiltonian isotopic to the zero section.
We thus conclude using Theorem \ref{theorem: lcs Arnold conjecture intro},
Proposition \ref{proposition: action invariant by strict exact}
and Proposition \ref{proposition: bijection translated points Lee chords graph}.
\end{proof}

\begin{remark}
The conclusions of Theorem \ref{theorem: translated points - intro} (i)
also hold (with the same proof,
without using Proposition \ref{proposition: action invariant by strict exact})
for translated points of the time-$1$ map
of a sufficiently $\mathcal{C}^0$-small
lcs Hamiltonian isotopy
of any (not necessarily Liouville) compact lcs manifold.
However,
the estimates are less pertinent in this case
because translated points are generically not isolated.
\end{remark}

\section{Generating functions for lcs diffeomorphisms of $S^1 \times \mathbb{R}^{2n+1}$
and $S^1 \times \mathbb{R}^{2n} \times S^1$}\label{section: lcs product etc}

In \cite{Viterbo, Bhupal, San11},
Viterbo, Bhupal and the third author
associate generating functions quadratic at infinity
to compactly supported Hamiltonian diffeomorphisms of $\mathbb{R}^{2n}$
and to compactly supported contactomorphisms
of $\mathbb{R}^{2n+1}$ and $\mathbb{R}^{2n} \times S^1$
contact isotopic to the identity.
In this section we describe a lcs analogue of these constructions,
to associate generating functions quadratic at infinity
to compactly supported lcs Hamiltonian diffeomorphisms
of $S^1 \times \mathbb{R}^{2n+1}$ and $S^1 \times \mathbb{R}^{2n} \times S^1$.
We start by recalling the constructions
in the symplectic and contact cases.

Consider the Euclidean space $\mathbb{R}^{2n}$,
with coordinates $(x_1, y_1, \cdots, x_n, y_n)$,
endowed with the standard symplectic form
$\omega_0 = \sum_{j=1}^n dx_j \wedge dy_j$.
The map
\[
\tau_{\symp}: \overline{\mathbb{R}^{2n}} \times \mathbb{R}^{2n}
\rightarrow T^{\ast} \mathbb{R}^{2n}
\]
defined by
\[
\tau_{\symp} \, (x, y, X, Y) =
\Big(\frac{ x + X}{2} \,,\,
\frac{ y + Y}{2}
\,,\, y -  Y \,,\, X - x\Big)
\]
is a symplectomorphism
between the symplectic product of $\mathbb{R}^{2n}$ 
and the cotangent bundle $T^{\ast} \mathbb{R}^{2n}$,
which sends the graph of the identity
to the zero section.
Let $\varphi$ be a symplectomorphism of $\mathbb{R}^{2n}$,
and consider its graph
\[
\gr (\varphi): \mathbb{R}^{2n} \rightarrow
\mathbb{R}^{2n} \times \mathbb{R}^{2n}  \,,\;
\gr(\varphi)(p) = \big( p, \varphi(p) \big)
\]
and the Lagrangian submanifold
$L_{\varphi} := \im \big( \tau_{\symp} \circ \gr(\varphi) \big)$
of $T^{\ast}\mathbb{R}^{2n}$.
If $F: \mathbb{R}^{2n} \times \mathbb{R}^N \rightarrow \mathbb{R}$
is a generating function for $L_{\varphi}$
then the critical points of $F$
are in bijection with the fixed points of $\varphi$,
with critical value given by the action
of the corresponding fixed point.
Suppose now that $\varphi$
is a compactly supported Hamiltonian symplectomorphism\footnote{
By a compactly supported Hamiltonian symplectomorphism
we always mean a symplectomorphism
that is the time-$1$ map of a compactly supported Hamiltonian isotopy.}.
In order to apply the existence and uniqueness theorems
for generating functions quadratic at infinity,
the Lagrangian submanifold $L_{\varphi}$
of $T^{\ast}\mathbb{R}^{2n}$
is compactified to a Lagrangian submanifold
$\overline{L_{\varphi}}$ of $T^{\ast} S^{2n}$
as follows.
Let $\{\varphi_t\}_{t \in [0,1]}$
be a compactly supported Hamiltonian isotopy with $\varphi_1 = \varphi$,
and consider the compactly supported Hamiltonian isotopy $\{\Psi_{\varphi_t}\}$
of $T^{\ast} \mathbb{R}^{2n}$
that is the conjugation by $\tau_{\symp}$
of the compactly supported Hamiltonian isotopy
$\{ \id \times \varphi_t \}$
of $\overline{\mathbb{R}^{2n}} \times \mathbb{R}^{2n}$.
Then $L_{\varphi}$ is the image of the zero section by $\Psi_{\varphi_1}$.
Since $L_{\varphi}$ coincides with the zero section
outside a compact set,
it extends to a Lagrangian submanifold
$\overline{L_{\varphi}}$ of $T^{\ast} S^{2n}$.
Moreover,
$\{\Psi_{\varphi_t}\}$ extends to a Hamiltonian isotopy of $T^{\ast}S^{2n}$
whose time-$1$ map sends the zero section to $\overline{L_{\varphi}}$.
Thus,
$\overline{L_{\varphi}}$ is a Lagrangian submanifold of $T^{\ast} S^{2n}$
Hamiltonian isotopic to the zero section.
We say that $\overline{F}: S^{2n} \times \mathbb{R}^N \rightarrow \mathbb{R}$
is a generating function for $\varphi$
if it is a generating function for $\overline{L_{\varphi}}$.
The induced function $F: \mathbb{R}^{2n} \times \mathbb{R}^N \rightarrow \mathbb{R}$
(which has the same critical values as $\overline{F}$)
is then a generating function for $L_{\varphi}$,
thus the set of critical values of $\overline{F}$
coincides with the action spectrum of $\varphi$.
By \cite{Sikorav87, Viterbo, Theret}
any compactly supported Hamiltonian symplectomorphism of $\mathbb{R}^{2n}$
has a generating function quadratic at infinity,
unique up to fibre preserving diffeomorphism and stabilization.

Consider now the Euclidean space $\mathbb{R}^{2n+1}$,
with coordinates $(x_1, y_1, \cdots, x_n, y_n, z)$,
endowed with the contact structure $\xi_0$
given by the kernel of the contact form
$\alpha_0 = \sum_{j=1}^n \frac{x_idy_j - y_jdx_j}{2} - dz$,
and the quotient $\mathbb{R}^{2n} \times S^1$,
where $S^1 = \mathbb{R}/\mathbb{Z}$,
with the induced contact form and contact structure,
still denoted by $\alpha_0$ and $\xi_0$.
The map
\[
\tau_{\cont}: \mathbb{R}^{2n+1} \times \mathbb{R}^{2n+1} \times \mathbb{R}
\rightarrow J^1 \mathbb{R}^{2n+1}
\]
defined by
\begin{gather*}
\tau_{\cont} \, (x, y, z, X, Y, Z, \rho) =
\\
\Big(\frac{ x + e^{-\frac{\rho}{2}}X}{2} \,,\,
\frac{ y + e^{-\frac{\rho}{2}}Y}{2}
\,,\, Z \,,
y - e^{-\frac{\rho}{2}} Y \,,\,
e^{-\frac{\rho}{2}} X - x\,,\,
1 - e^{-\rho} \,,\,
Z - z 
+ \frac{e^{-\frac{\rho}{2}} (yX - xY)}{2}\Big)
\end{gather*}
is a strict contactomorphism\footnote{
Several identifications
of the contact product of $\mathbb{R}^{2n+1}$
with $J^1\mathbb{R}^{2n+1}$
are used in the literature
\cite{Bhupal, San10, San11, FSZ},
and different choices are made
for the contact forms on $\mathbb{R}^{2n+1}$
and on the contact product
(cf.\ footnote 1).
In all these references everything could have been done
without any inconvenience
using the choices of the present article.}
between the contact product
of $\mathbb{R}^{2n+1}$ with respect to $\alpha_0$
and the 1-jet bundle $J^1\mathbb{R}^{2n+1}$,
which sends the contact graph of the identity
to the zero section.
Let $\phi$ be a contactomorphism of $\mathbb{R}^{2n+1}$
with conformal factor $g$ with respect to $\alpha_0$,
and consider its contact graph
\[
\gr (\phi): \mathbb{R}^{2n+1} \rightarrow
\mathbb{R}^{2n+1} \times \mathbb{R}^{2n+1} \times \mathbb{R} \,,\;
\gr(\phi)(p) = \big( p, \phi(p), g(p) \big)
\]
and the Legendrian submanifold
$\Lambda_{\phi} := \im \big( \tau_{\cont} \circ \gr(\phi) \big)$
of $J^1\mathbb{R}^{2n+1}$.
If $F: \mathbb{R}^{2n+1} \times \mathbb{R}^N \rightarrow \mathbb{R}$
is a generating function for $\Lambda_{\phi}$
then the critical points of $F$
are in bijection with the translated points of $\phi$,
with critical value given by the action
of the corresponding translated point.
Suppose now that $\phi$
is a compactly supported contactomorphism
contact isotopic to the identity\footnote{
By a compactly supported contactomorphism
contact isotopic to the identity
we always mean a contactomorphism
that is isotopic to the identity
through a compactly supported contact isotopy.}.
In order to apply the existence and uniqueness theorems
for generating functions quadratic at infinity,
the Legendrian submanifold $\Lambda_{\phi}$
of $J^1\mathbb{R}^{2n+1}$
is compactified to a Legendrian submanifold
$\overline{\Lambda_{\phi}}$ of $J^1S^{2n+1}$
as follows.
Let $\{\phi_t\}_{t \in [0,1]}$
be a compactly supported contact isotopy with $\phi_1 = \phi$
and conformal factors $g_t$,
and consider the compactly supported contact isotopy
$\{\Psi_{\phi_t}\}$ of $J^1 \mathbb{R}^{2n+1}$
that is the conjugation by $\tau_{\cont}$
of the compactly supported contact isotopy
$\{\Phi_t\}$ of $\mathbb{R}^{2n+1} \times \mathbb{R}^{2n+1} \times \mathbb{R}$
defined by
\[
\Phi_t (p, P, \rho)
= \big( p \,,\, \phi_t (P) \,,\, \rho + g_t (P) \big) \,.
\]
Then $\Lambda_{\phi}$ is the image of the zero section by $\Psi_{\phi_1}$.
Since $\Lambda_{\phi}$ coincides with the zero section
outside a compact set,
it extends to a Legendrian submanifold
$\overline{\Lambda_{\phi}}$ of $J^1 S^{2n+1}$.
Moreover,
$\{ \Psi_{\phi_t} \}$ extends to a contact isotopy of $J^1S^{2n+1}$
whose time-$1$ map sends the zero section to $\overline{\Lambda_{\phi}}$.
Thus, $\overline{\Lambda_{\phi}}$ is a Legendrian submanifold of $J^1S^{2n+1}$
contact isotopic to the zero section.
We say that $\overline{F}: S^{2n+1} \times \mathbb{R}^N \rightarrow \mathbb{R}$
is a generating function for $\phi$
if it is a generating function for $\overline{\Lambda_{\phi}}$.
The induced function
$F: \mathbb{R}^{2n+1} \times \mathbb{R}^N \rightarrow \mathbb{R}$
(which has the same critical values as $\overline{F}$)
is then a generating function for $\Lambda_{\phi}$,
thus the set of critical values of $\overline{F}$
coincides with the action spectrum of $\phi$.
Similarly,
if $\phi$ a compactly supported contactomorphism
of $\mathbb{R}^{2n} \times S^1$
contact isotopic to the identity
then the Legendrian submanifold $\Lambda_{\phi_{\mathbb{R}}}$
of $J^1\mathbb{R}^{2n+1}$
associated to the lift $\phi_{\mathbb{R}}$
of $\phi$ to $\mathbb{R}^{2n+1}$
(the contactomorphism of $\mathbb{R}^{2n+1}$
that projects to $\phi$
and is the identity over the complement of the support of $\phi$)
extends to a Legendrian submanifold of $J^1 (S^{2n} \times \mathbb{R})$
contact isotopic to the zero section.
The last descends to a Legendrian submanifold
$\overline{\Lambda_{\phi}} := \overline{\Lambda_{\phi_{\mathbb{R}}}}$
of $J^1 (S^{2n} \times S^1)$
contact isotopic to the zero section,
indeed
\[
\tau_{\cont} (x, y, z+1, X, Y, Z+1, \rho)
= \tau_{\cont} (x, y, z, X, Y, Z, \rho) + (0, 0, 1, 0, 0, 0, 0) \,.
\]
We say that
$\overline{F}: S^{2n} \times S^1 \times \mathbb{R}^N \rightarrow \mathbb{R}$
is a generating function for $\phi$
if it is a generating function for $\overline{\Lambda_{\phi}}$.
The induced function
$F: \mathbb{R}^{2n+1} \times \mathbb{R}^N \rightarrow \mathbb{R}$
(which has the same critical values as $\overline{F}$)
is then a generating function for $\Lambda_{\phi_{\mathbb{R}}}$,
thus the set of critical values of $\overline{F}$
coincides with the action spectrum of $\phi$
(which is defined to be the action spectrum of $\phi_{\mathbb{R}}$).

By \cite{Chaperon, Chekanov, Viterbo, Theret}
any compactly supported contactomorphism
of $\mathbb{R}^{2n+1}$ or $\mathbb{R}^{2n} \times S^1$
contact isotopic to the identity
has a generating function quadratic at infinity,
unique up to fibre preserving diffeomorphism and stabilization.

In the lcs case,
the analogue of the identifications $\tau_{\symp}$ and $\tau_{\cont}$
is obtained as follows.
By Lemma \ref{lemma: identification twisted cotangent bundle with conformal symplectization}
and Proposition \ref{proposition: lcs product}
we have strict lcs diffeomorphisms
\begin{equation}\label{equation: identification jet}
T^{\ast}_{-d\theta} (S^1 \times \mathbb{R}^{2n+1})
\rightarrow S^1 \times J^1 \mathbb{R}^{2n+1}
\end{equation}
and 
\begin{equation}\label{equation: identification product}
(S^1 \times \mathbb{R}^{2n+1}) \boxtimes (S^1 \times \mathbb{R}^{2n+1})
\rightarrow S^1 \times ( \mathbb{R}^{2n+1} \times \mathbb{R}^{2n+1} \times \mathbb{R} ) \,.
\end{equation}
By Example \ref{example: lift of maps to symplectization},
the lift 
\[
\widetilde{\tau_{\cont}}:
S^1 \times (\R^{2n+1} \times \R^{2n+1} \times \R) \to S^1 \times J^1\R^{2n+1} \,,\;
(\theta, p) \mapsto (\theta, \tau_{\cont}(p))
\]
of the strict contactomorphism
$\tau_{\cont}:
\mathbb{R}^{2n+1} \times \mathbb{R}^{2n+1} \times \mathbb{R}
\rightarrow J^1 \mathbb{R}^{2n+1}$
to the locally conformal symplectizations
is a strict lcs diffeomorphism.
Composing it with the strict lcs diffeomorphisms
\eqref{equation: identification jet} and \eqref{equation: identification product}
we obtain a strict lcs diffeomorphism
\begin{equation}\label{equation: tau}
\tau: (S^1 \times \mathbb{R}^{2n+1}) \boxtimes (S^1 \times \mathbb{R}^{2n+1})
\rightarrow T^{\ast}_{-d\theta} (S^1 \times \mathbb{R}^{2n+1}) \,,
\end{equation}
defined by
\begin{gather*}
\tau \big( [(\theta, x, y, z, \Theta, X, Y, Z)] \big) =
\\
\Big([\theta], \frac{ x + e^{\frac{\Theta - \theta}{2}} X}{2} \,,\,
\frac{ y + e^{\frac{\Theta - \theta}{2}} Y}{2}
\,,\, Z \,,\,
Z - z
+ \frac{e^{\frac{ \Theta - \theta}{2}} (yX - xY)}{2} \,,\,
y - e^{\frac{\Theta - \theta}{2}} Y \,,\,
e^{\frac{\Theta - \theta}{2}} X - x \,,\,
1 - e^{\Theta - \theta} \Big) \,.
\end{gather*}
Denote by $\lambda_{S^1 \times \mathbb{R}^{2n+1}}$
the tautological $1$-form on the cotangent bundle of $S^1 \times \mathbb{R}^{2n+1}$.
A direct calculation shows that
\begin{equation}\label{tau*lambda}
\tau^{\ast} \, \lambda_{S^1 \times \mathbb{R}^{2n+1}}
= \alpha_0 \boxtimes \alpha_0
- d_{-d\theta} \Big( z - Z + \frac{e^{\frac{\Theta - \theta}{2}} (xY - yX) }{2} \Big) \,.
\end{equation}
The lcs diffeomorphism $\tau$ is thus exact,
with action
\[
S_{\tau} = z - Z + \frac{e^{\frac{\Theta - \theta}{2}} (xY - yX) }{2} \,.
\]

Consider now a lcs Hamiltonian diffeomorphism $\varphi$
of $S^1 \times \mathbb{R}^{2n+1}$.
Its lcs graph is the Lagrangian embedding
\[
\gr_{\lcs} (\varphi): S^1 \times \mathbb{R}^{2n+1} \rightarrow
(S^1 \times \mathbb{R}^{2n+1}) \boxtimes (S^1 \times \mathbb{R}^{2n+1})
\]
defined by
\[
\gr_{\lcs} (\varphi) ([\theta], x, y, z)
= [(\theta, x, y, z, \bar{\varphi} (\theta, x, y, z))] \,,
\]
where $\bar{\varphi}$ is the lift of $\varphi$
to the symplectic cover $\mathbb{R} \times \mathbb{R}^{2n+1}$.
We define
\begin{equation*}
\Gamma_{\varphi}: S^1 \times \mathbb{R}^{2n+1} \rightarrow
T^{\ast}_{-d\theta} (S^1 \times \mathbb{R}^{2n+1})
\end{equation*}
to be the composition of $\gr_{\lcs}(\varphi)$ with $\tau$,
thus
$$
\Gamma_{\varphi} ([\theta], x, y, z) =
$$
\begin{equation*}
\Big( [\theta] \,,\, \frac{ x + e^{\frac{\bar{\varphi}_{\theta} - \theta}{2}} \bar{\varphi}_x}{2} \,,\,
\frac{ y + e^{\frac{\bar{\varphi}_{\theta} - \theta}{2}} \bar{\varphi}_y}{2}
\,,\, \bar{\varphi}_{z} \,,\, \bar{\varphi}_{z} - z
+ \frac{e^{\frac{ \bar{\varphi}_{\theta} - \theta}{2}}
(y\bar{\varphi}_x - x\bar{\varphi}_y)}{2} \,,\,
y - e^{\frac{\bar{\varphi}_{\theta} - \theta}{2}} \bar{\varphi}_y \,,\,
e^{\frac{\bar{\varphi}_{\theta} - \theta}{2}} \bar{\varphi}_x - x \,,\,
1 - e^{\bar{\varphi}_{\theta} - \theta} \Big) \,,    
\end{equation*}
where $\bar{\varphi}_{\theta}$, $\bar{\varphi}_x$, $\bar{\varphi}_y$ and $\bar{\varphi}_z$
are the components of $\bar{\varphi}$,
evaluated at the point $(\theta, x, y, z)$.
By Corollary \ref{corollary: symplectic lift},
\begin{equation}\label{equation: corollary symplectic lift Euclidean}
\theta - \bar{\varphi}_{\theta} (\theta, x, y, z)
= g ([\theta], x, y, z) \,,
\end{equation}
where $g$ is the conformal factor of $\varphi$.
Moreover, 
$\bar{\varphi}_x = \varphi_x$, $\bar{\varphi}_y = \varphi_y$
and $\bar{\varphi}_z = \varphi_z$,
where $\varphi_{\theta}$, $\varphi_x$, $\varphi_y$, $\varphi_z$
denote the components of $\varphi$.
Thus
\begin{equation}\label{equation: formula for associated Lagrangian}
\Gamma_{\varphi} ([\theta], x, y, z) =
\end{equation}
\[
\Big( [\theta] \,,\, \frac{ x + e^{-\frac{1}{2}g} \varphi_x}{2} \,,\,
\frac{ y + e^{-\frac{1}{2}g} \varphi_y}{2}
\,,\, \varphi_{z} \,,\, 
\varphi_z - z + \frac{e^{-\frac{1}{2}g} (y\varphi_x - x\varphi_y)}{2} \,,\,
y - e^{-\frac{1}{2}g} \varphi_y \,,\,
e^{-\frac{1}{2}g} \varphi_x - x \,,\,
1 - e^{-g} \Big) \,.
\]
The relation \eqref{tau*lambda}
and Lemma \ref{lemma: graph exact} give
\begin{equation}\label{equation: action associated Lagrangian embedding}
\Gamma_{\varphi}^{\;\ast} \, \lambda_{S^1 \times \mathbb{R}^{2n+1}}
= d_{-d\theta} \Big( \varphi_z - z
+ \frac{ e^{-\frac{1}{2}g} (y \varphi_x - x \varphi_y) }{2} - S_{\varphi} \Big) \,,
\end{equation}
where $S_{\varphi}$ is the action of $\varphi$.
The Lagrangian embedding $\Gamma_{\varphi}$ is thus exact,
with action
$$
S_{\Gamma_{\varphi}} = \varphi_z - z
+ \frac{ e^{-\frac{1}{2}g} (y \varphi_x - x \varphi_y)}{2}  - S_{\varphi} \,.   
$$
Its image
\[
L_{\varphi} = \im (\Gamma_{\varphi})
\]
is an exact Lagrangian submanifold
of $T^{\ast}_{-d\theta} (S^1 \times \mathbb{R}^{2n+1})$.

\begin{prop}\label{proposition: bijection Lee chords translated points}
Let $\varphi$ be lcs Hamiltonian diffeomorphism of $S^1 \times \mathbb{R}^{2n+1}$,
and consider the associated Lagrangian submanifold
$L_{\varphi}$ of $T^{\ast}_{-d\theta} (S^1 \times \mathbb{R}^{2n+1})$.
Then there is a bijection between
the translated points of $\varphi$
and the Lee chords from $L_{\varphi}$ to the zero section.
Moreover,
this bijection sends non-degenerate translated points
to transverse Lee chords,
and essential translated points to essential Lee chords.
\end{prop}

\begin{proof}
Since there are no closed Lee orbits
in $S^1 \times \mathbb{R}^{2n+1}$
and $T^{\ast}_{-d\theta} (S^1 \times \mathbb{R}^{2n+1})$
and since $\tau$ is a strict exact lcs diffeomorphism,
by Propositions \ref{proposition: action invariant by strict exact}
and \ref{proposition: bijection translated points Lee chords graph}
there is a bijection between the translated points of $\varphi$
and the Lee chords from $L_{\varphi}$ to the zero section,
sending non-degenerate translated points to transverse Lee chords
and essential translated points to essential Lee chords.
\end{proof}

We say that
$F: S^1 \times \mathbb{R}^{2n+1} \times \mathbb{R}^N \rightarrow \mathbb{R}$
is a generating function for a lcs Hamiltonian diffeomorphism
$\varphi$ of $S^1 \times \mathbb{R}^{2n+1}$
if it is a generating function
for the associated Lagrangian submanifold
$L_{\varphi}$ of $T^{\ast}_{-d\theta} (S^1 \times \mathbb{R}^{2n+1})$.
Then $i_{F, -d\theta}: \Sigma_F \rightarrow
T_{-d\theta}^{\ast} (S^1 \times \mathbb{R}^{2n+1})$
induces a diffeomorphism between $\Sigma_F$ and $L_{\varphi}$.
We consider the composition
\[
\begin{tikzcd}
{\Phi_F:\, S^1 \times \mathbb{R}^{2n+1} } \arrow[r, "\Gamma_{\varphi}"]
& L_{\varphi} \arrow[r] &[3ex] \Sigma_F
\end{tikzcd}
\]
of the inverse of this diffeomorphism with $\Gamma_{\varphi}$.

\begin{lemma}\label{lemma: bijection critical points explicit}
For $(p, \zeta) \in \Sigma_F$
and $([\theta], x, y, z) = (\Phi_F)^{-1} (p, \zeta)$
we have
\[
p = \Big( [\theta] \,,\, \frac{x + e^{-\frac{1}{2}g} \varphi_x }{2} \,,\,
\frac{y + e^{-\frac{1}{2}g} \varphi_y }{2} \,,\, \varphi_z \Big)
\]
and
\[
\begin{cases}
d_{p_{\theta}} F \, (p, \zeta)
= \varphi_z - z + \frac{ e^{- \frac{1}{2}g} (y \varphi_x - x \varphi_y) }{2} - F(p, \zeta)
= S_{\varphi} ([\theta], x, y, z) \\
d_{p_x} F \, (p, \zeta) = y - e^{-\frac{1}{2}g} \varphi_y \\
d_{p_y} F \, (p, \zeta) = e^{-\frac{1}{2}g} \varphi_x - x \\
d_{p_z} F \, (p, \zeta) = 1 - e^{-g} \,,
\end{cases}
\]
where $g$, $\varphi_x$, $\varphi_y$ and $\varphi_z$
are evaluated at $([\theta], x, y, z)$.
\end{lemma}

\begin{proof}
By assumption we have
\[
\Gamma_{\varphi} ([\theta], x, y, z) = i_{F, -d\theta} \, (p, \zeta) \,,
\]
thus all the equalities
except the one involving $S_{\varphi}$
follow from the equality
\[
i_{F, -d\theta} \, (p, \zeta) =
\Big( p \,,\, \frac{\partial F}{\partial p} (p, \zeta) + F(p, \zeta) \, d\theta \Big)
\]
and formula \eqref{equation: formula for associated Lagrangian}.
For the equality involving $S_{\varphi}$,
using \eqref{equation: action associated Lagrangian embedding}
and \eqref{equation: pullback of lambda}
we have
\[
d_{-d\theta}
\Big( \varphi_z - z + \frac{ e^{- \frac{1}{2}g} (y \varphi_x - x \varphi_y) }{2} - S_{\varphi} \Big)
= \Gamma_{\varphi}^{\ast} \, \lambda_{S^1 \times \mathbb{R}^{2n+1}}
\]
\[
= ( i_{F, -d\theta} \circ \Phi_F )^{\ast} \, \lambda_{S^1 \times \mathbb{R}^{2n+1}}
= \Phi_F^{\ast} \big( d_{-d\theta} (\left. F \right\lvert_{\Sigma_F}) \big)
= d_{-d\theta} (\left. F \right\lvert_{\Sigma_F} \circ \Phi_F) \,.
\]
By Remark \ref{remark: unique solution} we thus have
\[
\varphi_z - z + \frac{ e^{- \frac{1}{2}g} (y \varphi_x - x \varphi_y) }{2}
- S_{\varphi} ([\theta], x, y, z)
= F(p, \zeta) \,,
\]
which is equivalent to the desired equality.
\end{proof}

Using Lemma \ref{lemma: bijection critical points explicit}
we obtain the following result.

\begin{prop}\label{proposition: bijection critical points F translated point varphi}
A point $e \in \Sigma_F$ is a (non-degenerate) critical point of $F$
if and only if $(\Phi_F)^{-1} (e)$ is an essential
(non-degenerate) translated point of $\varphi$.
In this case,
the action (hence the time-shift)
of the translated point $\Phi_F^{-1} (e)$ of $\varphi$
is equal to the critical value of the critical point $e$ of $F$.
\end{prop}

\begin{proof}
The relations in Lemma \ref{lemma: bijection critical points explicit}
and Example \ref{example: action translated point in conformal symplectization}
imply that $e \in \Sigma_F$ is a critical point of $F$
if and only if $(\Phi_F)^{-1} (e)$ is an essential translated point of $\varphi$,
and that the critical value $F(e)$
is equal to the action (hence the time-shift)
of the translated point $(\Phi_F)^{-1} (e)$.
The bijection $e \mapsto (\Phi_F)^{-1} (e)$
between the critical points of $F$
and the essential translated points of $\varphi$
coincides with the one induced by
Lemma \ref{lemma: bijection critical points F}
and Proposition \ref{proposition: bijection Lee chords translated points},
and thus sends non-degenerate critical points
to non-degenerate translated points.
\end{proof}

Lemma \ref{lemma: bijection critical points explicit}
is also used to obtain the following result.

\begin{prop}\label{proposition: difference function}
Let $\varphi_1$ and $\varphi_2$ be lcs Hamiltonian diffeomorphisms
of $S^1 \times \mathbb{R}^{2n+1}$,
with generating functions 
$F_1: S^1 \times \mathbb{R}^{2n+1} \times \mathbb{R}^{N_1} \rightarrow \mathbb{R}$
and $F_2: S^1 \times \mathbb{R}^{2n+1} \times \mathbb{R}^{N_2} \rightarrow \mathbb{R}$
respectively.
Then there is a bijection between
the critical points of the difference function
\[
F_1 - F_2:
S^1 \times \mathbb{R}^{2n+1} \times \mathbb{R}^{N_1} \times \mathbb{R}^{N_2}
\rightarrow \mathbb{R} \,,\;
(F_1 - F_2) (p, \zeta_1, \zeta_2)
= F_1 (p, \zeta_1) - F_2 (p, \zeta_2)
\]
and the essential translated points of $\varphi_2^{-1} \circ \varphi_1$.
Moreover,
the critical value of a critical point of $F_1 - F_2$
is equal to the action (hence the time-shift)
of the corresponding essential translated point
of $\varphi_2^{-1} \circ \varphi_1$.
\end{prop}

\begin{proof}
Suppose that $(p, \zeta_1, \zeta_2)$
is a critical point of $F_1 - F_2$.
Then 
$\frac{\partial F_1}{\partial p} (p, \zeta_1)
= \frac{\partial F_2}{\partial p} (p, \zeta_2)$ and $\frac{\partial F_1}{\partial \zeta_1} (p, \zeta_1)
= \frac{\partial F_2}{\partial \zeta_2} (p, \zeta_2) = 0$.
In particular, $(p, \zeta_1)$ and $(p, \zeta_2)$
are fibre critical points of $F_1$ and $F_2$ respectively.
Let
$$
([\theta_j], x_j, y_j, z_j) = (\Phi_{F_j})^{-1} (p, \zeta_j) \,.
$$
The relations in Lemma \ref{lemma: bijection critical points explicit}
and \eqref{equation: corollary symplectic lift Euclidean}
imply that $g_1 ([\theta_1], x_1, y_1, z_1) = g_2 ([\theta_2], x_2, y_2, z_2)$,
$\varphi_1 ([\theta_1], x_1, y_1, z_1) = \varphi_2 ([\theta_2], x_2, y_2, z_2)$,
$[\theta_1] = [\theta_2]$,
$x_1 = x_2$, $y_1 = y_2$
and $z_2 - z_1 = F_1 (p, \zeta_1) - F_2 (p, \zeta_2)$.
We thus see that
\[
([\theta_2], x_2, y_2, z_2)
= \varphi_2^{-1} \circ \varphi_1 \, ([\theta_1], x_1, y_1, z_1) \,,
\]
that there is a Lee chord
from $([\theta_1], x_1, y_1, z_1)$ to $([\theta_2], x_2, y_2, z_2)$
of time-shift equal to the critical value
$(F_1 - F_2) (p, \zeta_1, \zeta_2)$,
and that the conformal factor $- g_2 \circ \varphi_2^{-1} \circ \varphi_1 + g_1$
of $\varphi_2^{-1} \circ \varphi_1$ vanishes at $([\theta_1], x_1, y_1, z_1)$.
We thus conclude that $([\theta_1], x_1, y_1, z_1)$
is a translated point of $\varphi_2^{-1} \circ \varphi_1$
of time-shift equal to $(F_1 - F_2) (p, \zeta_1, \zeta_2)$.
Moreover,
Lemma \ref{lemma: bijection critical points explicit}
also gives
\[
S_{\varphi_1} ([\theta_1], x_1, y_1, z_1) = S_{\varphi_2} ([\theta_2], x_2, y_2, z_2) \,.
\]
Since the action of $\varphi_2^{-1} \circ \varphi_1$
is $S_{\varphi_1}
- e^{g_2 \circ \varphi_2^{-1} \circ \varphi_1 - g_1} \,
S_{\varphi_2} \circ \varphi_2^{-1} \circ \varphi_1$,
we conclude that the translated point $([\theta_1], x_1, y_1, z_1)$
of $\varphi_2^{-1} \circ \varphi_1$
has Hamiltonian action zero
and so, by Example \ref{example: action translated point in conformal symplectization},
is an essential translated point of $\varphi_2^{-1} \circ \varphi_1$.

Conversely,
suppose that $([\theta], x, y, z)$
is an essential translated point of $\varphi_2^{-1} \circ \varphi_1$.
Set $([\theta_1], x_1, y_1, z_1) = ([\theta], x, y, z)$
and $([\theta_2], x_2, y_2, z_2) = \varphi_2^{-1} \circ \varphi_1 ([\theta], x, y, z)$.
Then $g_1 ([\theta_1], x_1, y_1, z_1) = g_2 ([\theta_2], x_2, y_2, z_2)$,
$[\theta_1] = [\theta_2]$,
$x_1 = x_2$, $y_1 = y_2$,
$\varphi_1 ([\theta_1], x_1, y_1, z_1) = \varphi_2 ([\theta_2], x_2, y_2, z_2)$
and
$$
S_{\varphi_1} ([\theta_1], x_1, y_1, z_1) = S_{\varphi_2} ([\theta_2], x_2, y_2, z_2) \,.
$$
Let $(p_j, \zeta_j) = \Phi_{F_j} ([\theta_j], x_j, y_j, z_j)$.
The relations in Lemma \ref{lemma: bijection critical points explicit}
imply that $p_1 = p_2 =: p$
and $\frac{\partial F_1}{\partial p} (p, \zeta_1)
= \frac{\partial F_2}{\partial p} (p, \zeta_2)$,
thus $(p, \zeta_1, \zeta_2)$
is a critical point of $F_1 - F_2$,
and that the critical value
$(F_1 - F_2) (p, \zeta_1, \zeta_2) = F_1 (p, \zeta_1) - F_2 (p, \zeta_2)$
is equal to the time-shift $z_2 - z_1$.
\end{proof}

Similarly to the symplectic and contact cases,
in order to apply the existence and uniqueness results
for generating functions quadratic at infinity
(Theorem \ref{theorem: gf})
we need to compactify the Lagrangian submanifold
associated to a compactly supported lcs Hamiltonian diffeomorphism\footnote{
By a compactly supported lcs Hamiltonian diffeomorphism
we always mean the time-$1$ map
of a compactly supported lcs Hamiltonian isotopy.}
$\varphi$ of $S^1 \times \mathbb{R}^{2n+1}$
to a Lagrangian submanifold of the twisted cotangent bundle of a compact manifold.
This can be done as follows.
Let $\{\varphi_t\}$ be a compactly supported lcs Hamiltonian isotopy
with $\varphi_1 = \varphi$,
and consider the compactly supported lcs Hamiltonian isotopy $\{ \Psi_{\varphi_t} \}$
that is the conjugation by $\tau$
of the compactly supported lcs Hamiltonian isotopy $\{ \id \boxtimes\, \varphi_t \}$
of $(S^1 \times \mathbb{R}^{2n+1}) \boxtimes (S^1 \times \mathbb{R}^{2n+1})$
given by Lemma \ref{lemma: product of lcs diffeomorphisms}.
Then $L_{\varphi}$ is the image of the zero section by $\Psi_{\varphi_1}$.
Since $L_{\varphi}$ coincides with the zero section
outside a compact set,
it extends to a Lagrangian submanifold
$\overline{L_{\varphi}}$ of $T_{-d\theta}^{\ast} (S^1 \times S^{2n+1})$.
Moreover,
$\{\Psi_{\varphi_t}\}$ extends to a lcs Hamiltonian isotopy
of $T_{-d\theta}^{\ast} (S^1 \times S^{2n+1})$
whose time-$1$ map sends the zero section to $\overline{L_{\varphi}}$.
Thus,
$\overline{L_{\varphi}}$ is a Lagrangian submanifold
of $T_{-d\theta}^{\ast} (S^1 \times S^{2n+1})$
lcs Hamiltonian isotopic to the zero section.
We say that
$\overline{F}: S^1 \times S^{2n+1} \times \mathbb{R}^N \rightarrow \mathbb{R}$
is a generating function for $\varphi$
if it is a generating function for $\overline{L_{\varphi}}$.

Consider now a compactly supported lcs Hamiltonian diffeomorphism
$\varphi$ of $S^1 \times \mathbb{R}^{2n} \times S^1$.
Denote by $\varphi_{\mathbb{R}}$
the lift of $\varphi$ to $S^1 \times \mathbb{R}^{2n+1}$,
i.e.\ the lcs Hamiltonian diffeomorphism of $S^1 \times \mathbb{R}^{2n+1}$
that projects to $\varphi$
and is the identity over the complement of the support of $\varphi$.
We say that
$F: S^1 \times \mathbb{R}^{2n+1} \times \mathbb{R}^N \rightarrow \mathbb{R}$
is a generating function for $\varphi$
if it is a generating function for the Lagrangian submanifold
$L_{\varphi_{\mathbb{R}}}$ of $T_{-d\theta}^{\ast} (S^1 \times \mathbb{R}^{2n+1})$.
By similar arguments as before,
the Lagrangian submanifold $L_{\varphi_{\mathbb{R}}}$
of $T_{-d\theta}^{\ast} (S^1 \times \mathbb{R}^{2n+1})$
extends to a Lagrangian submanifold
of $T_{-d\theta}^{\ast} (S^1 \times S^{2n} \times \mathbb{R})$
lcs Hamiltonian isotopic to the zero section.
Moreover,
the last descends to a Lagrangian submanifold $\overline{L_{\varphi}}$
of $T_{-d\theta}^{\ast} (S^1 \times S^{2n} \times S^1)$
lcs Hamiltonian isotopic to the zero section,
indeed
\[
\tau \big( [(\theta, x, y, z+1, \Theta, X, Y, Z+1)] \big)
= \tau \big( [(\theta, x, y, z, \Theta, X, Y, Z)] \big)
+ (0, 0, 0, 1, 0, 0, 0, 0) \,.
\]
We say that
$\overline{F}: S^1 \times S^{2n} \times S^1 \times \mathbb{R}^N \rightarrow \mathbb{R}$
is a generating function for $\varphi$
if it is a generating function for $\overline{L_{\varphi}}$.

Applying Theorem \ref{theorem: gf}
we obtain the following result.

\begin{prop}\label{proposition: gf lcs diffeomorphisms}
Let $\varphi$ be a compactly supported lcs Hamiltonian diffeomorphism
of $S^1 \times \mathbb{R}^{2n+1}$ or $S^1 \times \mathbb{R}^{2n} \times S^1$.
Then $\varphi$ has a generating function quadratic at infinity
$\overline{F}: S^1 \times B \times \mathbb{R}^N \rightarrow \mathbb{R}$,
where $B$ is either $S^{2n+1}$ or $S^{2n} \times S^1$,
which is unique up to equivalence.
\end{prop}

We define the action spectrum
of a lcs Hamiltonian diffeomorphism of $S^1 \times \mathbb{R}^{2n+1}$
to be the set of actions (hence time-shifts)
of its essential translated points.
The translated points of a lcs Hamiltonian diffeomorphisms
$\varphi$ of $S^1 \times \mathbb{R}^{2n} \times S^1$
do not have a well-defined time-shift,
because the Lee flow is periodic
and so for any translated point $p$ of $\varphi$
there is a $\mathbb{Z}$-family
of Lee chords from $\varphi (p)$ to $p$.
However,
if $\varphi$ is compactly supported
then it has a unique lift $\varphi_{\mathbb{R}}$
that is the identity over the complement of the support of $\varphi$.
Any point $\bar{p}$ of $S^1 \times \mathbb{R}^{2n+1}$
projecting to $p$
is a translated point of $\varphi_{\mathbb{R}}$.
Moreover,
since $\varphi_{\mathbb{R}}$ is equivariant by translation by $1$
in the Lee direction,
all the points projecting to $p$
have the same time-shift $T$
(as translated points of $\varphi_{\mathbb{R}}$).
We say that $T$ is the time-shift
of the translated point $p$ of $\varphi$,
and we define the action spectrum of $\varphi$
to be the set of actions (hence time-shifts)
of its essential translated points
(in other words,
the action spectrum of $\varphi_{\mathbb{R}}$).

\begin{prop}\label{proposition: action spectrum}
Let $\varphi$, $\varphi_1$ and $\varphi_2$
be compactly supported lcs Hamiltonian diffeomorphisms
of $S^1 \times \mathbb{R}^{2n+1}$ or $S^1 \times \mathbb{R}^{2n} \times S^1$,
and $\overline{F}: S^1 \times B \times \mathbb{R}^N \rightarrow \mathbb{R}$,
$\overline{F_1}: S^1 \times B \times \mathbb{R}^{N_1} \rightarrow \mathbb{R}$
and $\overline{F_2}: S^1 \times B \times \mathbb{R}^{N_2} \rightarrow \mathbb{R}$
generating functions for $\varphi$, $\varphi_1$ and $\varphi_2$ respectively,
where $B$ is either $S^{2n+1}$ or $S^{2n} \times S^1$.
Then
\begin{itemize}
\item[(i)] The action spectrum of $\varphi$
is equal to the set of critical values of $\overline{F}$.
\item[(ii)] The action spectrum of $\varphi_2^{-1} \circ \varphi_1$
is equal to the set of critical values of the difference function
\[
\overline{F_1} - \overline{F_2}:
S^1 \times B \times \mathbb{R}^{N_1} \times \mathbb{R}^{N_2}
\rightarrow \mathbb{R} \,,\;
(\overline{F_1} - \overline{F_2})
(\theta, p, \zeta_1, \zeta_2)
= \overline{F_1} (\theta, p, \zeta_1) - \overline{F_2} (\theta, p, \zeta_2) \,.
\]
\end{itemize}
\end{prop}

\begin{proof}
If $\overline{F}: S^1 \times S^{2n+1} \times \mathbb{R}^N \rightarrow \mathbb{R}$
is a generating function for a compactly supported
lcs Hamiltonian diffeomorphism $\varphi$ of $S^1 \times \mathbb{R}^{2n+1}$
then the induced function
$F: S^1 \times \mathbb{R}^{2n+1} \times \mathbb{R}^N \rightarrow \mathbb{R}$
is a generating function for $L_{\varphi}$.
Similarly,
if $\overline{F}: S^1 \times S^{2n} \times S^1 \times \mathbb{R}^N \rightarrow \mathbb{R}$
is a generating function for a compactly supported
lcs Hamiltonian diffeomorphism $\varphi$ of $S^1 \times \mathbb{R}^{2n} \times S^1$
then the induced function
$F: S^1 \times \mathbb{R}^{2n+1} \times \mathbb{R}^N \rightarrow \mathbb{R}$
is a generating function for $L_{\varphi_{\mathbb{R}}}$.
In both cases,
$F$ and $\overline{F}$ have the same critical values.
Thus (i) and (ii) follow
from Propositions \ref{proposition: bijection critical points F translated point varphi}
and \ref{proposition: difference function}.
\end{proof}

For compactly supported lcs Hamiltonian diffeomorphisms $\varphi_1$ and $\varphi_2$
of $S^1 \times \mathbb{R}^{2n+1}$ or $S^1 \times \mathbb{R}^{2n} \times S^1$,
we pose $\varphi_1 \leq \varphi_2$
if $\varphi_2 \circ \varphi_1^{-1}$ is the time-1 map
of a non-negative compactly supported lcs Hamiltonian isotopy.
We will prove in Section \ref{section: spectral invariants}
that $\leq$ is a partial order.
One of the main ingredients
will be the next proposition.

\begin{prop}\label{proposition: partial order and gf}
Let $\varphi_1$ and $\varphi_2$
be compactly supported lcs Hamiltonian diffeomorphisms
of $S^1 \times \mathbb{R}^{2n+1}$ or $S^1 \times \mathbb{R}^{2n} \times S^1$
with $\varphi_1 \leq \varphi_2$.
Then there are
generating functions quadratic at infinity
$\overline{F_1}, \overline{F_2}: S^1 \times B \times \mathbb{R}^N \rightarrow \mathbb{R}$
for $\varphi_1$ and $\varphi_2$ respectively,
where $B$ is either $S^{2n+1}$ or $S^{2n} \times S^1$,
with $\overline{F_1} \leq \overline{F_2}$.
\end{prop}

Before proving Proposition \ref{proposition: partial order and gf}
we recall from \cite[Theorem 3]{Chaperon}
and \cite[Section III.2]{Theret_thesis}
the following construction
(see also \cite[Proposition 3.4]{San11},
\cite[Proposition 3.4]{Giroux}
and \cite[Lemma 2.10]{FSZ}).
The \textit{transition function}
of a $\mathcal{C}^1$-small contactomorphism $\phi$
of a $1$-jet bundle $J^1 \mathbb{R}^m$
is the function $G: J^1 \mathbb{R}^m \rightarrow \mathbb{R}$
defined  by
\[
G (q, p, z) = g_{p, z} (q) - f_{p, z} (q) \,,
\]
where, for every $p$ and $z$,
$f_{p, z} (q) = z + pq$
and $g_{p, z}$ is the function such that
$\im (j^1 g_{p, z}) = \phi \big( \im (j^1 f_{p, z}) \big)$.
In particular,
$G = 0$ if and only if $\phi$ is the identity.
Moreover,
we have the following facts.

\begin{itemize}

\item If $\phi$ is a $\mathcal{C}^1$-small
contactomorphism of $J^1\mathbb{R}^m$
with transition function $G: J^1\mathbb{R}^m \rightarrow \mathbb{R}$
and $\Lambda$ a Legendrian submanifold of $J^1 \mathbb{R}^m$
with generating function
$F: \mathbb{R}^m \times \mathbb{R}^N \rightarrow \mathbb{R}$
then the function
\[
G \,\sharp\, F:
\mathbb{R}^m \times
\big((\mathbb{R}^{m})^{\ast} \times \mathbb{R}^{m} \times \mathbb{R}^N\big)
\rightarrow \mathbb{R} \,,\;
\]
\[
(\underline{q}; p, q, \zeta) \mapsto
G \big( \underline{q}, p, F (\underline{q} + q, \zeta) - p (\underline{q} + q) \big)
+ F (\underline{q} + q, \zeta) - pq
\]
is a generating function for $\phi (\Lambda)$.

\item Denote the coordinates in $J^1\mathbb{R}^m$
by $(q_1, \cdots, q_m, p_1, \cdots, p_m, z)$.
If $\phi$ is a $\mathcal{C}^1$-small
contactomorphism of $J^1\mathbb{R}^m$
that is equivariant by translation by $1$
in the $q_j$-direction
then the transition function
$G: J^1\mathbb{R}^m \rightarrow \mathbb{R}$ of $\phi$
satisfies
\[
G (q_1, \cdots, q_j + 1, \cdots, q_m, p_1, \cdots, p_m, z - p_j)
= G (q_1, \cdots, q_m, p_1, \cdots, p_m, z) \,.
\]
Thus, if moreover
$F: \mathbb{R}^m \times \mathbb{R}^N \rightarrow \mathbb{R}$
is a function that is invariant by translation by $1$ in the $q_j$-direction
then $G \, \sharp \, F$ is invariant  by translation by $1$
in the $\underline{q}_j$-direction.
This implies the following.
We say that $F: \mathbb{R}^{m_1+m_2} \times \mathbb{R}^N \rightarrow \mathbb{R}$
is a generating function for a Legendrian submanifold $\Lambda$
of $J^1 (\mathbb{R}^{m_1} \times (S^1)^{m_2})$
if it is a generating function
for the preimage of $\Lambda$ by the covering map
$J^1 \mathbb{R}^{m_1 + m_2} \rightarrow J^1 (\mathbb{R}^{m_1} \times (S^1)^{m_2})$.
Let $\Lambda$ be a Legendrian submanifold
of $J^1 (\mathbb{R}^{m_1} \times (S^1)^{m_2})$
with generating function $F: \mathbb{R}^{m_1 + m_2}\times \R^N \rightarrow \mathbb{R}$,
$\{\phi_t\}$ a $\mathcal{C}^1$-small
contact isotopy of $J^1 (\mathbb{R}^{m_1} \times (S^1)^{m_2})$
and $G_t: J^1 \mathbb{R}^{m_1 + m_2} \rightarrow \mathbb{R}$
the transition functions of the lift of $\{\phi_t\}$
to $J^1 \mathbb{R}^{m_1 + m_2}$.
Then $G_t \,\sharp\, F$ is a $1$-parameter family
of generating functions for $\{\phi_t(\Lambda)\}$.

\item Let $\{\phi_t\}$ be a $\mathcal{C}^1$-small contact isotopy of $J^1\mathbb{R}^m$
with Hamiltonian function $H_t$ and transition functions $G_t$.
For every $p$ and $z$,
denote by $\{ (\phi_t)_{p,z} \}$ the isotopy of $\mathbb{R}^m$
defined by
\[
(\phi_t)_{p,z} (q) = \pi \circ \phi_t \circ j^1f_{p,z} (q) \,,
\]
where $\pi: J^1\mathbb{R}^m \rightarrow \mathbb{R}^m$
is the projection.
Then
\[
\left. \frac{dG_t}{dt} \right\lvert_{t = t_0} \big( (\phi_t)_{p,z} (q) , p, z \big)
= H_{t_0} \circ \phi_{t_0} \circ j^1f_{p,z} (q) \,.
\]

\item Let $\Lambda$ be a Legendrian submanifold
of $J^1 (\mathbb{R}^{m_1} \times (S^1)^{m_2})$
that coincides with the zero section
outside a compact set of $\mathbb{R}^{m_1} \times (S^1)^{m_2}$,
and $F: \mathbb{R}^{m_1 + m_2} \times \mathbb{R}^N \rightarrow \mathbb{R}$
a generating function for $\Lambda$.
We say that $F$ is a special generating function quadratic at infinity for $\Lambda$
if there is a non-degenerate quadratic form $F_{\infty}$ on $\mathbb{R}^N$
such that $F (x_1, \cdots, x_{m_1}, \theta_1, \cdots, \theta_{m_2}, \zeta) = F_{\infty} (\zeta)$
for $(x_1, \cdots, x_{m_1})$ outside a compact set.
If $F: \mathbb{R}^{m_1 + m_2} \times \mathbb{R}^N \rightarrow \mathbb{R}$
is a special generating function quadratic at infinity for $\Lambda$
and $\{\phi_t\}_{t \in [0, 1]}$ a $\mathcal{C}^1$-small
compactly supported contact isotopy of $J^1 (\mathbb{R}^{m_1} \times (S^1)^{m_2})$
whose lift to $J^1 \mathbb{R}^{m_1 + m_2}$
has transition functions $G_t: J^1 \mathbb{R}^{m_1 + m_2} \rightarrow \mathbb{R}$,
then there is a fibre preserving diffeomorphism $\Phi$
such that $F_t := (G_t \,\sharp\, F) \circ \Phi$
is a $1$-parameter family of special generating functions quadratic at infinity
for $\{\phi_t(\Lambda)\}$
with $F_0$ a stabilization of $F$.
If $\lVert \phi_t \rVert_{\mathcal{C}^1} \leq \epsilon$ for all $t$
then $\lVert G_t \rVert_{\mathcal{C}^0} \leq \epsilon$ for all $t$
and so, since $G_0 = 0$,
$F_1$ and $F_0$ are generating functions quadratic at infinity
for $\phi_1 (\Lambda)$ and $\Lambda$ respectively
with $\lVert F_1 - F_0 \rVert_{\mathcal{C}^0} \leq \epsilon$.

\end{itemize}

In particular the above facts
imply the following results.

\begin{lemma}\label{lemma: with transition functions}
Let $\Lambda_1$ and $\Lambda_2$ be Legendrian submanifolds
of $J^1 (\mathbb{R}^{m_1} \times (S^1)^{m_2})$
that are the image of the zero section
by a compactly supported contact isotopy
and such that $\Lambda_2$ is the image of $\Lambda_1$
by the time-$1$ map of a compactly supported
non-negative contact isotopy.
Let $\overline{\Lambda_1}$ and $\overline{\Lambda_2}$
be the induced Legendrian submanifolds
of $J^1 \big(S^{m_1} \times (S^1)^{m_2}\big)$.
Then there are generating functions quadratic at infinity
$\overline{F_1}, \overline{F_2}:
S^{m_1} \times (S^1)^{m_2} \times \mathbb{R}^N \rightarrow \mathbb{R}$
for $\overline{\Lambda_1}$ and $\overline{\Lambda_2}$
with $\overline{F_1} \leq \overline{F_2}$.
\end{lemma}

\begin{lemma}\label{lemma: with transition functions small}
Let $\phi$ be a compactly supported contactomorphism
of $J^1 \big( \mathbb{R}^{m_1} \times (S^1)^{m_2} \big)$
with $\lVert \phi \rVert_{\mathcal{C}^1} \leq \epsilon$
sufficiently small,
and $\Lambda$ a Legendrian submanifold
of $J^1 \big( \mathbb{R}^{m_1} \times (S^1)^{m_2} \big)$
that is the image of the zero section
by a compactly supported contact isotopy.
Let $\overline{\phi}$ and $\overline{\Lambda}$
be the induced contactomorphism and Legendrian submanifold
of $J^1 \big( S^{m_1} \times (S^1)^{m_2} \big)$.
Then there are generating functions quadratic at infinity
$\overline{F_0}, \overline{F_1}:
S^{m_1} \times (S^1)^{m_2} \times \mathbb{R}^N \rightarrow \mathbb{R}$
for $\overline{\Lambda}$ and $\overline{\phi} (\overline{\Lambda})$
with $\lVert \overline{F_1} - \overline{F_0} \rVert_{\mathcal{C}^0} \leq \epsilon$.
\end{lemma}

Using Lemma \ref{lemma: with transition functions}
we can now prove Proposition \ref{proposition: partial order and gf}.

\begin{proof}[Proof of Proposition \ref{proposition: partial order and gf}]
Suppose first that $\varphi_1$ and $\varphi_2$
are compactly supported lcs Hamiltonian diffeomorphisms
of $S^1 \times \mathbb{R}^{2n+1}$.
Let $\{\psi_t\}$ be a compactly supported lcs Hamiltonian isotopy
with $\psi_1 = \varphi_2 \circ \varphi_1^{-1}$
and with Hamiltonian function $H_t \geq 0$,
and consider the compactly supported
lcs Hamiltonian isotopy $\{ \Psi_{\psi_t} \}$
of $T^{\ast}_{-d\theta} (S^1 \times \mathbb{R}^{2n+1})$
that is the conjugation by $\tau$
by the compactly supported lcs Hamiltonian isotopy $\{ \id \boxtimes \psi_t \}$
of $(S^1 \times \mathbb{R}^{2n+1}) \boxtimes (S^1 \times \mathbb{R}^{2n+1})$.
Then $L_{\varphi_2} = \Psi_{\psi_1} (L_{\varphi_1})$.
Moreover,
$\{\Psi_{\psi_t}\}$ is a non-negative lcs Hamiltonian isotopy.
Indeed, 
by Lemma \ref{lemma: product of lcs diffeomorphisms}
the Hamiltonian function of $\{\id \boxtimes\, \psi_t\}$ is
\[
0 \boxplus H_t: \big[\big( (\theta_1, p_1) , (\theta_2, p_2) \big)\big]
\mapsto e^{\theta_2 - \theta_1} \, H_t (p_2)
\]
and so, by Lemma \ref{lemma: Hamiltonian conjugation},
the Hamiltonian function of $\{\Psi_{\psi_t}\}$
is $(0 \boxplus H_t) \circ \tau^{-1}$,
which is non-negative
because so is $H_t$.
By Proposition \ref{proposition: lift of lcs Hamiltonian isotopy to the contactization},
the lift of $\{\Psi_{\psi_t}\}$ to the contactization
$J^1_{-d\theta} (S^1 \times \mathbb{R}^{2n+1})$
is a non-negative compactly supported contact isotopy $\{\widetilde{\Psi_{\psi_t}}\}$.
The conjugation
$\{ U_{-d\theta} \circ \widetilde{\Psi_{\psi_t}} \circ (U_{-d\theta})^{-1} \}$
by the untwisting map \eqref{equation: contactomorphism 1 jet}
is thus a non-negative compactly supported contact isotopy of $J^1 (S^1 \times \mathbb{R}^{2n+1})$,
which sends $U_{-d\theta} (\widetilde{L_{\varphi_1}})$
to $U_{-d\theta} (\widetilde{L_{\varphi_2}})$.
By Lemma \ref{lemma: with transition functions},
there are generating functions quadratic at infinity
$\overline{F_1}, \overline{F_2}:
S^1 \times S^{2n+1} \times \mathbb{R}^N \rightarrow \mathbb{R}$
for the compactifications in $J^1 (S^1 \times S^{2n+1})$
of $U_{-d\theta} (\widetilde{L_{\varphi_1}})$ and
$U_{-d\theta} (\widetilde{L_{\varphi_2}})$
with $\overline{F_1} \leq \overline{F_2}$.
Since these compactifications are
$U_{-d\theta} (\widetilde{\overline{L_{\varphi_1}}})$ and
$U_{-d\theta} (\widetilde{\overline{L_{\varphi_2}}})$,
by Proposition \ref{proposition: relation gf 1 jet}
we conclude that $\overline{F_1}$ and $\overline{F_2}$
are generating functions quadratic at infinity
for $\overline{L_{\varphi_1}}$ and $\overline{L_{\varphi_2}}$,
hence for $\varphi_1$ and $\varphi_2$,
with $\overline{F_1} \leq \overline{F_2}$.
If $\varphi_1$ and $\varphi_2$
are compactly supported lcs Hamiltonian diffeomorphisms
of $S^1 \times \mathbb{R}^{2n} \times S^1$
a similar proof applies by considering their lifts
to $S^1 \times \mathbb{R}^{2n+1}$.
\end{proof}

Using Lemma \ref{lemma: with transition functions small}
we also obtain the following result,
which is needed later.

\begin{lemma}\label{lemma: gf of close lcs diffeomorphisms}
Let $\varphi$ be a compactly supported lcs Hamiltonian diffeomorphism
of $S^1 \times \mathbb{R}^{2n+1}$ or $S^1 \times \mathbb{R}^{2n} \times S^1$.
For every $\epsilon > 0$ small enough
there is $\delta > 0$
such that if $\psi$ is a compactly supported
lcs Hamiltonian diffeomorphism
with $\lVert \varphi \circ \psi^{-1} \rVert_{\mathcal{C}^1} \leq \delta$
then there are generating functions quadratic at infinity
$\overline{F_{\varphi}}, \overline{F_{\psi}}:
S^1 \times B \times \mathbb{R}^N \rightarrow \mathbb{R}$
for $\varphi$ and $\psi$ respectively,
where $B$ is either $S^{2n+1}$ or $S^{2n} \times S^1$,
with
$\lVert \overline{F_{\varphi}} - \overline{F_{\psi}} \rVert_{\mathcal{C}^0}
\leq \epsilon$.
\end{lemma}

\begin{proof}
Suppose first that $\varphi$ is
a compactly supported lcs Hamiltonian diffeomorphism
of $S^1 \times \mathbb{R}^{2n+1}$.
For every $\psi$ we have
$L_{\varphi} = \Psi_{\varphi \circ \psi^{-1}} (L_{\psi})$,
where $\Psi_{\varphi \circ \psi^{-1}}$ is the conjugation by $\tau$
of $\id \boxtimes (\varphi \circ \psi^{-1})$,
and so
$$
U_{-d\theta} (\widetilde{L_{\varphi}})
= U_{-d\theta} \circ \widetilde{\Psi_{\varphi \circ \psi^{-1}}} \circ (U_{-d\theta})^{-1}
\; \big(U_{-d\theta} (\widetilde{L_{\psi}})\big) \,,
$$
where $\widetilde{L_{\varphi}}$,
$\widetilde{L_{\psi}}$ and $\widetilde{\Psi_{\varphi \circ \psi^{-1}}}$
are the lifts of $L_{\varphi}$, $L_{\psi}$ and $\Psi_{\varphi \circ \psi^{-1}}$
to the contactization $J^1_{-d\theta} (S^1 \times \mathbb{R}^{2n+1})$
and $U_{-d\theta}$ is the untwisting map \eqref{equation: contactomorphism 1 jet}.
Take $\delta > 0$ such that if $\lVert \phi \rVert_{\mathcal{C}^1} \leq \delta$
then $\lVert U_{-d\theta} \circ \widetilde{\Psi_{\phi}} \circ (U_{-d\theta})^{-1} \rVert_{\mathcal{C}^1}
\leq \epsilon$.
Let $\psi$ be a compactly supported lcs Hamiltonian diffeomorphism
of $S^1 \times \mathbb{R}^{2n+1}$
with $\lVert \varphi \circ \psi^{-1} \rVert_{\mathcal{C}^1} \leq \delta$,
and let $\{ \phi_t \}_{t \in [0,1]}$
be a compactly supported lcs Hamiltonian isotopy
of $S^1 \times \mathbb{R}^{2n+1}$
with $\phi_1 = \varphi \circ \psi^{-1}$
and $\lVert \phi_t \rVert_{\mathcal{C}^1} \leq \delta$
for all $t$.
Then
$\lVert U_{-d\theta} \circ \widetilde{\Psi_{\phi_t}} \circ (U_{-d\theta})^{-1} \rVert_{\mathcal{C}^1}
\leq \epsilon$ for all $t$.
By Lemma \ref{lemma: with transition functions small},
there are generating functions quadratic at infinity
$\overline{F_{\varphi}}$ and $\overline{F_{\psi}}$
for the compactifications of $U_{-d\theta} (\widetilde{L_{\varphi}})$
and $U_{-d\theta} (\widetilde{L_{\psi}})$
with $\lVert \overline{F_{\varphi}} - \overline{F_{\psi}} \rVert_{\mathcal{C}^0}
\leq \epsilon$.
Since these compactifications are
$U_{-d\theta} (\widetilde{\overline{L_{\varphi}}})$
and $U_{-d\theta} (\widetilde{\overline{L_{\psi}}})$,
by Proposition \ref{proposition: relation gf 1 jet}
we deduce that $\overline{F_{\varphi}}$ and $\overline{F_{\psi}}$
are generating functions quadratic at infinity
for $\overline{L_{\varphi}}$ and $\overline{L_{\psi}}$,
hence of $\varphi$ and $\psi$.
If $\varphi$ is a compactly supported lcs Hamiltonian diffeomorphism
of $S^1 \times \mathbb{R}^{2n} \times S^1$
the proof is similar.
\end{proof}

We end this section by discussing the relation
with generating functions for contactomorphisms.
We start with the following lemma.

\begin{lemma}\label{lemma: relation contact case associated Lagrangian}
Let $\phi$ be a contactomorphism of $\mathbb{R}^{2n+1}$,
and $\widetilde{\phi}$ its lift to $S^1 \times \mathbb{R}^{2n+1}$.
Consider the Legendrian submanifold
$\Lambda_{\phi}$ of $J^1 \mathbb{R}^{2n+1}$,
and the exact Lagrangian submanifold $L_{\widetilde{\phi}}$
of $T^{\ast}_{-d\theta} (S^1 \times \mathbb{R}^{2n+1})$.
Let $\widetilde{L_{\widetilde{\phi}}}$
be the lift of $L_{\widetilde{\phi}}$
to the contactization $J^1_{-d\theta} (S^1 \times \mathbb{R}^{2n+1})$.
Then
\[
U_{-d\theta} \, \big( \, \widetilde{L_{\widetilde{\phi}}} \, \big)
= 0_{S^1} \times \Lambda_{\phi} \,,
\]
where $U_{-d\theta}$ is the untwisting map
\eqref{equation: contactomorphism 1 jet}
from $J^1_{-d\theta} (S^1 \times \mathbb{R}^{2n+1})$
to $J^1 (S^1 \times \mathbb{R}^{2n+1})$
and where we identify $J^1 (S^1 \times \mathbb{R}^{2n+1})$
with $T^{\ast}S^1 \times J^1 \mathbb{R}^{2n+1}$.
\end{lemma}

\begin{proof}
Identify $T^{\ast}_{-d\theta} (S^1 \times \mathbb{R}^{2n+1})$
with $S^1 \times J^1 \mathbb{R}^{2n+1}$
by the strict lcs diffeomorphism
of Lemma \ref{lemma: identification twisted cotangent bundle with conformal symplectization}.
By Proposition \ref{proposition: relation graphs lcs contact}
we then have $L_{\widetilde{\phi}} = S^1 \times \Lambda_{\phi}$.
We thus conclude using Lemma \ref{lemma: untwisting of lift}.
\end{proof}

Using Lemma \ref{lemma: relation contact case associated Lagrangian}
we now obtain the following relation between generating functions.

\begin{prop}\label{proposition: relation gf contact}
Let $\phi$ be a compactly supported contactomorphism
of $\mathbb{R}^{2n+1}$ or $\mathbb{R}^{2n} \times S^1$
contact isotopic to the identity,
and $\widetilde{\phi}$ its lift to $S^1 \times \mathbb{R}^{2n+1}$
or $S^1 \times \mathbb{R}^{2n} \times S^1$ respectively.
Let $\overline{F}: B \times \mathbb{R}^N \rightarrow \mathbb{R}$,
where $B$ is either $S^{2n+1}$ or $S^{2n} \times S^1$,
be a generating function for $\phi$.
Then the function
$\overline{F} \circ \pr_2:
S^1 \times (B \times \mathbb{R}^N) \rightarrow \mathbb{R}$,
where $\pr_2$ is the projection to the second factor,
is a generating function for $\widetilde{\phi}$.
\end{prop}

\begin{proof}
Suppose first that $\phi$ is a contactomorphism of $\mathbb{R}^{2n+1}$.
By definition,
$\overline{F}: S^{2n+1} \times \mathbb{R}^N \rightarrow \mathbb{R}$
is a generating function for the Legendrian submanifold
$\overline{\Lambda_{\phi}}$ of $J^1 S^{2n+1}$.
This implies that
$\overline{F} \circ \pr_2$
is a generating function for the Legendrian submanifold
$0_{S^1} \times \overline{\Lambda_{\phi}}$ of $J^1 (S^1 \times S^{2n+1})
\equiv T^{\ast}S^1 \times J^1 S^{2n+1}$.
Lemma \ref{lemma: relation contact case associated Lagrangian}
implies that $0_{S^1} \times \overline{\Lambda_{\phi}}
= U_{-d\theta} ( \widetilde{\overline{L_{\widetilde{\phi}}}})$.
By Proposition \ref{proposition: relation gf 1 jet}
we thus conclude that $\overline{F} \circ \pr_2$ is a generating function
for $\overline{L_{\widetilde{\phi}}}$,
hence for $\widetilde{\phi}$.
If $\phi$ is a contactomorphism of $\mathbb{R}^{2n} \times S^1$
the proof is similar,
by applying Lemma \ref{lemma: relation contact case associated Lagrangian}
to its lift to $\mathbb{R}^{2n+1}$.
\end{proof}

\section{Spectral selectors for compactly supported lcs
Hamiltonian diffeomorphisms of $S^1 \times \mathbb{R}^{2n+1}$
and $S^1 \times \mathbb{R}^{2n} \times S^1$}
\label{section: spectral invariants}

In this section we define spectral selectors
for compactly supported lcs
Hamiltonian diffeomorphisms of $S^1 \times \mathbb{R}^{2n+1}$
and $S^1 \times \mathbb{R}^{2n} \times S^1$.
Our construction is analogous to those
in \cite{Viterbo, Bhupal, San10, San11}
of spectral selectors for compactly supported
Hamiltonian diffeomorphisms of $\mathbb{R}^{2n}$
and for compactly supported contactomorphisms
of $\mathbb{R}^{2n+1}$ and $\mathbb{R}^{2n} \times S^1$
contact isotopic to the identity.

We first recall from \cite{Viterbo}
the definition of spectral selectors
for functions quadratic at infinity.
Let $B$ be a compact manifold
and $F: E = B \times \mathbb{R}^N \rightarrow \mathbb{R}$
a function quadratic at infinity.
Let $E^- \rightarrow B$ be the subbundle of $E \rightarrow B$
on which the quadratic form $F_{\infty}$
is negative definite,
and consider the Thom isomorphism
\[
T: H^{\ast}(B) \rightarrow H^{\ast + \ind(F_{\infty})} \big( D(E^-), S(E^-) \big) \,,
\]
where we denote by $D(E^-)$ and $S(E^-)$ the disc and sphere bundles
of the vector bundle $E^- \rightarrow B$.
We consider the sublevel sets
\[
E^{a} = \{\, e \in E \;\lvert\; F(e) \leq a \,\} \,,
\]
and we denote $E^{-\infty} = E^{a}$
for $a$ smaller than all the critical values of $F$.
We then consider the inclusion
\[
i_a: (E^{a}, E^{-\infty}) \hookrightarrow (E, E^{-\infty}) \,,
\]
and the induced homomorphism in cohomology
\[
i_a^{\ast}: H^{\ast} (E, E^{-\infty}) \rightarrow H^{\ast} (E^{a}, E^{-\infty}) \,.
\]
Finally we consider the composition
\[
i_a^{\,\ast} \,\circ\, \ex \,\circ\, T:
H^{\ast} (B) \rightarrow H^{\ast + \ind (F_{\infty})} (E^{a}, E^{-\infty}) \,,
\]
where
\[
\ex: H^{\ast} \big( D(E^-), S(E^-) \big) \rightarrow H^{\ast} (E, E^{-\infty})
\]
is the excision isomorphism.
For $a$ smaller than all the critical values of $F$
the homomorphism $i_a^{\,\ast} \,\circ\, \ex \,\circ\, T$ vanishes,
while for $a$ bigger than all the critical values of $F$
it is an isomorphism.
For any non-zero class $u$ in $H^{\ast}(B)$ we define
$$
c(u,F) = \inf \left\{\, a \in \mathbb{R} ~|~
i_a^{\, \ast} \,\circ\, \ex \,\circ\, T \, (u) \neq 0 \,\right\} \,.
$$
The properties in the following proposition
are proved in \cite{Viterbo}.
In (\ref{properties spectral sel gf: PD})
and (\ref{properties spectral sel gf: non-degeneracy}),
$[B]$ and $1$ denote respectively
the fundamental and unit classes in $H^{\ast}(B)$.
For (\ref{properties spectral sel gf: equivalence}),
recall that two functions $F_1$ and $F_2$ quadratic at infinity
are said to be equivalent
if there are non-degenerate quadratic forms $Q_1$ and $Q_2$
and a fibre preserving diffeomorphism $\Phi$
such that $F_1 \oplus Q_1 = (F_2 \oplus Q_2) \circ \Phi$.

\begin{prop}\label{proposition: properties Viterbo}
The spectral selectors $c (u, F)$ satisfy the following properties:
\renewcommand{\theenumi}{\roman{enumi}}
\begin{enumerate}
\item \label{properties spectral sel gf: spectrality} \emph{Spectrality:}
$c(u, F)$ is a critical value of $F$.
\item \label{properties spectral sel gf: continuity} \emph{Continuity:}
if $\lvert F_1 - F_2 \, \rvert_{\mathcal{C}^0} \leq \epsilon$
then
$$
\lvert \, c(u, F_1) - c(u, F_2) \, \rvert \leq \epsilon \,.
$$
\item \label{properties spectral sel gf: monotonicity} \emph{Monotonicity:}
if $F_1, F_2: B \times \mathbb{R}^{N} \rightarrow \mathbb{R}$
satisfy $F_1 \leq F_2$
then $c (u, F_1) \leq c (u, F_2)$.
\item \label{properties spectral sel gf: PD} \emph{Poincar\'e duality:}
$c([B], - F) = - c(1, F)$.
\item \label{properties spectral sel gf: non-degeneracy} \emph{Non-degeneracy:}
if $F: B \times \mathbb{R}^N \rightarrow \mathbb{R}$
is a generating function quadratic at infinity
such that $c([B], F) = c(1, F)$
then $F$ generates the zero section of $T^{\ast}B$.
\item \label{properties spectral sel gf: sum} \emph{Triangle inequality:}
let $F_1: B \times \mathbb{R}^{N_1} \rightarrow \mathbb{R}$
and $F_2: B \times \mathbb{R}^{N_2} \rightarrow \mathbb{R}$
be functions quadratic at infinity,
and consider the function quadratic at infinity
$F_1 + F_2: B \times \mathbb{R}^{N_1} \times \mathbb{R}^{N_2} \rightarrow \mathbb{R}$
defined by
\[
(F_1 + F_2) (q, \zeta_1, \zeta_2) = F_1 (q, \zeta_1) + F_2 (q, \zeta_2) \,.
\]
Then
\[
c (u \cup v, F_1 + F_2) \geq c (u, F_1) + c (v, F_2) \,.
\]
\item \label{properties spectral sel gf: equivalence} \emph{Equivalence:}
if $F_1$ and $F_2$ are equivalent
then $c (u, F_1) = c (u, F_2)$ for any $u$.
\end{enumerate}
\end{prop}

Let now $\varphi$ be a compactly supported
lcs Hamiltonian diffeomorphism of $S^1 \times \mathbb{R}^{2n+1}$
or $S^1 \times \mathbb{R}^{2n} \times S^1$,
and $\overline{F}: S^1 \times B \times \mathbb{R}^N \rightarrow \mathbb{R}$,
where $B$ is either $S^{2n+1}$ or $S^{2n} \times S^1$,
a generating function quadratic at infinity for $\varphi$.
We define 
\[
c_+ (\varphi) = c ([S^1 \times B], \overline{F})
\quad \text{ and } \quad c_- (\varphi) = c (1, \overline{F}) \,,
\]
where $[S^1 \times B]$ and $1$
are the fundamental and unit classes in $H^{\ast} (S^1 \times B)$.
By Proposition \ref{proposition: properties Viterbo}
(\ref{properties spectral sel gf: equivalence})
and the uniqueness statement
in Proposition \ref{proposition: gf lcs diffeomorphisms}, the quantities 
$c_+ (\varphi)$ and $c_- (\varphi)$ are well-defined,
in the sense that they do not depend on the choice
of the generating function quadratic at infinity
$\overline{F}$ for $\varphi$.
Moreover,
by Proposition \ref{proposition: properties Viterbo}
(\ref{properties spectral sel gf: spectrality})
and Proposition \ref{proposition: action spectrum} (i)
they belong to the action spectrum of $\varphi$.
In particular,
$c_{\pm} (\id) = 0$.

Before discussing the other properties of $c_+$ and $c_-$
we prove the following lemma,
which will be the key ingredient
to obtain the invariance by conjugation property.

\begin{lemma}\label{lemma: translated points of conjugation}
\renewcommand{\theenumi}{\roman{enumi}} \
\begin{enumerate}
\item Let $\varphi$ and $\psi$ be  compactly supported
lcs Hamiltonian diffeomorphisms of $S^1 \times \mathbb{R}^{2n + 1}$.
Suppose that $p$ is an essential translated point of $\varphi$
of time-shift equal to zero.
Then $\psi(p)$ is an essential translated points of $\psi \circ \varphi \circ \psi^{-1}$
of time-shift equal to zero.
Moreover,
$\psi(p)$ is a non-degenerate translated point
of $\psi \circ \varphi \circ \psi^{-1}$
if and only if $p$ is a non-degenerate translated point of $\varphi$.

\item Let $\varphi$ and $\psi$ be compactly supported
lcs Hamiltonian diffeomorphisms of $S^1 \times \mathbb{R}^{2n} \times S^1$.
Suppose that $p$ is an essential translated point of $\varphi$
of time-shift equal to an integer $k$.
Then $\psi(p)$ is an essential translated points of $\psi \circ \varphi \circ \psi^{-1}$
of time-shift equal to $k$.
Moreover,
$\psi(p)$ is a non-degenerate translated point
of $\psi \circ \varphi \circ \psi^{-1}$
if and only if $p$ is a non-degenerate translated point of $\varphi$.
\end{enumerate}
\end{lemma}

\begin{proof}
We prove (ii),
the proof of (i) being similar.
Since the time-shift of $p$ is an integer,
we have $\varphi (p) = p$.
Let $g$ and $f$ be the conformal factors
of $\varphi$ and $\psi$ respectively.
Then the conformal factor of $\psi \circ \varphi \circ \psi^{-1}$ is
$(f \circ \varphi + g - f) \circ \psi^{-1}$.
In particular,
since $g(p) = 0$ and $\varphi (p) = p$
we have
\[
(f \circ \varphi + g - f ) \circ \psi^{-1}
\, \big(\psi(p)\big) = 0 \,.
\]
Let $\varphi_{\mathbb{R}}$ and $\psi_{\mathbb{R}}$
be the lifts of $\varphi$ and $\psi$ to $S^1 \times \mathbb{R}^{2n+1}$,
and $\bar{p}$ a point of $S^1 \times \mathbb{R}^{2n+1}$
projecting to $p$.
Then $\bar{p}$ is a translated point of $\varphi_{\mathbb{R}}$
of time-shift $k$,
in particular there is a Lee chord of time-shift $k$
from $\varphi_{\mathbb{R}} (\bar{p})$ to $\bar{p}$.
Since $k$ is an integer
and $\psi_{\mathbb{R}}$ is equivariant by translation by $1$
in the Lee direction,
this implies that there is a Lee chord
of time-shift $k$
between $\psi_{\mathbb{R}} \big(\varphi_{\mathbb{R}}(\bar{p})\big)
= (\psi_{\mathbb{R}} \circ \varphi_{\mathbb{R}} \circ \psi_{\mathbb{R}}^{-1})
\big(\psi_{\mathbb{R}}(\bar{p})\big)$
and $\psi_{\mathbb{R}}(\bar{p})$.
Since $\psi_{\mathbb{R}}(\bar{p})$ projects to $\psi(p)$
and $\psi_{\mathbb{R}} \circ \varphi_{\mathbb{R}} \circ \psi_{\mathbb{R}}^{-1}
= ( \psi \circ \varphi \circ \psi^{-1} )_{\mathbb{R}}$,
we conclude that $\psi(p)$ is a translated point
of $\psi \circ \varphi \circ \psi^{-1}$
of time-shift $k$.
Let now $S_{\varphi}$ be the action of $\varphi$
and $S_{\psi}$ the action of $\psi$.
Then the action of $\psi \circ \varphi \circ \psi^{-1}$ is
\[
S_{\psi \circ \varphi \circ \psi^{-1}} =
e^{f \circ \psi^{-1}} (S_{\psi} - S_{\varphi} - e^{- g} \, S_{\psi} \circ \varphi) \circ \psi^{-1} \,.
\]
Since $p$ is an essential translated point of $\varphi$,
by Example \ref{example: action translated point in conformal symplectization}
we have $S_{\varphi} (p) = 0$.
Since moreover $g(p) = 0$ and $\varphi(p) = p$,
we conclude that $S_{\psi \circ \varphi \circ \psi^{-1}} \big(\psi(p)\big) = 0$
and so, by Example \ref{example: action translated point in conformal symplectization},
$\psi(p)$ is an essential translated point
of $\psi \circ \varphi \circ \psi^{-1}$.
We now show that $\psi(p)$ is non-degenerate
if and only if so is $p$.
Suppose that $p$ is a degenerate translated point of $\varphi$,
thus there is a non-zero tangent vector $Y$ at $p$
such that $\varphi_{\ast}(Y) = Y$ and $dg(Y) = 0$.
Then
\[
(\psi \circ \varphi \circ \psi^{-1})_{\ast} \big( \psi_{\ast}(Y) \big)
= \psi_{\ast} (Y)
\]
and
\[
d \big( ( f \circ \varphi + g  - f) \circ \psi^{-1} \big)
\big( \psi_{\ast}(Y) \big)
= df \big( \varphi_{\ast}(Y) \big) + dg(Y) - df (Y)
= 0 \,,
\]
thus $\psi(p)$ is a degenerate translated point
of $\psi \circ \varphi \circ \psi^{-1}$.
The other implication is similar.
\end{proof}

We can now prove the following properties
of the spectral selectors $c_+$ and $c_-$.

\begin{prop}\label{proposition: properties of spectral invariants}
The spectral selectors $c_+$ and $c_-$
on $\Ham^c (S^1 \times \mathbb{R}^{2n+1})$
and $\Ham^c (S^1 \times \mathbb{R}^{2n} \times S^1)$
satisfy the following properties:
\renewcommand{\theenumi}{\roman{enumi}}
\begin{enumerate}
\item \label{properties spectral sel diffeomorphisms: continuity} \emph{Continuity:}
$c_+$ and $c_-$ are continuous with respect to the $\mathcal{C}^1$-topology.
\item \label{properties spectral sel diffeomorphisms: monotonicity} \emph{Monotonicity:}
if $\varphi_1 \leq \varphi_2$ then $c_+ (\varphi_1) \leq c_+ (\varphi_2)$
and $c_- (\varphi_1) \leq c_- (\varphi_2)$.
\item \label{properties spectral sel diffeomorphisms: PD} \emph{Poincar\'e duality:}
$c_- (\varphi) = - c_+ (\varphi ^{-1})$.
\item \label{properties spectral sel diffeomorphisms: positivity} \emph{Positivity:}
$c_+ (\varphi) \geq 0$ and $c_- (\varphi) \leq 0$.
\item \label{properties spectral sel diffeomorphisms: non-degeneracy} \emph{Non-degeneracy:}
$c_+ (\varphi) = c_- (\varphi) = 0$ if and only if $\varphi$ is the identity.
\item \label{properties spectral sel diffeomorphisms: triangle inequality} \emph{Triangle inequality:}
$c_+ (\varphi_1 \circ \varphi_2) \leq c_+ (\varphi_1) + c_+ (\varphi_2)$
and $c_- (\varphi_1 \circ \varphi_2) \geq c_- (\varphi_1) + c_- (\varphi_2)$.
\item \label{properties spectral sel diffeomorphisms: invariance by conjugation}
\emph{Invariance by conjugation:}
if $\varphi$ and $\psi$ are compactly supported
lcs Hamiltonian diffeomorphisms of $S^1 \times \mathbb{R}^{2n+1}$
then $c_+ (\varphi) = 0$ if and only if $c_+ (\psi \circ \varphi \circ \psi^{-1}) = 0$,
and similarly for $c_-$.
If $\varphi$ and $\psi$ are compactly supported lcs Hamiltonian diffeomorphisms
of $S^1 \times \mathbb{R}^{2n} \times S^1$
then $\lceil c_+ (\psi \circ \varphi \circ \psi^{-1})\rceil
= \lceil c_+ (\phi) \rceil$
and $\lfloor c_+ (\psi \circ \varphi \circ \psi^{-1}) \rfloor
= \lfloor c_+ (\phi) \rfloor$,
and similarly for $c_-$,
where $\lceil \,\cdot\, \rceil$ and $\lfloor \,\cdot\, \rfloor$
denote the ceiling and floor functions.
\end{enumerate}
\end{prop}

\begin{proof}
Continuity (\ref{properties spectral sel diffeomorphisms: continuity})
follows from Proposition \ref{proposition: properties Viterbo}
(\ref{properties spectral sel gf: continuity})
and Lemma \ref{lemma: gf of close lcs diffeomorphisms},
and monotonicity (\ref{properties spectral sel diffeomorphisms: monotonicity})
from Proposition \ref{proposition: properties Viterbo}
(\ref{properties spectral sel gf: monotonicity})
and Proposition \ref{proposition: partial order and gf}.

Poincar\'e duality (\ref{properties spectral sel diffeomorphisms: PD})
can be seen as follows.
Let $\overline{F_{\varphi}}$ and $\overline{F_{\varphi^{-1}}}$
be generating functions quadratic at infinity
for $\varphi$ and $\varphi^{-1}$ respectively.
We first show that,
for any non-zero class $u$,
\begin{equation}\label{equation: in proof PD}
c (u, \overline{F_{\varphi^{-1}}}) = c (u, - \overline{F_{\varphi}}) \,.    
\end{equation}
To see this,
let $\{\varphi_t\}$ be a compactly supported lcs Hamiltonian isotopy
with $\varphi_1 = \varphi$,
let $\overline{F_t}$ and $\overline{G_t}$
be $1$-parameter families of generating functions quadratic at infinity
for $\{\varphi_t\}$ and $\{\varphi_t \circ \varphi^{-1}\}$ respectively,
and set
\[
c_t = c (u \,,\, \overline{G_t} - \overline{F_t}) \,.
\]
By Proposition \ref{proposition: action spectrum} (ii),
for every $t$ the set of critical values
of $\overline{G_t} - \overline{F_t}$
is equal to the action spectrum of $\varphi^{-1}$.
The functions $\overline{G_t} - \overline{F_t}$
have thus all the same set $S$ of critical values.
By Proposition \ref{proposition: properties Viterbo}
(\ref{properties spectral sel gf: spectrality}),
$c_t$ is in $S$ for all $t$.
By Proposition \ref{proposition: properties Viterbo}
(\ref{properties spectral sel gf: continuity}),
the map $t \mapsto c_t$ is continuous.
Since $S$ is nowhere dense
we conclude that $t \mapsto c_t$ is constant,
and so $c_0 = c (u, \overline{G_0} - \overline{F_0})$
is equal to $c_1 = c (u, \overline{G_1} - \overline{F_1})$.
But $\overline{G_0}$
is a generating function quadratic at infinity for $\varphi^{-1}$,
hence equivalent to $\overline{F_{\varphi^{-1}}}$,
and $\overline{F_0}$
is a generating function quadratic at infinity for the identity,
hence equivalent to the zero function.
Thus $\overline{G_0} - \overline{F_0}$
is equivalent to $\overline{F_{\varphi^{-1}}}$.
Similarly,
$\overline{G_1} - \overline{F_1}$
is equivalent to $- \overline{F_{\varphi}}$.
By Proposition \ref{proposition: properties Viterbo}
(\ref{properties spectral sel gf: equivalence}),
we thus obtain \eqref{equation: in proof PD}.
Using \eqref{equation: in proof PD}
and Proposition \ref{proposition: properties Viterbo} (\ref{properties spectral sel gf: PD})
we then conclude that
\[
c_-(\varphi) = c (1, \overline{F_{\varphi}})
= - c ([S^1 \times B], - \overline{F_{\varphi}})
= - c ([S^1 \times B], \overline{F_{\varphi^{-1}}})
= - c_+ (\varphi^{-1}) \,.
\]

By the Poincar\'e duality property
(\ref{properties spectral sel diffeomorphisms: PD}),
positivity (\ref{properties spectral sel diffeomorphisms: positivity})
follows if we prove that $c_- (\varphi) \leq 0$.
This can be seen as follows.
Let $\overline{F}: E = S^1 \times B \times \mathbb{R}^N \rightarrow \mathbb{R}$,
where $B$ is either $S^{2n+1}$ or $S^{2n} \times S^1$,
be a generating function quadratic at infinity for $\varphi$.
Since
\[
c_- (\varphi) = c (1, \overline{F})
= \inf \{\, a \in \mathbb{R} \;\lvert\; i_a^{\,\ast} \,\circ\, \ex \,\circ\, T \, (1) \neq 0 \,\} \,,
\]
we need to show that $i_0^{\,\ast} \,\circ\, \ex \,\circ\, T \, (1) \neq 0$.
Suppose first that $B$ is $S^{2n+1}$,
and denote by $p_{\infty} \in S^{2n+1}$
the point at infinity 
when seeing $S^{2n+1}$ as the $1$-point compactification of $\mathbb{R}^{2n+1}$.
Denote by $E_{S^1 \times \{p_{\infty}\}}$,
$E^0_{S^1 \times \{p_{\infty}\}}$ and $E^{-\infty}_{S^1 \times \{p_{\infty}\}}$
the preimages of $S^1 \times \{p_{\infty}\}$
by the projections of $E$, $E^0$ and $E^{-\infty}$
to $S^1 \times S^{2n+1}$.
Consider the commutative diagram
\[
\begin{tikzcd}
H^{\ast} (S^1 \times S^{2n+1})
\arrow{r}{\ex \circ T} \arrow{d}
& H^{\ast + \ind (\overline{F}_{\infty})} (E, E^{-\infty})
\arrow{r}{i_0} \arrow{d}
&  H^{\ast + \ind (\overline{F}_{\infty})} (E^0, E^{-\infty})
\arrow{d}\\
H^{\ast} (S^1 \times \{p_{\infty}\})
\arrow{r}{\ex \circ T}
& H^{\ast + \ind (\overline{F}_{\infty})}
\big(E_{S^1 \times \{p_{\infty}\}}, E_{S^1 \times \{p_{\infty}\}}^{-\infty}\big)
\arrow{r}{i_0}
&  H^{\ast + \ind (\overline{F}_{\infty})}
\big(E_{S^1 \times \{p_{\infty}\}}^0, E_{S^1 \times \{p_{\infty}\}}^{-\infty}\big)
\end{tikzcd}
\]
where the vertical arrows are the homomorphisms induced by the inclusions.
Since the image of
\[
i_{\overline{F}, -d\theta}: \Sigma_{\overline{F}}
\rightarrow T^{\ast}_{-d\theta} (S^1 \times S^{2n+1}) \,,\;
(\theta, p, \zeta) \mapsto
\Big( \theta \,,\, p \,,\,
\frac{\partial \overline{F}}{\partial \theta} (\theta, p, \zeta) + \overline{F} (\theta, p, \zeta)
\,,\, \frac{\partial \overline{F}}{\partial p} (\theta, p, \zeta) \Big)
\]
is $\overline{L_{\varphi}}$,
which by construction is equal to the zero section
on a neighborhood of $S^1 \times \{p_{\infty}\}$,
for every $\theta \in S^1$ and $\zeta \in \mathbb{R}^N$
the point $i_{\overline{F}, -d\theta} (\theta, p_{\infty}, \zeta)$
belongs to the zero section.
If $(\theta, p_{\infty}, \zeta)$ is a critical point of $\overline{F}$
it thus has critical value zero.
We conclude that zero is the only critical value
of $\left. \overline{F} \right\lvert_{E_{S^1 \times \{p_{\infty}\}}}$,
and so that the map
\[
i_0: H^{\ast + \ind (\overline{F}_{\infty})}
\big(E_{S^1 \times \{p_{\infty}\}}, E_{S^1 \times \{p_{\infty}\}}^{-\infty}\big)
\rightarrow H^{\ast + \ind (\overline{F}_{\infty})}
\big(E_{S^1 \times \{p_{\infty}\}}^0, E_{S^1 \times \{p_{\infty}\}}^{-\infty}\big)
\]
is an isomorphism.
In particular,
the image of the unit class of $H^{\ast} (S^1 \times \{p_{\infty}\})$
by
\[
i_0^{\,\ast} \,\circ\, \ex \,\circ\, T:
H^{\ast} (S^1 \times \{p_{\infty}\})
\rightarrow H^{\ast + \ind (\overline{F}_{\infty})}
\big(E_{S^1 \times \{p_{\infty}\}}^0, E_{S^1 \times \{p_{\infty}\}}^{-\infty}\big)
\]
is different than zero.
Since the first vertical arrow
sends the unit class of $H^{\ast} (S^1 \times S^{2n+1})$
to the unit class of $H^{\ast} (S^1 \times \{p_{\infty}\})$,
we conclude that the image
of the unit class of $H^{\ast} (S^1 \times S^{2n+1})$ by
\[
i_0^{\,\ast} \,\circ\, \ex \,\circ\, T:
H^{\ast} (S^1 \times S^{2n+1})
\rightarrow H^{\ast + \ind (\overline{F}_{\infty})}
(E^0, E^{-\infty})
\]
is different than zero,
as we wanted.
If $B$ is $S^{2n} \times S^1$
the proof is similar,
considering the restriction of $E$
to $S^1 \times \{p_{\infty}\} \times S^1$,
where $p_{\infty}$ is the point at infinity of $S^{2n}$.

We now prove non-degeneracy
(\ref{properties spectral sel diffeomorphisms: non-degeneracy}).
Suppose that $c_+ (\varphi) = c_- (\varphi) = 0$,
thus
\[
c ([S^1 \times B], \overline{F}) = c (1, \overline{F}) = 0
\]
for any generating function quadratic at infinity
$\overline{F}: S^1 \times B \times \mathbb{R}^N \rightarrow \mathbb{R}$
for $\varphi$.
By Proposition \ref{proposition: properties Viterbo}
(\ref{properties spectral sel gf: non-degeneracy}),
$\overline{F}$ generates the zero section of $T^{\ast} (S^1 \times B)$,
hence the graph in $J^1 (S^1 \times B)$
of the $1$-jet of a constant function.
But, by definition,
$\overline{F}$ generates the Lagrangian submanifold $\overline{L_{\varphi}}$
of $T^{\ast}_{-d\theta} (S^1 \times B)$
and so, by Proposition \ref{proposition: relation gf 1 jet},
the Legendrian submanifold $U_{-d\theta} (\widetilde{\overline{L_{\varphi}}})$
of $J^1 (S^1 \times B)$.
Thus,
$U_{-d\theta} (\widetilde{\overline{L_{\varphi}}})$
is the graph in $J^1 (S^1 \times B)$
of the $1$-jet of a constant function.
Since $U_{-d\theta} (\widetilde{\overline{L_{\varphi}}})$
intersects the zero section,
we conclude that it is the zero section.
But this implies that $\widetilde{\overline{L_{\varphi}}}$ is the zero section,
thus $\overline{L_{\varphi}}$ is the zero section
and so $\varphi$ is the identity.

By the Poincar\'e duality property
(\ref{properties spectral sel diffeomorphisms: PD}),
it is enough to prove the triangle inequality
(\ref{properties spectral sel diffeomorphisms: triangle inequality})
for $c_+$.
This can be done as follows.
Let $\overline{F_{\varphi_1}}$,
$\overline{F_{\varphi_2}}$ and $\overline{F_{\varphi_1 \circ \varphi_2}}$
be generating functions quadratic at infinity
for $\varphi_1$, $\varphi_2$ and $\varphi_1 \circ \varphi_2$ respectively.
We first show that,
for any non-zero class $u$,
\begin{equation}\label{equation: in proof triangle inequality}
c (u, \overline{F_{\varphi_2}})
= c (u, \overline{F_{\varphi_1 \circ \varphi_2}} - \overline{F_{\varphi_1}}) \,.
\end{equation}
To see this,
let $\{\varphi_t\}$ be a lcs Hamiltonian isotopy
with time-$1$ map $\varphi_1$,
let $\overline{F_t}$ and $\overline{G_t}$
be $1$-parameter families of generating functions quadratic at infinity
for $\{\varphi_t\}$ and $\{\varphi_t \circ \varphi_2\}$ respectively,
and set
\[
c_t = c (u \,,\, \overline{G_t} - \overline{F_t}) \,.
\]
By Proposition \ref{proposition: action spectrum} (ii),
for every $t$ the set of critical values of $\overline{G_t} - \overline{F_t}$
is equal to the action spectrum of $\varphi_2$.
The functions $\overline{G_t} - \overline{F_t}$
have thus all the same set $S$ of critical values.
By Proposition \ref{proposition: properties Viterbo}
(\ref{properties spectral sel gf: spectrality}),
$c_t$ is in $S$ for all $t$.
By Proposition \ref{proposition: properties Viterbo}
(\ref{properties spectral sel gf: continuity})
the map $t \mapsto c_t$ is continuous. 
Since $S$ is nowhere dense
we conclude that $t \mapsto c_t$ is constant,
and so $c_0 = c (u, \overline{G_0} - \overline{F_0})$
is equal to $c_1 = c (u, \overline{G_1} - \overline{F_1})$.
Since $\overline{G_0} - \overline{F_0}$
is equivalent to $\overline{F_{\varphi_2}}$
and $\overline{G_1} - \overline{F_1}$ is equivalent to 
$\overline{F_{\varphi_1 \circ \varphi_2}} - \overline{F_{\varphi_1}}$,
by Proposition \ref{proposition: properties Viterbo}
(\ref{properties spectral sel gf: equivalence})
we obtain \eqref{equation: in proof triangle inequality}.
Using \eqref{equation: in proof triangle inequality}
and Proposition \ref{proposition: properties Viterbo}
(\ref{properties spectral sel gf: PD})
and (\ref{properties spectral sel gf: sum})
we thus conclude that
\[
c_+(\varphi_2) = c ([S^1 \times B], \overline{F_{\varphi_2}})
= c ([S^1 \times B], \overline{F_{\varphi_1 \circ \varphi_2}} - \overline{F_{\varphi_1}})
\geq c ([S^1 \times B], \overline{F_{\varphi_1 \circ \varphi_2}}) + c (1, - \overline{F_{\varphi_1}})
\]
\[
= c ([S^1 \times B], \overline{F_{\varphi_1 \circ \varphi_2}})
- c ([S^1 \times B], \overline{F_{\varphi_1}})
= c_+ (\varphi_1 \circ \varphi_2) - c_+ (\varphi_1) \,.
\]

The proof of the invariance by conjugation property
(\ref{properties spectral sel diffeomorphisms: invariance by conjugation})
is analogous to the proof of \cite[Theorem 1.1 (i)]{Bhupal}
and \cite[Lemma 3.16]{San11}.
Suppose first that $\varphi$ and $\psi$
are compactly supported lcs Hamiltonian diffeomorphisms
of $S^1 \times \mathbb{R}^{2n} \times S^1$.
Let $\{\psi_t\}$ be a compactly supported lcs Hamiltonian isotopy
with $\psi_1 = \psi$,
let $\overline{F_t}: E \rightarrow \mathbb{R}$
a $1$-parameter family of generating functions quadratic at infinity
for $\{\psi_t \circ \varphi \circ \psi_t^{-1} \}$,
and set $c_t = c (u, \overline{F_t})$,
where $u$ is either the fundamental or the unit class of $S^1 \times S^{2n} \times S^1$.
The result follows if we prove that
if $c_{t_0} = k \in \mathbb{Z}$ for some $t_0$
then $c_t = k$ for all $t$.
Let $x_t$ be a path in $E$,
for $t$ in some subinterval of $[0, 1]$ containing $t_0$,
such that each $x_t$ is a critical point of $\overline{F_t}$
with critical value $c_t$.
Let $q_t$ be the corresponding path of essential translated points
of $\psi_t \circ \varphi \circ \psi_t^{-1}$.
Assume first that $q_{t_0}$ is a non-degenerate translated point
of $\psi_{t_0} \circ \varphi \circ \psi_{t_0}^{-1}$.
Since $q_0$ has time-shift $c_{t_0} = k \in \mathbb{Z}$,
by Lemma \ref{lemma: translated points of conjugation}
for every $t$ the point $\psi_t \circ \psi_{t_0}^{-1} (q_{t_0})$
is an essential non-degenerate translated point
of $\psi_t \circ \varphi \circ \psi_t^{-1}$
of time-shift $k$.
Let $y_t$ be the corresponding path of non-degenerate critical points
of $\overline{F_t}$ of critical value $k$.
Since $y_{t_0} = x_{t_0}$
we deduce that $x_t = y_t$ for all $t$,
and so $c_t = k$ for all t.
This ends the proof for $S^1 \times \mathbb{R}^{2n} \times S^1$
under the assumption that $q_{t_0}$
is a non-degenerate translated point
of $\psi_{t_0} \circ \varphi \circ \psi_{t_0}^{-1}$.
If $q_{t_0}$ is degenerate,
we obtain the result by an approximation argument as in \cite{Bhupal}.
For every $\epsilon > 0$
there is a generating function quadratic at infinity
$G: E \rightarrow \mathbb{R}$
such that $\lVert G - \overline{F_{t_0}} \rVert_{\mathcal{C}^{\infty}} < \epsilon$,
$G \equiv \overline{F_{t_0}}$ on an neighborhood
of $S^1 \times \{p_{\infty}\} \times S^1 \times \mathbb{R}^N$,
all the critical points of $G$ outside this neighborhood are non-degenerate
and $c (u, G) = c (u, \overline{F_{t_0}})$.
Since $\lVert G - \overline{F_{t_0}} \rVert_{\mathcal{C}^{\infty}} < \epsilon$,
$G$ is a generating function quadratic at infinity
for a compactly supported lcs Hamiltonian diffeomorphism $\phi_{t_0}$
of $S^1 \times \mathbb{R}^{2n} \times S^1$ with
$\lVert \phi_{t_0} \circ (\psi_{t_0} \circ \varphi \circ \psi_{t_0})^{-1} \rVert_{\mathcal{C}^{\infty}}
< \epsilon$.
Let $\phi_t = (\psi_t \circ \psi_{t_0}) \circ \phi_{t_0} \circ (\psi_t \circ \psi_{t_0})^{-1}$.
Then
$\lVert \phi_t \circ (\psi_t \circ \varphi \circ \psi_t)^{-1} \rVert_{\mathcal{C}^{\infty}}
< \epsilon$.
Let $G_t$ be a $1$-parameter family of generating functions quadratic at infinity
for $\{\phi_t\}$.
By the first part of the proof,
$c (u, G_t) = k$ for all $t$.
But $c (u, G_t)$ is equal to $c_+(\phi_t)$ or $c_- (\phi_t)$.
Since $\epsilon$ is arbitrary,
by continuity (\ref{properties spectral sel diffeomorphisms: continuity})
we thus obtain that $c_t = k$ for all $t$.
For compactly supported lcs Hamiltonian diffeomorphisms
of $S^1 \times \mathbb{R}^{2n+1}$
a similar argument works for $k = 0$,
showing that $c_+ (\varphi) = 0$
if and only if $c_+ (\psi \circ \varphi \circ \psi^{-1}) = 0$
and similarly for $c_-$.
\end{proof}

Let $\phi$ be a compactly supported
contactomorphism of $\mathbb{R}^{2n+1}$ or $\mathbb{R}^{2n} \times S^1$
contact isotopic to the identity,
and $\overline{F}: B \times \mathbb{R}^N \rightarrow \mathbb{R}$,
where $B$ is either $S^{2n+1}$ or $S^{2n} \times S^1$,
a generating function quadratic at infinity for $\phi$.
As in \cite{Bhupal, San10, San11},
we define $c_+ (\phi) = c ( [B], \overline{F} )$
and $c_- (\phi) = c ( 1, \overline{F} )$.
The spectral selectors $c_+$ and $c_-$
on $\Cont_0^c (\mathbb{R}^{2n+1})$
and $\Cont_0^c (\mathbb{R}^{2n} \times S^1)$
enjoy properties similar to those discussed above
for the spectral selectors $c_+$ and $c_-$
on $\Ham^c (S^1 \times \mathbb{R}^{2n+1})$
and $\Ham^c (S^1 \times \mathbb{R}^{2n} \times S^1)$.
Except for the analogues of (\ref{properties spectral sel diffeomorphisms: PD})
and (\ref{properties spectral sel diffeomorphisms: triangle inequality}),
these properties are proved in \cite{Bhupal, San10, San11}.
As discussed in the next remark,
the contact analogues of (\ref{properties spectral sel diffeomorphisms: PD})
and (\ref{properties spectral sel diffeomorphisms: triangle inequality})
are in fact stronger than the corresponding properties
proved in \cite{Bhupal, San10, San11}.

\begin{remark}\label{remark: properties for contactomorphisms}
Arguments as in the proof
of Proposition \ref{proposition: properties of spectral invariants}
(\ref{properties spectral sel diffeomorphisms: PD})
and (\ref{properties spectral sel diffeomorphisms: triangle inequality})
also apply
to the spectral selectors $c_{\pm}$
of compactly supported contactomorphisms
of $\mathbb{R}^{2n+1}$ and $\mathbb{R}^{2n} \times S^1$,
proving in particular that $c_-(\phi) = - c_+ (\phi^{-1})$,
$c_+ (\phi_1 \circ \phi_2) \leq c_+ (\phi_1) + c_+ (\phi_2)$
and $c_- (\phi_1 \circ \phi_2) \geq c_- (\phi_1) + c_- (\phi_2)$.
In \cite{San11} and \cite{San10}
only weaker versions of these properties are proved:
$\lfloor c_-(\phi) \rfloor = - \lceil c_+ (\phi^{-1}) \rceil$,
$\lceil c_+ (\phi_1 \circ \phi_2) \rceil \leq
\lceil c_+ (\phi_1) \rceil + \lceil c_+ (\phi_2) \rceil$
and $\lfloor c_- (\phi_1 \circ \phi_2) \rfloor
\geq \lfloor c_- (\phi_1) \rfloor + \lfloor c_- (\phi_2) \rfloor$.
It turns out that this difference is due to the fact that
in the present article we use a different formula
for the identification $\tau_{\cont}$
of the contact product of $\mathbb{R}^{2n+1}$
with $J^1\mathbb{R}^{2n+1}$.
If we would use the identification of \cite{San11, San10},
in the contact analogue of Proposition \ref{proposition: action spectrum} (ii)
we would obtain that
for any two compactly supported contactomorphisms
$\phi_1$ and $\phi_2$ of $\mathbb{R}^{2n+1}$ or $\mathbb{R}^{2n} \times S^1$
contact isotopic to the identity
and generating functions quadratic at infinity
$\overline{F_1}$ and $\overline{F_2}$
for $\phi_1$ and $\phi_2$,
the set of critical values
of $\overline{F_1} - \overline{F_2}$
is equal to the action spectrum of $\phi_1 \circ \phi_2^{-1}$,
not of $\phi_2^{-1} \circ \phi_1$.
Thus in the contact analogue of the proof
of Proposition \ref{proposition: properties of spectral invariants}
(\ref{properties spectral sel diffeomorphisms: PD})
the set of critical values of $\overline{G_t} - \overline{F_t}$
would not be equal to the action spectrum of $\varphi^{-1}$
but to the action spectrum of $\varphi_t \circ \varphi^{-1} \circ \varphi_t^{-1}$
and so we would need to integrate in the proof
also arguments as in the proof of 
Proposition \ref{proposition: properties of spectral invariants}
(\ref{properties spectral sel diffeomorphisms: invariance by conjugation}),
where the floor and ceiling functions
are necessary.
Something similar also happens
for the triangle inequality.
\end{remark}

We now discuss the relation between the spectral selectors
of lcs Hamiltonian diffeomorphisms
and those of contactomorphisms.

\begin{prop}\label{proposition: relation spectral invariants contactomorphisms}
Let $\phi$ be a compactly supported contactomorphism
of $\mathbb{R}^{2n+1}$ or $\mathbb{R}^{2n} \times S^1$
contact isotopic to the identity,
and $\widetilde{\phi}$ its lift to $S^1 \times \mathbb{R}^{2n+1}$
or $S^1 \times \mathbb{R}^{2n} \times S^1$.
Then $c_+ (\widetilde{\phi}) = c_+ (\phi)$
and $c_- (\widetilde{\phi}) = c_- (\phi)$.
\end{prop}

\begin{proof}
By the Poincar\'e duality property,
Proposition \ref{proposition: properties of spectral invariants}
(\ref{properties spectral sel diffeomorphisms: PD}),
and the analogous property for the spectral selectors of contactomorphisms
(cf.\ Remark \ref{remark: properties for contactomorphisms}),
it is enough to prove that $c_- (\widetilde{\phi}) = c_- (\phi)$
for any compactly supported contactomorphism $\phi$
of $\mathbb{R}^{2n+1}$ or $\mathbb{R}^{2n} \times S^1$
contact isotopic to the identity.
Let $\overline{F}: B \times \mathbb{R}^N \rightarrow \mathbb{R}$
be a generating function quadratic at infinity for $\phi$,
where $B$ is either $S^{2n+1}$ or $S^{2n} \times S^1$.
By Proposition \ref{proposition: relation gf contact},
the function $\overline{F} \circ \pr_2:
S^1 \times (B \times \mathbb{R}^N) \rightarrow \mathbb{R}$,
where $\pr_2$ is the projection to the second factor,
is a generating function quadratic at infinity for $\widetilde{\phi}$.
Denote $E_B = B \times \mathbb{R}^N$ and $E = S^1 \times B \times \mathbb{R}^N$,
and similarly $E_B^{a} = \{ \overline{F} \leq a \}$
and $E^{a} = \{ \overline{F} \circ \pr_2 \leq a \}$.
Then $E = S^1 \times E_B$ and $E^{a} = S^1 \times E_B^{a}$.
Consider the commutative diagram
\[
\begin{tikzcd}
H^{\ast} (B)
\arrow{r}{\ex \circ T} \arrow{d}
& H^{\ast + \ind (\overline{F}_{\infty})} (E_B, E_B^{-\infty})
\arrow{r}{i_a} \arrow{d}
&  H^{\ast + \ind (\overline{F}_{\infty})} (E_B^a, E_B^{-\infty})
\arrow{d}\\
H^{\ast} (S^1 \times B)
\arrow{r}{\ex \circ T}
& H^{\ast + \ind (\overline{F}_{\infty})}
(E, E^{-\infty}) \arrow{r}{i_a}
&  H^{\ast + \ind (\overline{F}_{\infty})}
(E^a, E^{-\infty})
\end{tikzcd}
\]
where the vertical arrows are the homomorphisms
induced by the projection $S^1 \times B \rightarrow B$
and the induced projections on fibre bundles.
Since the first vertical arrow 
sends the unit class of $H^{\ast}(B)$
to the unit class of $H^{\ast} (S^1 \times B)$
and since the last vertical arrow is injective,
we deduce that the image of the unit class of $H^{\ast}(B)$
by $i_a^{\,\ast} \,\circ\, \ex \,\circ\, T$ is non-zero
if and only if the image of the unit class of $H^{\ast}(S^1 \times B)$
by $i_a^{\,\ast} \,\circ\, \ex \,\circ\, T$ is non-zero.
This implies that $c (1, \overline{F} \circ \pr_2) = c (1, \overline{F})$,
and so $c_- (\widetilde{\phi}) = c_- (\phi)$.
\end{proof}

Recall that the spectral selectors
for compactly supported Hamiltonian symplectomorphisms of $\mathbb{R}^{2n}$
are defined in \cite{Viterbo} by
\[
c_+ (\varphi) = c ([S^{2n}], \overline{F})
\quad \text{ and } \quad
c_- (\varphi) = c (1, \overline{F})
\]
for any generating function quadratic at infinity
$\overline{F}: S^{2n} \times \mathbb{R}^N \rightarrow \mathbb{R}$
for $\varphi$.
For any $\varphi$ we have
\[
c_{\pm} (\widetilde{\varphi}) = c_{\pm} (\varphi) \,,
\]
where $\widetilde{\varphi}$ denotes the lift of $\varphi$
to $\mathbb{R}^{2n+1}$ or $\mathbb{R}^{2n} \times S^1$.
Indeed, this follows from the fact that
the Legendrian submanifold $\Lambda_{\widetilde{\varphi}}$
of $J^1\mathbb{R}^{2n+1} \equiv T^{\ast}\mathbb{R} \times J^1\mathbb{R}^{2n}$
is equal to $0_{\mathbb{R}} \times \widetilde{L_{\varphi}}$,
where $\widetilde{L_{\varphi}}$ is the lift
of $L_{\varphi} \subset T^{\ast}\mathbb{R}^{2n}$
to $J^1\mathbb{R}^{2n}$
(cf. \cite[Proposition 3.18]{San11}).

As a first application of our spectral selectors
we obtain a proof of Theorem \ref{theorem: translated points - intro} (ii).

\begin{proof}[Proof of Theorem \ref{theorem: translated points - intro} (ii)]
Let $\varphi$ be a compactly supported lcs Hamiltonian diffeomorphisms
of $S^1 \times \mathbb{R}^{2n+1}$ or $S^1 \times \mathbb{R}^{2n} \times S^1$
with non-empty support.
We have to show that $\varphi$
has at least one essential translated point
in the interior of its support.
By Proposition \ref{proposition: properties of spectral invariants}
(\ref{properties spectral sel diffeomorphisms: non-degeneracy}),
either $c_+ (\varphi) \neq 0$ or $c_- (\varphi) \neq 0$.
But,
$c_+ (\varphi)$ and $c_- (\varphi)$
belong to the action spectrum of $\varphi$.
We thus conclude that $\varphi$ has at least one essential translated point
that has action (and time-shift) different from zero,
and thus belongs to the interior of the support of $\varphi$.
\end{proof}

\section{Bi-invariant partial orders on the Hamiltonian groups
of $S^1 \times \mathbb{R}^{2n+1}$ and $S^1 \times \mathbb{R}^{2n} \times S^1$,
and an integer-valued bi-invariant metric on the Hamiltonian group
of $S^1 \times \mathbb{R}^{2n} \times S^1$}
\label{section: order and metric}

Recall that for compactly supported lcs Hamiltonian diffeomorphisms
$\varphi_1$ and $\varphi_2$ of $S^1 \times \mathbb{R}^{2n+1}$
or $S^1 \times \mathbb{R}^{2n} \times S^1$
we have posed $\varphi_1 \leq \varphi_2$
if $\varphi_2 \circ \varphi_1^{-1}$ is the time-$1$ map
of a non-negative compactly supported lcs Hamiltonian isotopy.
We now show that $\leq$ is a bi-invariant partial order,
and so $S^1 \times \mathbb{R}^{2n+1}$ and $S^1 \times \mathbb{R}^{2n} \times S^1$
are strongly orderable.
We obtain this by comparing $\leq$
with a bi-invariant partial order $\preceq$
defined using the spectral selectors of the previous section,
which is the lcs analogue
of the partial orders on $\Ham^c (\mathbb{R}^{2n})$,
$\Cont_0^c (\mathbb{R}^{2n+1})$ and $\Cont_0^c (\mathbb{R}^{2n} \times S^1)$
defined in \cite{Viterbo},
\cite{Bhupal} and \cite{San11} respectively.

For compactly supported lcs Hamiltonian diffeomorphisms
$\varphi_1$ and $\varphi_2$ of $S^1 \times \mathbb{R}^{2n+1}$
or $S^1 \times \mathbb{R}^{2n} \times S^1$
we pose
\[
\varphi_1 \preceq \varphi_2 \quad \text{ if } \;
c_+ (\varphi_1 \circ \varphi_2^{-1}) = 0 \,.
\]

\begin{thm}\label{theorem: po 1}
The relation $\preceq$ on $\Ham^c (S^1 \times \mathbb{R}^{2n+1})$
or $\Ham^c (S^1 \times \mathbb{R}^{2n} \times S^1)$
is a bi-invariant partial order.
\end{thm}

\begin{proof}
Reflexivity follows from the fact that $c_+ (\id) = 0$.
Suppose that $\varphi_1 \preceq \varphi_2$ and $\varphi_2 \preceq \varphi_3$,
thus $c_+ (\varphi_1 \circ \varphi_2^{-1}) = c_+ (\varphi_2 \circ \varphi_3^{-1}) = 0$.
By Proposition \ref{proposition: properties of spectral invariants}
(\ref{properties spectral sel diffeomorphisms: positivity})
and (\ref{properties spectral sel diffeomorphisms: triangle inequality})
we then have
\[
0 \leq c_+ (\varphi_1 \circ \varphi_3^{-1})
\leq c_+ (\varphi_1 \circ \varphi_2^{-1}) + c_+ (\varphi_2 \circ \varphi_3^{-1}) = 0 \,,
\]
thus $c_+ (\varphi_1 \circ \varphi_3^{-1}) = 0$
and so $\varphi_1 \preceq \varphi_3$,
showing transitivity.
Suppose that $\varphi_1 \preceq \varphi_2$ and $\varphi_2 \preceq \varphi_1$,
thus $c_+ (\varphi_1 \circ \varphi_2^{-1}) = c_+ (\varphi_2 \circ \varphi_1^{-1}) = 0$.
By Proposition \ref{proposition: properties of spectral invariants}
(\ref{properties spectral sel diffeomorphisms: PD})
we then have
\[
c_- (\varphi_1 \circ \varphi_2^{-1})
= - c_+ (\varphi_2 \circ \varphi_1^{-1}) = 0 \,.
\]
It thus follows
from Proposition \ref{proposition: properties of spectral invariants}
(\ref{properties spectral sel diffeomorphisms: non-degeneracy})
that $\varphi_1 \circ \varphi_2^{-1}$ is the identity,
and so $\varphi_1 = \varphi_2$.
This proves anti-symmetry.
As for bi-invariance,
if $\varphi_1 \preceq \varphi_2$,
thus $c_+ (\varphi_1 \circ \varphi_2^{-1}) = 0$,
then for every $\psi$ we have
\[
c_+ \big( (\varphi_1 \circ \psi) \circ (\varphi_2 \circ \psi)^{-1}\big)
= c_+ (\varphi_1 \circ \varphi_2^{-1})= 0 \,,
\]
thus $\varphi_1 \circ \psi \preceq \varphi_2 \circ \psi$,
and, by Proposition \ref{proposition: properties of spectral invariants}
(\ref{properties spectral sel diffeomorphisms: invariance by conjugation}),
\[
c_+ \big( (\psi \circ \varphi_1) \circ (\psi \circ \varphi_2)^{-1}\big)
= c_+ \big( \psi \circ \varphi_1 \circ \varphi_2^{-1} \circ \psi^{-1} \big) = 0 \,,
\]
thus $\psi \circ \varphi_1 \preceq \psi \circ \varphi_2$.
\end{proof}

In order to deduce that the relation $\leq$
is also a partial order,
and so that $S^1 \times \mathbb{R}^{2n+1}$ and $S^1 \times \mathbb{R}^{2n} \times S^1$
are strongly orderable
(Theorem \ref{theorem: orderability - intro} (ii)),
we use the following comparison.

\begin{prop}\label{proposition: comparison po}
Let $\varphi_1$ and $\varphi_2$ be compactly supported
lcs Hamiltonian diffeomorphisms of $S^1 \times \mathbb{R}^{2n+1}$
or $S^1 \times \mathbb{R}^{2n} \times S^1$
with $\varphi_1 \leq \varphi_2$.
Then $\varphi_1 \preceq \varphi_2$.
\end{prop}

\begin{proof}
Since $\varphi_1 \leq \varphi_2$,
$\varphi_2 \circ \varphi_1^{-1}$ is the time-$1$ map
of a compactly supported non-negative lcs Hamiltonian isotopy.
By definition this implies that $\id \leq \varphi_2 \circ \varphi_1^{-1}$,
and so by Proposition \ref{proposition: properties of spectral invariants}
(\ref{properties spectral sel diffeomorphisms: monotonicity})
and (\ref{properties spectral sel diffeomorphisms: positivity})
we have
\[
0 = c_- (\id) \leq c_- (\varphi_2 \circ \varphi_1^{-1}) \leq 0 \,,
\]
thus $c_- (\varphi_2 \circ \varphi_1^{-1}) = 0$.
By Proposition \ref{proposition: properties of spectral invariants}
(\ref{properties spectral sel diffeomorphisms: PD})
this implies that $c_+ (\varphi_1 \circ \varphi_2^{-1}) = 0$,
and so $\varphi_1 \preceq \varphi_2$.
\end{proof}

We can now prove 
that $S^1 \times \mathbb{R}^{2n+1}$ and $S^1 \times \mathbb{R}^{2n} \times S^1$
are strongly orderable.

\begin{proof}[Proof of Theorem \ref{theorem: orderability - intro} (ii)]
We have to show that the relation $\leq$ on $\Ham^c (S^1 \times \mathbb{R}^{2n+1})$
or $\Ham^c (S^1 \times \mathbb{R}^{2n} \times S^1)$
is a bi-invariant partial order.
Reflexivity follows from the fact that $\{\id\}$
has Hamiltonian function $H_t \equiv 0$.
Transitivity and bi-invariance follow
from Lemmas \ref{lemma: composition law Hamiltonians}
and \ref{lemma: Hamiltonian conjugation}.
As for anti-symmetry,
if $\varphi_1 \leq \varphi_2$ and $\varphi_2 \leq \varphi_1$
then by Proposition \ref{proposition: comparison po}
we also have $\varphi_1 \preceq \varphi_2$ and $\varphi_2 \preceq \varphi_1$.
But, by Theorem \ref{theorem: po 1},
$\preceq$ is anti-symmetric.
We thus conclude that $\varphi_1 = \varphi_2$.
\end{proof}

Recall that a bi-invariant metric on a group $G$
is a map $d: G \times G \rightarrow \mathbb{R}$
that satisfies the following properties:
\renewcommand{\theenumi}{\roman{enumi}}
\begin{enumerate}
\item \label{properties metric: positivity} \emph{Positivity:}
$d (g_1, g_2) \geq 0$,
and $d (g, g) = 0$.
\item \label{properties metric: non-degeneracy} \emph{Non-degeneracy:}
$d(g_1,g_2) = 0$ if and only if $g_1 = g_2$.
\item \label{properties metric: symmetry} \emph{Symmetry:}
$d (g_1, g_2) = d (g_2, g_1)$.
\item \label{properties metric: triangle inequality} \emph{Triangle inequality:}
$d (g_1, g_3) \leq  d (g_1, g_2) + d (g_2, g_3)$.
\item \label{properties metric: bi-invariance} \emph{Bi-invariance:}
$d (hg_1, hg_2) = d (g_1h, g_2h) = d (g_1, g_2)$.
\end{enumerate}

In \cite{San10} the third author defined
an integer valued bi-invariant metric $d_{\cont}$
on $\Cont_0^c (\mathbb{R}^{2n} \times S^1)$
by posing
\[
d_{\cont} (\phi_1, \phi_2)
= \lceil c_+ (\phi_1 \circ \phi_2^{-1}) \rceil
- \lfloor c_- (\phi_1 \circ \phi_2^{-1}) \rfloor \,,
\]
and proved that this metric is unbounded.
This metric is the contact analogue
of the Viterbo metric $d_{\symp}$ on $\Ham^c (\mathbb{R}^{2n})$,
which is defined by
\[
d_{\symp} (\varphi_1, \varphi_2)
= c_+ (\varphi_1 \circ \varphi_2^{-1}) - c_- (\varphi_1 \circ \varphi_2^{-1})
\]
and is also unbounded.
Similarly,
for compactly supported lcs Hamiltonian diffeomorphisms
$\varphi_1$ and $\varphi_2$ of $S^1 \times \mathbb{R}^{2n} \times S^1$
we define
\[
d (\varphi_1, \varphi_2)
= \lceil c_+ (\varphi_1 \circ \varphi_2^{-1}) \rceil
- \lfloor c_- (\varphi_1 \circ \varphi_2^{-1}) \rfloor \,.
\]

\begin{thm}\label{theorem: metric}
The map $d$ is an unbounded bi-invariant metric
on $\Ham^c (S^1 \times \mathbb{R}^{2n} \times S^1)$.
\end{thm}

\begin{proof}
Positivity follows from
Proposition \ref{proposition: properties of spectral invariants}
(\ref{properties spectral sel diffeomorphisms: positivity})
and the fact that $c_{\pm} (\id) = 0$.
Suppose that
\[
d (\varphi_1, \varphi_2)
= \lceil c_+ (\varphi_1 \circ \varphi_2^{-1}) \rceil
- \lfloor c_- (\varphi_1 \circ \varphi_2^{-1}) \rfloor = 0 \,.
\]
Since $c_+ (\varphi_1 \circ \varphi_2^{-1}) \geq 0$
and $c_- (\varphi_1 \circ \varphi_2^{-1}) \leq 0$,
this implies that
$c_+ (\varphi_1 \circ \varphi_2^{-1}) = c_- (\varphi_1 \circ \varphi_2^{-1}) = 0$ and,
by Proposition \ref{proposition: properties of spectral invariants}
(\ref{properties spectral sel diffeomorphisms: non-degeneracy}), that $\varphi_1 \circ \varphi_2^{-1}$ is the identity, or, equivalently, that  $\varphi_1 = \varphi_2$.
This proves non-degeneracy.
Symmetry, the triangle inequality and bi-invariance
follow from Proposition \ref{proposition: properties of spectral invariants}
(\ref{properties spectral sel diffeomorphisms: PD}),
(\ref{properties spectral sel diffeomorphisms: triangle inequality})
and (\ref{properties spectral sel diffeomorphisms: invariance by conjugation})
respectively.
Observe now that,
by Proposition \ref{proposition: relation spectral invariants contactomorphisms},
if $\phi_1$ and $\phi_2$ are compactly supported
contactomorphisms of $\mathbb{R}^{2n} \times S^1$
contact isotopic to the identity
and $\widetilde{\phi_1}$ and $\widetilde{\phi_2}$
their lifts to $S^1 \times \mathbb{R}^{2n} \times S^1$
then $d (\widetilde{\phi_1}, \widetilde{\phi_2})= d_{\cont} (\phi_1, \phi_2)$.
Since $d_{\cont}$ is unbounded,
we conclude $d$ is also unbounded.
\end{proof}

Recall that a partially ordered metric space
is a metric space $(Z, d)$
endowed with a partial order $\leq$
that is compatible with $d$,
in the sense that for all $a, b, c \in Z$
with $a \leq b \leq c$
we have $d (a, b) \leq d (a, c)$.

\begin{thm}\label{theorem: partiall ordered metric spaces}
$(\Ham^c (S^1 \times \mathbb{R}^{2n} \times S^1), d, \leq)$
and $(\Ham^c (S^1 \times \mathbb{R}^{2n} \times S^1), d, \preceq)$
are partially ordered metric spaces.
\end{thm}

\begin{proof}
By Proposition \ref{proposition: comparison po},
it is enough to prove that $\preceq$ is compatible with $d$.
For this,
suppose that $\varphi_1 \preceq \varphi_2 \preceq \varphi_3$.
Then $c_+ (\varphi_1 \circ \varphi_2^{-1})
= c_+ (\varphi_2 \circ \varphi_3^{-1})
= c_
+ (\varphi_1 \circ \varphi_3^{-1}) = 0$.
Thus, using Proposition \ref{proposition: properties of spectral invariants}
(\ref{properties spectral sel diffeomorphisms: PD})
and (\ref{properties spectral sel diffeomorphisms: triangle inequality})
we get
\begin{eqnarray*}
d (\varphi_1, \varphi_2)
&=&\lceil c_+ (\varphi_1 \circ \varphi_2^{-1})\rceil
- \lfloor c_- (\varphi_1 \circ \varphi_2^{-1})\rfloor
= - \lfloor c_- (\varphi_1 \circ \varphi_2^{-1})\rfloor
= \lceil c_+ (\varphi_2 \circ \varphi_1^{-1})\rceil \\
&\leq& \lceil c_+(\varphi_2 \circ \varphi_3^{-1})\rceil
+ \lceil c_+(\varphi_3 \circ \varphi_1^{-1})\rceil
=\lceil c_+(\varphi_3 \circ \varphi_1^{-1})\rceil \\
&=& \lceil c_+ (\varphi_1 \circ \varphi_3^{-1})\rceil
+ \lceil c_+ (\varphi_3 \circ \varphi_1^{-1})\rceil
= \lceil c_+ (\varphi_1 \circ \varphi_3^{-1})\rceil
- \lfloor c_- (\varphi_1 \circ \varphi_3^{-1})\rfloor \\
&=& d(\varphi_1, \varphi_3) \,.
\end{eqnarray*}
\end{proof}

\begin{remark}\label{Rmk:Finemetrics}
By an argument of Polterovich,
which is written down in \cite[Proposition 5.1]{Sandon - SP},
the identity component of the group of compactly supported contactomorphisms
of any contact manifold
does not admit any bi-invariant metric 
taking values arbitrarily close to zero.
A similar argument also works in the lcs case
to show that the group of compactly supported lcs  diffeomorphisms
of any lcs manifold $\big( M , [(\eta, \omega)] \big)$
does not admit any bi-invariant metric
that takes values arbitrarily close to zero.
Indeed,
let $B$ be a Darboux ball in $M$
and let $\psi_1$ and $\psi_2$ be two lcs  diffeomorphisms
that are supported in $B$ and do not commute,
i.e. $[\psi_1,\psi_2] := \psi_1 \circ \psi_2 \circ \psi_1^{-1} \circ \psi_2^{-1} \neq \id$. 
Suppose by contradiction that $d$
is bi-invariant metric on the group of compactly supported lcs diffeomorphisms
of $\big( M , [(\eta, \omega)] \big)$
that takes values arbitrarily close to zero,
and consider the associated conjugation-invariant norm
$\lVert \,\cdot\, \rVert = d (\,\cdot\,,\id)$.
Let $\varphi$ be a compactly supported lcs  diffeomorphism
with $\lVert \varphi \rVert = \epsilon$ for $\epsilon$ arbitrarily small.
Since $\varphi$ is not the identity,
there is a small ball $B_{\varphi}$ in $M$
that is displaced by $\varphi$.
Since it is possible to squeeze any given domain
of $\big(\mathbb{R}^{2n}, [(0, \omega_0)] \big)$
into an arbitrarily small one
by the Liouville flow
(which is a lcs isotopy),
we can find a compactly supported lcs diffeomorphism $\theta$ of $M$
such that $\theta (B) \subset B_{\phi}$.
By Lemma \ref{lemma: Hamiltonian conjugation},
$\theta \circ \psi_1 \circ \theta^{-1}$
and $\theta \circ \psi_2 \circ \theta^{-1}$
are lcs  diffeomorphisms supported in $B_{\phi}$.
Since we assume that $\lVert \, \cdot\,\rVert$ is conjugation-invariant
and $[\psi_1,\psi_2] \neq \id$,
we have
$$
\lVert \, [\, \theta \circ \psi_1 \circ \theta^{-1} \,,\,
\theta \circ \psi_2 \circ \theta^{-1} \, ] \, \rVert
= \lVert \, \theta \circ [\psi_1,\psi_2] \circ \theta^{-1} \, \rVert 
= \lVert \, [\psi_1,\psi_2] \, \rVert \neq 0 \,.
$$
On the other hand,
we also have
$$
\lVert \, [\, \theta \circ \psi_1 \circ \theta^{-1} \,,\,
\theta \circ \psi_2 \circ \theta^{-1} \,] \,\rVert
\leq 4 \, \lVert \varphi \rVert
= 4 \epsilon \,.
$$
Indeed,
as in \cite[Proposition 5.1]{Sandon - SP},
this follows from the fact that if $a_1$ and $a_2$
are supported in a domain $\mathcal{U}$
and $b$ displaces $\mathcal{U}$
then for every conjugation-invariant norm $\lVert \, \cdot\,\rVert$
we have $\lVert [a_1,a_2]\rVert \leq 4 \, \lVert b\rVert$.
Since $\epsilon$ was chosen arbitrarily small,
we thus obtain a contradiction. As the Liouville flow is not Hamiltonian, it is unclear to us whether the group of compactly supported lcs Hamiltonian diffeomorphisms admits a bi-invariant metric that takes values arbitrarily close to zero. 
\end{remark}

\section{A lcs capacity for domains of $S^1 \times \mathbb{R}^{2n} \times S^1$,
and proof of lcs non-squeezing}\label{section: capacity}

We define the \emph{lcs capacity}
of an open and bounded domain $\mathcal{U}$
of $S^1 \times \mathbb{R}^{2n} \times S^1$
by
\[
c (\mathcal{U})
= \sup \{\, \lceil c_+ (\varphi) \rceil \;\lvert \; \varphi \in \Ham (\mathcal{U})\,\} \,,
\]
where $\Ham (\mathcal{U})$ denotes
the group of lcs diffeomorphisms of $S^1 \times \mathbb{R}^{2n} \times S^1$
that are the time-$1$ map
of a lcs Hamiltonian isotopy supported in $\mathcal{U}$.
By the following lemma,
$c (\mathcal{U})$ is a well-defined integer number.

\begin{lemma}\label{lemma: displacement}
For every compactly supported
lcs Hamiltonian diffeomorphism $\psi$
of $S^1 \times \mathbb{R}^{2n} \times S^1$
such that $\psi (\mathcal{U}) \cap \mathcal{U} = \emptyset$
we have
\[
c (\mathcal{U}) \leq d (\id, \psi) \,.
\]
\end{lemma}

\begin{proof}
We first prove that
\begin{equation}\label{equation: in proof displacement}
\lfloor c_+ (\psi \circ \varphi) \rfloor = \lfloor c_+ (\psi) \rfloor
\end{equation}
for every $\varphi \in \Ham (\mathcal{U})$.
Let $\{\varphi_t\}$ be a lcs Hamiltonian isotopy
supported in $\mathcal{U}$ with $\varphi_1 = \varphi$,
let $\overline{F_t}: E \rightarrow \mathbb{R}$ be a $1$-parameter family
of generating functions quadratic at infinity
for $\{ \psi \circ \varphi_t \}$,
and set
$$
c_t = c_+ (\psi \circ \varphi_t) = c ( [S^1 \times S^{2n} \times S^1], \overline{F_t} ) \,.
$$
The results follows if we prove that if $c_{t_0} = k \in \mathbb{Z}$ for some $t_0$ 
then $c_t = k$ for all $t$.
Let $x_t$ be a path in $E$,
for $t$ in a subinterval of $[0,1]$ containing $t_0$,
such that each $x_t$ is a critical point of $\overline{F_t}$
with critical value $c_t$.
Let $q_t$ be the corresponding path
of essential translated points of $\psi \circ \varphi_t$.
As in the proof of Proposition \ref{proposition: properties of spectral invariants}
(\ref{properties spectral sel diffeomorphisms: invariance by conjugation}),
without loss of generality we can assume that
$q_{t_0}$ is a non-degenerate essential translated point of $\psi \circ \varphi_{t_0}$.
Since $q_{t_0}$ has time-shift $k \in \mathbb{Z}$,
we have $\psi \circ \varphi_{t_0} (q_{t_0}) = q_{t_0}$.
Since $\varphi_{t_0}$ is supported in $\mathcal{U}$
and $\psi (\mathcal{U}) \cap \mathcal{U} = \emptyset$,
this implies that $q_{t_0} \notin \mathcal{U}$
and so $\varphi_t (q_{t_0}) = q_{t_0}$ for all $t$
and $\psi (q_{t_0}) = q_{t_0}$,
thus $\psi \circ \varphi_t (q_{t_0}) = q_{t_0}$ for all $t$.
Let $g_t$ and $f$ be the conformal factors
of $\varphi_t$ and $\psi$ respectively.
Then the conformal factor of $\psi \circ \varphi_t$
is $f \circ \varphi_t + g_t$.
Since $q_{t_0}$ is a translated point of $\psi \circ \varphi_{t_0}$
and since $g_t (q_{t_0}) = 0$ for all $t$
(because $q_{t_0}$ does not belong to the support of $\varphi_t$),
we have
\[
0 = (f \circ \varphi_{t_0} + g_{t_0}) (q_{t_0}) = f (q_{t_0}) \,,
\]
and so
\[
(f \circ \varphi_t + g_t) (q_{t_0}) = f(q_{t_0}) = 0
\]
for all $t$.
We conclude that $q_{t_0}$ is a translated point of $\psi \circ \varphi_t$
of time-shift $k$, for all $t$.
The action of $\psi \circ \varphi_t$
is $S_{\varphi_t} - e^{-g_t} S_{\psi} \circ \varphi_t$ \,.
Since $q_{t_0}$ is an essential translated point
of $\psi \circ \varphi_{t_0}$,
by Example \ref{example: action translated point in conformal symplectization}
we have
$$
0 = (S_{\varphi_{t_0}} - e^{-g_{t_0}} S_{\psi} \circ \varphi_{t_0}) (q_{t_0})
= S_{\varphi_{t_0}} (q_{t_0}) - S_{\psi} (q_{t_0}) \,.
$$
Since moreover $S_{\varphi_t} (q_{t_0}) = 0$ for all $t$
(by Remark \ref{remark: unique solution},
since $q_{t_0}$ does not belong to the support of $\varphi_t$
and the Lee form $-d\theta$ is not exact
on the complement of the support of $\varphi_t$),
we thus have
\[
(S_{\varphi_t} - e^{-g_t} S_{\psi} \circ \varphi_t) (q_{t_0}) = - S_{\psi} (q_{t_0}) = 0
\]
for all $t$.
By Example \ref{example: action translated point in conformal symplectization},
we conclude that $q_{t_0}$ is an essential translated point
of $\psi \circ \varphi_t$ for all $t$.
Moreover,
the fact that $q_{t_0}$ is a non-degenerate translated point of $\psi \circ \varphi_{t_0}$
implies that it is a non-degenerate translated point of $\psi \circ \varphi_t$ for all $t$
(using again that $q_{t_0}$ does not belong to the support of $\varphi_t$).
Let $y_t$ be the corresponding path of non-degenerate critical points of $\overline{F_t}$
of critical value $k$.
Since $y_{t_0} = x_{t_0}$,
we deduce that $y_t = x_t$ for all $t$,
and so $c_t = k$ for all $t$.
This proves \eqref{equation: in proof displacement}.
Using \eqref{equation: in proof displacement}
and Proposition \ref{proposition: properties of spectral invariants}
(\ref{properties spectral sel diffeomorphisms: PD})
and (\ref{properties spectral sel diffeomorphisms: triangle inequality})
we then obtain
\[
\lceil c_+ (\varphi) \rceil
\leq \lceil c_+ (\psi \circ \varphi) \rceil + \lceil c_+ (\psi^{-1}) \rceil
= \lceil c_+ (\psi) \rceil - \lfloor c_- (\psi) \rfloor
= d (\id, \psi)
\]
for every $\varphi \in \Ham (\mathcal{U})$,
and so $c(\mathcal{U}) \leq d (\id, \psi)$.
\end{proof}

We define the \emph{displacement energy}
of an open and bounded domain $\mathcal{U}$
of $S^1 \times \mathbb{R}^{2n} \times S^1$
by
\[
E (\mathcal{U}) = \inf \{\, d (\id, \psi) \;\lvert\;
\psi (\mathcal{U}) \cap \mathcal{U} \neq \emptyset \,\} \,.
\]
Lemma \ref{lemma: displacement} then also implies
the energy--capacity inequality
\[
c (\mathcal{U}) \leq E (\mathcal{U}) \,.
\]

\begin{prop}\label{proposition: properties lcs capacity}
The lcs capacity $c$ satisfies the following properties:
\renewcommand{\theenumi}{\roman{enumi}}
\begin{enumerate}
\item \label{properties capacity: monotonicity} \emph{Monotonicity:}
if $\mathcal{U}_1 \subset \mathcal{U}_2$
then $c (\mathcal{U}_1) \leq c (\mathcal{U}_2)$.
\item \label{properties capacity: lcs invariance} \emph{Lcs invariance:}
for any compactly supported lcs Hamiltonian diffeomorphism
$\psi$ of $S^1 \times \mathbb{R}^{2n} \times S^1$,
we have $c (\mathcal{U}) = c \big(\psi (\mathcal{U}) \big)$.
\end{enumerate}
\end{prop}

\begin{proof}
Monotonicity holds by definition.
Lcs invariance follows from
Proposition \ref{proposition: properties of spectral invariants}
(\ref{properties spectral sel diffeomorphisms: invariance by conjugation}),
because $\varphi \in \Ham (\mathcal{U})$ if and only if
$\psi \circ \varphi \circ \psi^{-1} \in \Ham \big(\psi (\mathcal{U}) \big)$.
\end{proof}

As in \cite{Viterbo},
for any open and bounded domain $\mathcal{U}$ of $\mathbb{R}^{2n}$
we define
\[
c_{\symp} (\mathcal{U})
= \sup \{\, c_+ (\varphi) \;\lvert \; \varphi \in \Ham (\mathcal{U})\,\} \,,
\]
where $\Ham (\mathcal{U})$ denotes
the group of diffeomorphisms of $\mathbb{R}^{2n}$
that are the time-$1$ map
of a Hamiltonian isotopy supported in $\mathcal{U}$.
Similarly,
as in \cite{San11},
for any open and bounded domain $\mathcal{U}$ of $\mathbb{R}^{2n} \times S^1$
we define
\[
c_{\cont} (\mathcal{U})
= \sup \{\, \lceil c_+ (\phi) \rceil \;\lvert \; \phi \in \Cont (\mathcal{U})\,\} \,,
\]
where $\Cont (\mathcal{U})$ denotes
the group of diffeomorphisms of $\mathbb{R}^{2n} \times S^1$
that are the time-$1$ map
of a contact isotopy supported in $\mathcal{U}$.
For every open and bounded domain $\mathcal{U}$ of $\mathbb{R}^{2n}$
we then have
\begin{equation}\label{equation: comparison capacities symplectic and contact}
c_{\cont} (\mathcal{U} \times S^1) = \lceil c_{\symp} (\mathcal{U}) \rceil \,.
\end{equation}
We also have the following result.

\begin{prop}\label{proposition: comparison capacities lcs contact}
For every open and bounded domain
$\mathcal{U}$ of $\mathbb{R}^{2n} \times S^1$
we have
\[
c (S^1 \times \mathcal{U}) = c_{\cont} (\mathcal{U}) \,.
\]
\end{prop}

\begin{proof}
The result follows from Proposition \ref{proposition: properties of spectral invariants}
(\ref{properties spectral sel diffeomorphisms: monotonicity})
and Proposition \ref{proposition: relation spectral invariants contactomorphisms}
if we prove that for every $\varphi \in \Ham (S^1 \times \mathcal{U})$
there is $\phi \in \Cont (\mathcal{U})$
such that $\varphi \leq \widetilde{\phi}$.
This can be seen as follows.
Let $\{\varphi_t\}$ be a lcs Hamiltonian isotopy
of $S^1 \times \mathbb{R}^{2n} \times S^1$
with $\varphi_1 = \varphi$
and with Hamiltonian function $H_t$
supported in $S^1 \times \mathcal{U}$.
Let $h_t$ be the conformal factors of $\{\varphi_t\}$,
and $G$ a function on $\mathbb{R}^{2n} \times S^1$
supported in $\mathcal{U}$ such that
\[
G (p) \leq - e^{-h_t(\theta,p)} \, H_t \circ \varphi_t (\theta, p)
\quad \text{ for all } p, \theta \text{ and } t \,.
\]
Let $\{\chi_t\}$ be the contact isotopy
with Hamiltonian function $G$,
and $\{\widetilde{\chi_t}\}$ its lift to $S^1 \times \mathbb{R}^{2n} \times S^1$.
By Example \ref{example: lift of maps to symplectization}
and Example \ref{Example:Lift of contact Hamiltonians},
$\{\widetilde{\chi_t}\}$ has conformal factor
$\widetilde{g_t} (\theta, p) = g_t(p)$
and Hamiltonian function $\widetilde{G} (\theta, p) = G(p)$.
Let $\{\phi_t\}$ be the contact isotopy defined by $\phi_t = \chi_t^{-1}$.
Then $\{\widetilde{\phi_t}\}$
has conformal factor $- \widetilde{g_t} \circ \widetilde{\chi_t}^{-1}$
and, by Lemma \ref{lemma: composition law Hamiltonians} (i),
Hamiltonian function $- e^{-\widetilde{g_t}} \, \widetilde{G} \circ \widetilde{\chi_t}$.
By Lemma \ref{lemma: composition law Hamiltonians},
$\{\widetilde{\phi_t} \circ \varphi_t^{-1}\}$
has thus Hamiltonian function
\[
- e^{-\widetilde{g_t}} \, \widetilde{G} \circ \widetilde{\chi_t}
+ e^{-\widetilde{g_t}} \, (-e^{-h_t} H_t \circ \varphi_t) \circ \widetilde{\chi_t}
= - e^{-\widetilde{g_t}} \, (\widetilde{G} + e^{-h_t} H_t \circ \varphi_t) \circ \widetilde{\chi_t}
\geq 0 \,.
\]
This shows that $\varphi_1 \leq \widetilde{\phi_1}$,
as we wanted.
\end{proof}

In particular,
by Proposition \ref{proposition: comparison capacities lcs contact}
and \eqref{equation: comparison capacities symplectic and contact}
we have
\begin{equation}\label{equation: capacities balls}
c (S^1 \times B^{2n} (R) \times S^1)
= c_{\cont} (B^{2n} (R) \times S^1)
= \lceil c_{\symp} (B^{2n} (R)) \rceil
= \lceil \pi R^2 \rceil \,.
\end{equation}

Using this we now obtain a proof
of the lcs non-squeezing theorem for integers.

\begin{proof}[Proof of Theorem \ref{theorem: main}]
Suppose that $\pi R_2^2 \leq k \leq \pi R_1^2$ for $k \in \mathbb{N}_0$.
We have to show that there is no lcs squeezing
of $S^1 \times B^{2n}(R_1) \times S^1$
into $S^1 \times B^{2n}(R_2) \times S^1$,
i.e.\ no lcs Hamiltonian isotopy $\{\varphi_t\}_{t \in [0,1]}$
of $S^1 \times \mathbb{R}^{2n} \times S^1$
such that
\[
\varphi_1 \big(\overline{S^1 \times B^{2n}(R_1) \times S^1}\big)
\subset S^1 \times B^{2n}(R_2) \times S^1 \,.
\]
Suppose by contradiction that
such a lcs squeezing exists.
We then also have a lcs squeezing
of a neighborhood of $S^1 \times B^{2n}(R_1) \times S^1$
into $S^1 \times B^{2n}(R_2) \times S^1$,
so without loss of generality
we can assume that $\pi R_2^2 \leq k < \pi R_1^2$.
By \eqref{equation: capacities balls}
and Proposition \ref{proposition: properties lcs capacity}
we have
\[
c \Big( \varphi_1 \big(S^1 \times B^{2n}(R_1) \times S^1\big) \Big)
= c \big(S^1 \times B^{2n}(R_1) \times S^1\big)
= \lceil \pi R_1^2 \rceil > k
\]
and
\[
c \Big( \varphi_1 \big(S^1 \times B^{2n}(R_1) \times S^1\big) \Big)
\leq c \big(S^1 \times B^{2n}(R_2) \times S^1\big)
= \lceil \pi R_2^2 \rceil \leq k\,.
\]
This contradiction finishes the proof.
\end{proof}

\section{The Hamiltonian group of $S^1 \times \mathbb{R}^{2n+1} \times S^1$
does not have the Rokhlin property}\label{section: Rokhlin}

Recall that a topological group $G$
is said to have the Rokhlin property
if there exists $g \in G$
whose conjugacy class $\{\, hgh^{-1} \;\lvert\; h \in G \,\}$
is dense in $G$.
This notion has been introduced
by Glasner and Weiss \cite{GW1, GW2}.
Serraille and Stojisavljevi\'c \cite{SS}
discovered a connection
between the Rokhlin property for groups
of contactomorphisms or contact homeomorphisms
and contact non-squeezing.
More precisely,
they proved that the closure of the identity component
of the group of compactly supported contactomorphisms of $\mathbb{R}^{2n+1}$
inside the group of compactly supported homeomorphisms
endowed with the $\mathcal{C}^0$-topology
(with respect to the Euclidean distance)
has the Rokhlin property,
while the closure of the identity component
of the group of compactly supported contactomorphisms
of $\mathbb{R}^{2n} \times S^1$
inside the group of compactly supported homeomorphisms
endowed with the $\mathcal{C}^0$-topology
(with respect to the distance on
$\mathbb{R}^{2n} \times S^1$ = $\mathbb{R}^{2n} \times \mathbb{R}/\mathbb{Z}$
induced by the Euclidean distance on $\mathbb{R}^{2n+1}$)
does not have the Rokhlin property.
The proof of the Rokhlin property for $\mathbb{R}^{2n+1}$
uses the fact that any ball in $\mathbb{R}^{2n+1}$
can be squeezed by a compactly supported contact isotopy
into an arbitrarily small ball.
On the other hand,
the proof of the fact that the Rokhlin property
does not hold in the case of $\mathbb{R}^{2n} \times S^1$
is based on the contact non-squeezing theorem
of Eliashberg, Kim and Polterovich.
In fact,
the same proof also shows that 
the identity component
of the group of compactly supported contactomorphisms
of $\mathbb{R}^{2n} \times S^1$,
endowed with the $\mathcal{C}^0$-topology,
does not have the Rokhlin property.
Similarly,
we obtain the following corollary
of Theorem \ref{theorem: main},
where $\Ham^c (S^1 \times \R^{2n} \times S^1)$
is endowed with the  $\mathcal{C}^0$-topology
with respect to the distance
induced by the Euclidean distance on $\mathbb{R}^{2n+2}$.

\begin{cor}\label{corollary: Rokhlin}
The topological group  $\Ham^c (S^1 \times \R^{2n} \times S^1)$
does not have the Rokhlin property. 
\end{cor}

\begin{proof}
We  proceed exactly as in \cite{SS}.
Suppose by contradiction that there exists
$g \in \Ham^c (S^1 \times \R^{2n} \times S^1)$
whose conjugacy class is dense
in $\Ham^c (S^1 \times \R^{2n} \times S^1)$.
Since $g$ has compact support,
we can assume that its support is contained in
$S^1 \times B(R) \times S^1$
for some $R$.
Choose $f \in \Ham^c (S^1 \times \R^{2n} \times S^1)$
such that $f(x) \neq x$ for all $x \in S^1 \times B(R+1) \times S^1$.
Such $f$ can be found
by considering the time-$1$ map of a lcs Hamiltonian isotopy
with Hamiltonian function that is equal to $\frac{1}{2}$
on $S^1 \times B(R+1) \times S^1$
and vanishes outside a larger set.
Since we assume that the conjugacy class of $g$
is dense in $\Ham^c (S^1 \times \R^{2n} \times S^1)$,
there exists a sequence $(h_n)$
in $\Ham^c( S^1 \times \R^{2n} \times S^1)$
such that $(h_n \circ g \circ h_n^{-1})$ converges to $f$.
Theorem \ref{theorem: main} implies that
\begin{equation}\label{equation: in RP}
S^1 \times B(R+1) \times S^1 \not\subset \supp(h_n \circ g \circ h_n^{-1}) = h_n(\supp(g))
\end{equation}
for all $n$.
Indeed,
if, on the contrary, $S^1 \times B(R+1) \times S^1 \subset h_n (\supp(g))$
then $S^1 \times B(R+1) \times S^1
\subset h_n (S^1 \times B(R) \times S^1)$,
and so $h_n^{-1}(S^1 \times B(R+1) \times S^1) \subset S^1 \times B(R) \times S^1$, contradicting Theorem \ref{theorem: main}.

Now \ref{equation: in RP} implies the existence of a sequence $(z_n)_{n \in \N}$ in $S^1 \times B(R+1) \times S^1$ such that $$z_n = h_n \circ g \circ h_n^{-1} (z_n).$$ 
Since $S^1 \times B(R+1) \times S^1$ is compact, the sequence $(z_n)$ admits a converging subsequence $(z_{n_k})_{k \in \N}$. Its limit $z \in S^1 \times B(R+1) \times S^1$ is necessarily a fixed point of $f$, contradicting the construction of the latter function. 

\end{proof}





\end{document}